\UseRawInputEncoding 

\documentclass[10pt]{article}
\usepackage[letterpaper, left=1in, top=1in, right=1in, bottom=1in, verbose, ignoremp]{geometry}
\usepackage[
    appendix=inline    
    ]{apxproof}

\include{Preamble}
\usepackage[mathscr]{euscript}

\begin{document}

\def\spacingset#1{\renewcommand{\baselinestretch}%
{#1}\small\normalsize} \spacingset{1}

\begin{flushleft}
{\Large{\textbf{On the Uniqueness of Fr\'echet Means for Polytope Norms}}}
\newline
\\
Roan Talbut$^{1,2\dagger}$, Andrew McCormack$^{3}$, and Anthea Monod$^{2}$
\\
\bigskip
\bf{1} Department of Mathematical Sciences, Durham University, UK
\\

\bf{2} Department of Mathematics, Imperial College London, UK
\\
\bf{3} Department of Mathematical and Statistical Sciences, University of Alberta, Edmonton, Canada
\\
\bigskip
$\dagger$ Corresponding e-mail: roan.s.talbut@durham.ac.uk
\end{flushleft}


\section*{Abstract}

Fr\'echet means are a popular type of average for non-Euclidean datasets, defined as those points that minimise the average squared distance to a set of data points. We consider the behaviour of sample Fr\'echet means on normed spaces whose unit ball is a polytope; this setting is particularly important for semi-discrete data, but is rarely covered by existing literature on Fr\'echet means, which focuses primarily on manifolds and spaces with bounded Alexandrov curvature. We study the geometry of the set of Fr\'echet means over polytope normed spaces, with a focus on dimension and probabilistic conditions for uniqueness. In particular, we provide a geometric characterisation of the threshold sample size at which Fr\'echet means have a positive probability of being unique, and we prove that this threshold is at most one more than the dimension of our space. This geometric characterisation presents a novel framework for computing the unique Fr\'echet mean sample threshold, which we demonstrate in the case of the $\ell_1$, $\ell_{\infty}$ and tropical norms. In particular, we show that the $\ell_1$ sample threshold is always 3, but the $\ell_{\infty}$ and tropical sample thresholds grow at different rates with the dimension of our space.

\paragraph{Keywords: Fr\'echet means, normed space, polytope, cross-polytope, cube, tropical geometry, uniqueness.}


	\section{Introduction}
    Semi-discrete data types are increasingly prevalent as combinatorial representations of complex systems; evolutionary processes are represented as phylogenetic trees \citep{kapli2020phylogenetic}, conditional independence models are represented by graphs \citep{drton2017structure}, and persistence diagrams capture the topological features of a dataset \citep{cohen2005stability}. Across each of these settings, we see \emph{polytope norms} --- norms whose unit ball is a polytope --- playing a significant role. The tropical norm (\cref{ex:tropical-ball_representation}) forms the foundation of several recent statistical methods developed for phylogenetic data \citep{monod2022tropicalgeometryphylogenetictree,aliatimis2023tropical,lin2018FermatWeber}. In estimating graphical models from covariance data, the $\ell_1$ norm is used to induce sparsity \citep{schmidt2007learning,chandrasekaran2010latent,lin2016estimation}. The bottleneck distance between persistence diagrams is a piecewise linear metric based on the $\ell_{\infty}$ norm, which has been foundational in the development of topological data analysis \citep{cohen2005stability,kerber2017geometry,fasy2018approximate}. Motivated by these various semi-discrete applications, we consider data on a normed space whose unit ball is a polytope, and study the averages induced by these polytope norms.
    
    When studying data in a non-Euclidean space, Fr\'echet means are a particularly popular formulation of the ``average'' of a dataset.
    The set of Fr\'echet means of a sample in a metric space is defined as the set of points that minimise the average squared distance to the data points \citep{frechet1948elements}. Fr\'echet means generalise the classical Euclidean notion of a mean to data taking values in general metric spaces. As this definition only depends on the metric of the underlying state space, Fr\'echet means have been used in a wide variety of different geometric settings such as the space of PSD matrices \citep{arsigny2007geometric,fiori2009learning,petersen2019frechet}, the space of persistence diagrams in topological data analysis \citep{mileyko2011probability,cao2022approximating}, and Wasserstein space \citep{panaretos2020Wasserstein}. \cref{def:frechet_mean} gives the definition of Fr\'echet means, for both populations and samples. In most geometric settings, the existence of Fr\'echet means often follows from the extreme value theorem, so long as a distribution has finite second moments.
	\begin{definition}\label{def:frechet_mean}
		Let $X$ be a random variable taking values in a general metric space $(\Omega,d)$. Its population Fr\'echet function $f: \Omega \rightarrow \R$ is defined as
		\[
		f(\theta) = \E \left[ d(X, \theta)^2 \right].
		\]
		For a fixed sample $\{  x_1, \dots,   x_n\} \subset \Omega$, the sample Fr\'echet function is defined as
		\[
		f_n( \theta) = \textstyle\frac{1}{n}\sum_{i = 1}^n d(  x_i, \theta)^2.
		\]
		The population and sample Fr\'echet mean sets are then defined respectively by the set of global minima:
		\[
		\FMset(X) \coloneqq \underset{ \theta \in \Omega}{\argmin} \, f( \theta), \qquad \FMset_n(  x_1, \dots,   x_n) \coloneqq \underset{ \theta \in \Omega}{\argmin} \, f_n( \theta).
		\]
	\end{definition}
    Fr\'echet means are well-studied in a range of settings, primarily those of manifolds \citep{BhatPatFMI,BhatPatFMII} or spaces with bounded curvature \citep{SturmHadamard}. Most notably, \cite[Proposition 4.3]{SturmHadamard} showed that both sample and population Fr\'echet means are unique on spaces with non-positive Alexandrov curvature, such as the cone of SPD matrices \citep{arsigny2007geometric}, and the BHV space of phylogenetic trees \citep{billera2001geometry,nye2017principal}. For Riemannian manifolds, \cite[Theorem 2.1]{BhatPatFMI} give sufficient conditions for uniqueness of population Fr\'echet means, and more recently it was shown that for any absolutely continuous distribution over a complete Riemannian manifold, sample Fr\'echet means will almost surely be unique \citep[Theorem 2.1]{arnaudon2014meansUniqueness}. 
    
    This existing literature on Fr\'echet means fails to include spaces with a polytope norm, where Alexandrov curvature can be unbounded \citep{amendola2024invitation}; this is the question we look to address in this paper --- how do Fr\'echet means behave when induced by a polytope norm? First and foremost, we can no longer guarantee uniqueness in polytope norm spaces \citep{lin2025tropical}. While sets of Fr\'echet means are well-defined and can be shown to form consistent estimators \citep[Theorem 2.2-4]{blanchard2025frechet}, identifying limiting distributions and therefore confidence intervals for Fr\'echet mean sets is particularly complex. As we will show in this paper, it may not be necessary for polytope normed spaces; sample Fr\'echet means induced by polytope norms will be unique asymptotically almost surely. While this suggests that our semi-discrete statistical applications do not require a theory of asymptotic statistics for sets of Fr\'echet means, we instead have to ask how many samples are needed before we can expect a unique Fr\'echet mean? This is particularly important for experimental design, where sample size can be controlled, but at a cost.  It also mirrors the study of maximum likelihood thresholds, which are the number of samples needed for a maximum likelihood estimator to exist almost surely \citep{bernstein2024maximum,drton2019maximum,10.3150/16-BEJ881}. The question of the number of samples required for unique Fr\'echet means is the focus of this paper.
    
    This paper presents foundational results on the uniqueness of Fr\'echet means from polytope norms when samples are drawn from a positive density distribution. In \cref{sec:polytope_geometry,sec:Normed_Space_FMs}, we provide an introduction to the relevant polytope geometry and the geometry of Fr\'echet means on norm spaces. In \cref{sec:large_sample_uniqueness}, we consider the probability of a unique Fr\'echet mean for large samples, showing that this probability will converge to one, but not reach one (\cref{thm:uniqueness_prob_convergence}). In fact, for small sample sizes, the sample Fr\'echet mean is almost surely non-unique (\cref{rmk:non_unique_for_2n}). We therefore ask how many observations are needed for the probability of uniqueness to become positive (\cref{sec:FM_uniqueness})? To define the sample threshold at which the probability of a unique Fr\'echet mean becomes positive, we characterise which configurations of our data can occur with positive probability (\cref{defn:face_type,defn:sample_threshold}); \cite{durier1985geometricalFermateber} take a similar view in studying Fermat--Weber points, though their work is deterministic where ours is probabilistic. In \cref{sec:calculating_examples}, we use our geometric characterisation of unique Fr\'echet means to show that the sample threshold for uniqueness on $\R^k$ is between 3 and $k+1$. This also provides a method to compute the exact sample threshold for certain examples --- for the $\ell_1$, $\ell_{\infty}$, and tropical norms, this threshold is 3, $k$, and $\lfloor (k+3)/2 \rfloor$ respectively for $k \geq 3$.
    We conclude by showing how Wolfe's combinatorial algorithm \citep{wolfe1976finding} can be used for the exact computation of Fr\'echet means in \cref{sec:computation}, and we verify our results with numerical simulation.
	
	\paragraph{Notation.} The relative interior, closure and affine span of a set $U \subseteq \mb{R}^k$ are denoted $\relint(U), \closure(U)$ and $\aff(U)$ respectively. We use $[k]$ to denote the set $\{1,\ldots,k\}$. Throughout this paper $f_n$ and $\FMset_n$ will refer to the sample Fr\'echet function and mean set, while $f$ and $\FMset$ will refer to the population Fr\'echet function and mean set. The boundary of a set $B$ will be denoted $\partial B$. We write a random vector as $\vec X$, whereas a determined data point is written $\vec x$. Throughout, we consider $n$ sample points over $\R^k$ with $n,k \geq 2$.

     \FloatBarrier

    \section{Polytope Geometry} \label{sec:polytope_geometry}
	
	In this section, we review the definitions of polytopes, their faces and their polars. For a more detailed introduction, we refer the reader to \cite{ziegler2012lectures}. We then formalise a polytope norm --- a norm whose unit ball is a polytope --- and review the convex geometry corresponding to such a norm.
	
	\subsection{Polytopes}
	
	Polytopes are a convex hull of finitely many points, or a bounded intersection of finitely many halfspaces. 
	
	\begin{definition}[Polytope]
		A polytope $Q$ in $\R^k$ is the convex hull of a finite set of points $\{ \vec v_1, \dots, \vec v_r \} \subset \mathbb{R}^k$:
		\[
		Q = \conv(\{ \vec v_1, \dots, \vec v_r \}) = \left\{ \textstyle \sum_{i=1}^r \lambda_i \vec v_i: \lambda_1, \dots, \lambda_r \geq 0, 
		\; \sum_{i=1}^r \lambda_i = 1 \right\}.
		\]
		Equivalently, given a matrix  $A \in \R^{r \times k}$ and a vector $\vec b \in \R^r$, a polytope can also be defined as the set of all solutions to the following linear inequalities:
		\[
		Q = P(A, \vec b) \coloneqq \left \{ \vec x \in \R^k \; |
		\; A \vec x \leq \vec b \text{ coordinatewise} \right\},
		\]
		where it is assumed that the set of solutions is a bounded set.
	\end{definition}
	
	\begin{example}[Hypercubes]\label{ex:hypercube_representation}
		We define the $k$-dimensional hypercube $B_{\infty}^k$ by
		\[
		B_{\infty}^k \coloneqq \conv \left ( \{ -1, 1 \}^k\right) = \{ \vec x \in \R^k : \forall i, |x_i| \leq 1 \}.
		\]
        This is the unit ball of the $\ell_{\infty}$ norm.
	\end{example}
	\begin{example}[Cross-polytopes]\label{ex:cross-polytope_representation}
		We define the $k$-dimensional cross-polytope $B_1^k$ by 
		\[
		B_1^k \coloneqq \conv\left( \{ \pm \vec e_1, \dots, \pm \vec e_k \} \right) = \{ \vec x \in \R^k: \textstyle \sum_{i=1}^k |x_i| \leq 1\}.
		\]
        This is the unit ball of the $\ell_1$ norm.
	\end{example}
    	\begin{example}[Tropical ball]\label{ex:tropical-ball_representation} 
		The \emph{tropical projective torus} \citep{maclagan2015introduction} is a foundational space of interest in tropical geometry, and a useful framework for the statistical analysis of phylogenetic trees \citep{speyer2004grassmannian,monod2022tropicalgeometryphylogenetictree}. It is constructed by endowing the real vector space $\R^{k+1}$ with the equivalence relation $\vec x \sim \vec y$ if and only if there is some $a \in \R$ such that $x_i = a+y_i$ for all $i \leq k+1$. We denote the $k$-dimensional tropical projective torus by $\TPT{k+1}$. The tropical norm on the tropical projective torus is given by $\|\vec x\|_{\tr} = \max_i x_i - \min x_i$. The tropical projective torus is naturally homeomorphic to $\R^k$ as can be seen by representing each equivalence class by its unique element which has first coordinate zero. Under this identification, the $k$-dimensional tropical unit ball in $\R^k$ becomes
		\[
		B_{\tr}^{k} \coloneqq \conv \left( \{ 0,1 \}^k \cup \{ 0,-1 \}^k \right) = \{ \vec x \in \R^k: \forall i,j, \,  |x_i|, |x_i-x_j| \leq 1\},
		\]
        while the tropical norm is $\|\vec x \|_{\tr} = \max_{i,j} \{ \pm x_i, x_i-x_j \}$
	\end{example}
	
	\begin{figure*}[ht!]
		\centering
		\begin{subfigure}[b]{0.32\textwidth}
			\centering
			\begin{tikzpicture}[scale=0.95]
				\filldraw[fill=gray,fill opacity=1/4] (1,1) -- (-1,1) -- (-1,-1) -- (1,-1) -- (1,1);
				\draw[->] (-2,0) -- (2,0);
				\draw[->] (0,-2) -- (0,2);
			\end{tikzpicture}
			\caption{2-dimensional hypercube}
		\end{subfigure}
		\begin{subfigure}[b]{0.32\textwidth}
			\centering
			\begin{tikzpicture}[scale=0.95]
				\filldraw[fill=gray,fill opacity=1/4] (1,0) -- (0,1) -- (-1,0) -- (0,-1) -- (1,0);
				\draw[->] (-2,0) -- (2,0);
				\draw[->] (0,-2) -- (0,2);
			\end{tikzpicture}
			\caption{2-dimensional cross-polytope}
		\end{subfigure}
        \begin{subfigure}[b]{0.32\textwidth}
			\centering
			\begin{tikzpicture}[scale=0.95]
				\filldraw[fill=gray,fill opacity=1/4] (1,0) -- (1,1) -- (0,1) -- (-1,0) -- (-1,-1) -- (0,-1) -- (1,0);
				\draw[->] (-2,0) -- (2,0);
				\draw[->] (0,-2) -- (0,2);
			\end{tikzpicture}
			\caption{2-dimensional tropical ball}
		\end{subfigure}
		\begin{subfigure}[b]{0.32\textwidth}
			\centering
			\tdplotsetmaincoords{70}{50}
			\begin{tikzpicture}[scale=0.66, tdplot_main_coords]
				\coordinate (O) at (0,0,0);
				
				\draw[->] (-3,0,0) -- (3,0,0) node[anchor=north east]{$x$};
				\draw[->] (0,-3,0) -- (0,3,0) node[anchor=north west]{$y$};
				\draw[->] (0,0,-2.25) -- (0,0,2.25) node[anchor=south]{$z$};
				
				\coordinate (A) at (1,1,1);
				\coordinate (B) at (1,-1,1);
				\coordinate (C) at (-1,-1,1);
				\coordinate (D) at (-1,1,1);
				\coordinate (E) at (1,1,-1);
				\coordinate (F) at (1,-1,-1);
				\coordinate (G) at (-1,-1,-1);
				\coordinate (H) at (-1,1,-1);
				
				\draw (A) -- (B) -- (C) -- (D) -- (A);
				\draw (E) -- (F) -- (G);
				\draw[dashed] (G) -- (H);
				\draw[dashed] (E) -- (H);
				\draw (A) -- (E);
				\draw (B) -- (F);
				\draw (C) -- (G);
				\draw[dashed] (D) -- (H);
				\fill[fill=gray,opacity=1/8](A) -- (B) -- (F) -- (E) -- cycle;
				\fill[fill=gray,opacity=1/8](C) -- (D) -- (H) -- (G) -- cycle;
				\fill[fill=gray,opacity=1/8](A) -- (D) -- (H) -- (E) -- cycle;
				\fill[fill=gray,opacity=1/8](C) -- (B) -- (F) -- (G) -- cycle;
				\fill[fill=gray,opacity=1/8](A) -- (B) -- (C) -- (D) -- cycle;
				\fill[fill=gray,opacity=1/8](E) -- (F) -- (G) -- (H) -- cycle;
			\end{tikzpicture}
			\caption{3-dimensional hypercube}
		\end{subfigure}
		\begin{subfigure}[b]{0.32\textwidth}
			\centering
			\tdplotsetmaincoords{70}{50}
			\begin{tikzpicture}[tdplot_main_coords]
				\coordinate (O) at (0,0,0);
				
				\draw[->] (-2,0,0) -- (2,0,0) node[anchor=north east]{$x$};
				\draw[->] (0,-2,0) -- (0,2,0) node[anchor=north west]{$y$};
				\draw[->] (0,0,-1.5) -- (0,0,1.5) node[anchor=south]{$z$};
				
				\coordinate (A) at (1,0,0);
				\coordinate (B) at (0,1,0);
				\coordinate (C) at (0,0,1);
				\coordinate (D) at (-1,0,0);
				\coordinate (E) at (0,-1,0);
				\coordinate (F) at (0,0,-1);
				
				\draw (D) -- (E) -- (A) -- (B);
				\draw[dashed] (D) -- (B);
				\draw (D) -- (C) -- (A) -- (F);
				\draw[dashed] (D) -- (F);
				\draw (F) -- (E) -- (C) -- (B);
				\draw[dashed] (F) -- (B);
				\fill[fill=gray,opacity=1/8](A) -- (B) -- (C) -- cycle;
				\fill[fill=gray,opacity=1/8](D) -- (E) -- (F) -- cycle;
				\fill[fill=gray,opacity=1/8](A) -- (B) -- (F) -- cycle;
				\fill[fill=gray,opacity=1/8](C) -- (D) -- (E) -- cycle;
				\fill[fill=gray,opacity=1/8](A) -- (E) -- (C) -- cycle;
				\fill[fill=gray,opacity=1/8](D) -- (B) -- (F) -- cycle;
				\fill[fill=gray,opacity=1/8](A) -- (E) -- (F) -- cycle;
				\fill[fill=gray,opacity=1/8](D) -- (B) -- (C) -- cycle;
			\end{tikzpicture}
			\caption{3-dimensional cross-polytope}
		\end{subfigure}
		\begin{subfigure}[b]{0.32\textwidth}
			\centering
			\tdplotsetmaincoords{70}{50}
			\begin{tikzpicture}[scale=0.66, tdplot_main_coords]
				\coordinate (O) at (0,0,0);
				
				\draw[->] (-3,0,0) -- (3,0,0) node[anchor=north east]{$x$};
				\draw[->] (0,-3,0) -- (0,3,0) node[anchor=north west]{$y$};
				\draw[->] (0,0,-2.25) -- (0,0,2.25) node[anchor=south]{$z$};
				
				\coordinate (A) at (1,1,1);
				\coordinate (B) at (1,0,1);
				\coordinate (C) at (0,0,1);
				\coordinate (D) at (0,1,1);
				\coordinate (E) at (1,1,0);
				\coordinate (F) at (1,0,0);
				\coordinate (G) at (0,1,0);
				\coordinate (H) at (-1,-1,-1);
				\coordinate (I) at (-1,0,-1);
				\coordinate (J) at (0,0,-1);
				\coordinate (K) at (0,-1,-1);
				\coordinate (L) at (-1,-1,0);
				\coordinate (M) at (-1,0,0);
				\coordinate (N) at (0,-1,0);
				
				\draw (A) -- (B) -- (C) -- (D) -- (A);
				\draw (A) -- (E);
				\draw[dashed] (E) -- (G);
				\draw[dashed] (D) -- (G);
				\draw (B) -- (F) -- (E);
				\draw (J) -- (K) -- (H);
				\draw[dashed] (I) -- (J);
				\draw[dashed] (I) -- (H);
				\draw (H) -- (L) -- (N) -- (K);
				\draw[dashed] (M) -- (L);
				\draw[dashed] (I) -- (M);
				\draw (B) -- (N);
				\draw (F) -- (K);
				\draw (E) -- (J);
				\draw (C) -- (L);
				\draw[dashed] (D) -- (M);
				\draw[dashed] (G) -- (I);
				
				\fill[fill=gray,opacity=1/8](A) -- (B) -- (C) -- (D) -- cycle;
				\fill[fill=gray,opacity=1/8](A) -- (B) -- (F) -- (E) -- cycle;
				\fill[fill=gray,opacity=1/8](A) -- (D) -- (G) -- (E) -- cycle;
				\fill[fill=gray,opacity=1/8](B) -- (C) -- (L) -- (N) -- cycle;
				\fill[fill=gray,opacity=1/8](B) -- (F) -- (K) -- (N) -- cycle;
				\fill[fill=gray,opacity=1/8](J) -- (K) -- (F) -- (E) -- cycle;
				
				\fill[fill=gray,opacity=1/8](H) -- (I) -- (J) -- (K) -- cycle;
				\fill[fill=gray,opacity=1/8](H) -- (I) -- (M) -- (L) -- cycle;
				\fill[fill=gray,opacity=1/8](H) -- (K) -- (N) -- (L) -- cycle;
				\fill[fill=gray,opacity=1/8](I) -- (J) -- (E) -- (G) -- cycle;
				\fill[fill=gray,opacity=1/8](I) -- (M) -- (D) -- (G) -- cycle;
				\fill[fill=gray,opacity=1/8](C) -- (D) -- (M) -- (L) -- cycle;
			\end{tikzpicture}
			\caption{3-dimensional tropical ball}
		\end{subfigure}
		
		\caption{The hypercubes, cross-polytopes, and tropical balls in dimension 2 and 3.}
		\label{fig:polytope_pictures}
	\end{figure*}

    Throughout this paper, we use these three classes of polytopes as running examples --- hypercubes, cross-polytopes and tropical balls, which arise as the unit balls of the $\ell_{\infty}$, $\ell_1$, and tropical norms respectively. \Cref{fig:polytope_pictures} shows visualisations of these polytopes for $k=2$ and 3. \\
    
	An important attribute of the geometry of a polytope is the structure of its \emph{faces}.  
	
	\begin{definition}[Face, Face Lattice]
		Let $\vec c \in \R^k$, $c_0 \in \R$ be such that $\vec c ^{\T} \vec x \leq c_0$ for all $\vec x$ in the polytope $Q$. A face $F$ of $Q$ is an intersection of the form:
		\[
		F= Q \bigcap \{ \vec x \in \R^k : \, \vec c^{\T} \vec x = c_0 \}.
		\]
		The set of all faces is denoted $\mathfrak{F}(Q)$, and forms a lattice with the partial ordering given by inclusion.
	\end{definition}
	
	Note that a polytope $Q$ is a face of itself, as is the empty set; any other face of $Q$ is called a \emph{proper} face of $Q$. The $0$-dimensional faces of a polytope are called vertices, and the maximal proper faces of $Q$ are \emph{facets}. When $Q$ is the convex hull of $\{ \vec v_1, \dots, \vec v_r \}$, every face of $Q$ will be a convex hull of a subset of $\{ \vec v_1, \dots, \vec v_r \}$. When $Q$ is given by solutions to a collection of linear inequalities, every face of $Q$ is the solution set of the same inequalities, with some taken to be equalities.
	
	Lastly, we introduce the polar of a polytope which is intimately tied to subgradients of polytope norms.
	
	\begin{definition}[Polar Polytope]
		Let $Q = \conv \{ \vec v_1, \dots, \vec v_r \} \subset \R^k$ be a full-dimensional polytope containing the origin in its interior. The polar of $Q$ is the polytope given by:
		\[
		Q^{\Delta} \coloneqq \{ \vec x \in \R^k : \forall \, \vec y \in Q , \, \vec y^{\T} \vec x \leq 1 \} = P(V,\mathbf 1),
		\]
		where $V$ is the matrix with rows $\vec v_1, \dots, \vec v_r$. The faces of $Q^{\Delta}$ are given by $\{ F^{\Diamond}: \, F \in \mathfrak{F}(Q) \}$, where
		\[
		F^{\Diamond} \coloneqq \{ \vec x \in Q^{\Delta} : \forall \, \vec y \in F, \; \vec y^{\T} \vec x = 1 \}.
		\]
		We call $F^{\Diamond}$ the polar of the face $F$.
	\end{definition}
	
	The faces of a polytope are in bijection with the faces of its polar polytope, in fact their face lattices are anti-isomorphic --- $F_1^{\Diamond} \subseteq F_2^{\Diamond}$ if and only if $F_2 \subseteq F_1$. The polar of a polar polytope $Q^{\Delta \Delta}$ is equal to $Q$, and as such, we see that the polar of $P(A, \vec 1)$ is given by the convex hull of the row vectors of $A$. 
	
	\begin{example}[Hypercubes and Cross-polytopes]
		The $k$-dimensional hypercube and cross-polytope are polars of each other. The polars of the hypercube vertices are the facets of the cross-polytope, and vice versa.
	\end{example}
	
	\subsection{Polytope Norms}
	
	We now review the convex geometry of a norm whose unit ball is a polytope $B$, for which we use the half-space representation of a polytope.
	
	\begin{definition}[Polytope Normed Space]
		Let $B = P(A, \vec 1) \subset \R^k$ be a full-dimensional, centrally symmetric polytope (in that $B = -B$). We assume $A$ is reduced, such that every $\vec a \in A$ is the normal vector for some facet of $B$. Then $(\R^k, \|\cdot\|_B)$ is a polytope normed space where the $B$-norm is defined by
		\begin{align}
			\label{eqn:BnormDefinition}
			\|\vec x\|_B \coloneqq \inf \{ d \in \R_{>0}: \,  d^{-1}\vec x \in B \} \equiv \max_{i \leq r} (A \vec x)_i.
		\end{align}
		The metric induced by the $B$-norm is denoted by $d_B$. 
	\end{definition}
	
	It is straightforward to show that $\|\vec x\|_B$ is a norm, and that the unit ball of $\|\cdot\|_B$ is the polytope $B$. As discussed in \cref{ex:hypercube_representation,ex:cross-polytope_representation,ex:tropical-ball_representation}, the $\ell_{\infty}, \ell_1$ and tropical norms have unit balls given by the hypercube, cross-polytope, and tropical ball respectively.
	
	To discuss the local behaviour of a polytope norm near $\vec x$, we need to know which rows of the coefficient matrix $A$ are attaining the componentwise maximum of $A\vec x$. We refer to these rows as the \emph{active constraints}. We will write $\vec a \in A$ to mean $\vec a$ is a row vector of $A$.
	
	\begin{definition}
		Let $F$ be a face of $B = P(A,\vec 1)$. The active constraints on $F$, denoted $\A(F)$, is the set:
		\[
		\A(F) = \{ \vec a \in A : \forall \, \vec x \in F, \,  \vec a^{\T} \vec x=1 \}. 
		\]
		Similarly, for any $\vec x \in \R^k$ we define $\A(\vec x)$ to be the active constraints of $\| \vec x \|_B$. That is:
		\[
		\A( \vec x ) = \{ \vec a \in A :  \vec a^{\T} \vec x = \| \vec x \|_B \}.
		\]
	\end{definition}
	Notice that $\conv \A(F) = F^{\Diamond}$. We also note that the set of open cones generated by proper faces of $B$, $\relint (\R_{>0} F)$, partition $\R^k \setminus \{ \vec 0 \}$ according to the active constraints; $\vec x \in \relint (\R_{>0} F)$ if and only if $\A(\vec x) = \A(F)$.  This partition of the space is the \emph{face fan} of $B$.
	
	Finally, we review the subgradient and directional derivatives of the $B$-norm at $\vec x$. Recall the subgradient of a convex function $g:\R^k \rightarrow \R$ at $\vec x$ is given by:
    \[
    \partial g(\vec x) \coloneqq \{ \vec v : \, g(\vec z) \geq g(\vec x) + \vec v^{\T}(\vec z-\vec x) \; \text{ for all } \vec z \in \R^k \}.
    \]
    The directional derivative in direction $\vec t \in \R^k$, denoted by $\partial_{\vec t} g(\vec x)$, is 
    \[
    \partial_{\vec t} g(\vec x) \coloneqq \lim_{h \rightarrow 0^+} h^{-1} \left( g(\vec x + h \vec t) - g(\vec x) \right) = \sup_{\vec v \in \partial g(\vec x)} \vec v^{\T} \vec t.
    \]
    This equality is from Theorem 23.4 of \citep{rockafellar1970convex}.	
	
	\begin{remark}
		\label{rmk:subgradient}
		The subgradient of a maximum of differentiable functions $g(\vec x) = \max \{ g_1(\vec x), \dots, g_n(\vec x)\}$ is given by the convex hull of gradients of active functions at $\vec x$; that is, $\partial g(\vec x) = \conv \{ \nabla_{\vec x} g_i(\vec x):g_i(\vec x) = g(\vec x)\}$. Applying this to a polytope $B$-norm, which is a maximum of linear functions by \eqref{eqn:BnormDefinition}, we obtain that
		\begin{align*}
			\partial \Vert \vec x \Vert_B^2 =  \conv \left \{ \nabla_{\vec x} (\vec a^{\top}\vec x)^2: \vec a \in \A (F) \right\} = 2 \|\vec x\|_B \conv \A(F) = 2\| \vec x \|_B F^{\Diamond}, 
		\end{align*}
		where $F$ is such that $\vec x \in \relint (\R_{>0} F)$. Consequently, the subgradient of the sample Fr\'echet function will be given by
		$\partial f( \bm \theta) =   -2n^{-1}\sum_{i \leq n} \| \vec x_i -  \bm \theta \|_B F_i^{\Diamond}$, where $F_i$ is the unique face with $\vec x_i -  \bm \theta \in \relint (\R_{>0} F_i)$, and the summation is a Minkowski sum. Since a convex function is minimized if and only if its subgradient contains $\vec 0$, it follows that $ \bm \theta$ is a sample Fr\'echet mean if and only if $\vec 0 \in  \sum_{i \leq n} \| \vec x_i -  \bm \theta \|_B F_i^{\Diamond}$.
	\end{remark}

     \FloatBarrier

    \section{Fr\'echet Means on Normed Spaces} \label{sec:Normed_Space_FMs}
	
	In this section, we discuss Fr\'echet means and outline their fundamental properties in the case of general finite-dimensional normed spaces. All Fr\'echet means in this section are taken with respect to a given $B$-norm, where $B$ is a convex symmetric set. Now applying \cref{def:frechet_mean} to $(\R^k, \|\cdot \|_B)$, the Fr\'echet function is given by $f( \bm \theta) = \E [\|\vec X - \bm \theta\|_B^2]$ for a random variable $\vec X$, and $f_n( \bm \theta)=n^{-1} \sum_{i=1}^n \|\vec x_i - \bm \theta\|_B^2$ for a fixed sample $\{ \vec x_1, \dots, \vec x_n \}$. We then denote the population Fr\'echet mean set by $\FMset(\vec X) \coloneqq \argmin_{ \bm \theta} f( \bm \theta)$ and the sample Fr\'echet mean set by $\FMset_n(\vec x_1, \dots, \vec x_n) \coloneqq \argmin_{ \bm \theta} f_n( \bm \theta)$. We note that sample Fr\'echet means on $(\R^k, \|\cdot \|_B)$ are always well-defined and exist by the extreme value theorem; the existence of population Fr\'echet means is dependent on the existence of finite moments of a distribution, which we assume throughout this paper.
	
	Understanding the convexity of the Fr\'echet function is crucial for understanding the geometry of Fr\'echet mean sets; these foundational convexity results have been established in many settings, such as Wasserstein spaces \citep{panaretos2020}, Riemannian spaces \citep{afsari2011riemannian}, metric trees \citep{romon2023convex}, and the tropical projective torus \citep{lin2025tropical}. Below, we state the convexity results we require in the case of a general normed space. These results below are generalisations of the convexity results established in the tropical setting by \cite{lin2025tropical}, and mirror the work of \cite{durier1985geometricalFermateber} on Fermat--Weber points.
    
	\begin{lemma}
		\label{lemma:NormConvexity}
		Fix a normed space $(\R^k, \|\cdot\|_B)$ with induced metric $d_B$. For all $\vec x$, $d_B(\cdot, \vec x)^2$ is convex, and in fact if, for some $t \in (0,1)$ we have 
        \begin{equation*}
        d_B((1-t) \bm \theta_1 + t  \bm \theta_2, \vec x)^2  = (1-t)d_B( \bm \theta_1, \vec x)^2+ t d_B( \bm \theta_2, \vec x)^2,
        \end{equation*}
        then $d_B( \bm \theta_1, \vec x) = d_B( \bm \theta_2, \vec x)$. The population/sample Fr\'echet function is also convex, and so population/sample Fr\'echet mean sets are convex.
	\end{lemma}
	\begin{proof}
		Convexity of a squared norm is immediate from the convexity of norms, as the function $y \mapsto y^2$ is strictly monotone increasing and strictly convex over the domain $\R_{\geq 0}$:
		\begin{align*}
			\Vert (1-t) \bm \theta_1 + t  \bm \theta_2 - \vec x\Vert_B^2 &\leq \left( (1-t)\Vert  \bm \theta_1 - \vec x \Vert_B + t \Vert  \bm \theta_2 - \vec x\Vert_B \right)^2, \\
			&\leq (1-t)\Vert \bm \theta_1 - \vec x \Vert_B^2 + t \Vert \bm \theta_2 - \vec x\Vert_B^2.
		\end{align*}
		Hence $\Vert  \bm \theta - \vec x\Vert_B^2 = d( \bm \theta,\vec x)_B^2$ is convex in $ \bm \theta$. The second statement follows from the strict convexity of $y \mapsto y^2$; we have equality if and only if $\Vert  \bm \theta_1 - \vec x \Vert_B^2 = \Vert  \bm \theta_2 - \vec x\Vert_B^2$. \\
        Convexity of the population/sample Fr\'echet function $f$ is a direct result of expectation preserving convexity, and the convexity of population/sample Fr\'echet mean sets follows from the convexity of $f$.
	\end{proof}
	
	Building on \Cref{lemma:NormConvexity}, we will see that Fr\'echet mean sets are given by intersections of spheres --- boundaries of balls. \Cref{lemma:FM_dists} expresses this in the case of a population Fr\'echet mean set for a general population $\vec X \sim \mb P$, while \Cref{cor:FMisPolytope} considers the case of sample Fr\'echet means.
	
	\begin{lemma}
		\label{lemma:FM_dists}
		Let $\vec X \sim \mathbb P$ be a random variable in $(\R^k, \| \cdot \|_B)$. Assume that $ \bm \theta_1$ and $ \bm \theta_2$ are Fr\'echet means of $\vec X$. Then for every $\vec x$ in the  $\esssupp(\mathbb P)$, we have $d_B(\vec x,  \bm \theta_1) = d_B(\vec x,  \bm \theta_2)$. 
	\end{lemma}
	\begin{proof} 
		Fix some $\vec x \in \esssupp(\mb P)$. By definition, for all $m \in \N$ we have $\mb P(B_{1/m}(\vec x)) > 0$, where $B_{\epsilon}(\vec x)$ is the ball of radius $\epsilon$ about $\vec x$. We will show that there is some $\vec x_m \in B_{1/m}(\vec x)$ such that $d_B(\vec x_m, \bm \theta_1) = d_B(\vec x_m,  \bm \theta_2)$ by contradiction. Set $ \bm \theta_3 = \tfrac{1}{2}( \bm \theta_1 +  \bm \theta_2)$, a midpoint of $ \bm \theta_1,  \bm \theta_2$. By the convexity of $d_B(\vec u,\cdot)$ and strict convexity of $y \mapsto y^2$:
		\begin{align*}
			d_B(\vec u, \bm \theta_3)^2 \leq \big(\tfrac{1}{2}d_B(\vec u, \bm \theta_1) + \tfrac{1}{2}d_B(\vec u, \bm \theta_2)\big)^2 \leq \tfrac{1}{2}\big( d_B(\vec u, \bm \theta_1)^2 + d_B(\vec u, \bm \theta_2)^2\big).
		\end{align*}
		We assume for all $\vec u \in B_{1/m}(\vec x)$, $d_B(\vec u, \bm \theta_1) \neq d_B(\vec u,  \bm \theta_2)$, so the above inequality is strict. It follows that
		\begin{align*}
			\E \left[I_{\vec X \in B_{1/m}}d_B(\vec X, \bm \theta_3)^2\right] &< \tfrac{1}{2}\left(\E\left[I_{\vec X \in B_{1/m}}d_B(\vec X, \bm \theta_1)^2\right] + \E \left[I_{\vec X \in B_{1/m}}d_B(\vec X, \bm \theta_2)^2\right]\right).
		\end{align*}
		Hence
		\[
		\E \left[d_B(\vec X, \bm \theta_3)^2\right] < \tfrac{1}{2}\left(\E\left[d_B(\vec X, \bm \theta_1)^2\right] + \E \left[d_B(\vec X, \bm \theta_2)^2\right]\right),
		\]
		which contradicts the optimality of $ \bm \theta_1$ and $ \bm \theta_2$. Therefore for all $m \in \N$, we can pick some $\vec x_m \in B_{1/m}(\vec x)$ such that $d_B(\vec x_m, \bm \theta_1) = d_B(\vec x_m, \bm \theta_2)$. By continuity of the distance function we conclude $d_B(\vec x, \bm \theta_1) = d_B(\vec x, \bm \theta_2)$.
	\end{proof}
	
	As an immediate consequence of \Cref{lemma:FM_dists}, we see that any distribution whose essential support is $\R^k$ must have a unique Fr\'echet mean. 

    \begin{theorem}\label{thm:unique_population_FM}
		Suppose $\vec X \sim \mb P$, with $\esssupp (\mb P) = \R^k$. Then $\FMset (\vec X)$ is unique.
	\end{theorem}
	
	\begin{proof}
		This follows directly from \Cref{lemma:FM_dists}, as for any two Fr\'echet means $ \bm \theta_1, \bm \theta_2$ we have $0 = d( \bm \theta_1, \bm \theta_1) = d( \bm \theta_1, \bm \theta_2)$, because $ \bm \theta_1 \in \text{esssupp}(\mb P)$. 
	\end{proof}
    
    However, for finite samples, uniqueness is not guaranteed. Fr\'echet means are given by intersections of spheres centred at each data point, as shown in \Cref{cor:FMisPolytope}. In fact, in the case of polytope norms, Fr\'echet means are intersections of faces. \Cref{fig:FM_sets} provides a visualisation of this behaviour for Fr\'echet mean sets of a small number of samples. 
	
	\begin{corollary}
		\label{cor:FMisPolytope}
		Fix a normed space $(\R^k, \| \cdot \|_B)$ with unit ball $B$. There exist unique $d_i \in \R_{\geq 0}$ such that $$\FMset_n (\vec x_1,\ldots,\vec x_n) = \textstyle\bigcap_{i \leq n}\{d_iB + \vec x_i\} = \bigcap_{i \leq n}\{d_i\partial B + \vec x_i\},
		$$ where $\partial B$ is the boundary of $B$.\\ 
        In the case where $B$ is a polytope, $\FMset_n(\vec x_1, \dots, \vec x_n)$ is also polytope. Moreover, for every $i \leq n$ there exists a proper face $F_i$ of $B$ such that $\FMset_n$ is given by $\bigcap_{i \leq n}\{d_iF_i + \vec x_i\}$. 
	\end{corollary}
	
	\begin{proof}
		By \Cref{lemma:FM_dists} applied to the empirical measure $n^{-1}\sum_{i \leq n} \delta_{\vec x_i}$, the distances $d_i = d_B(\vec x_i, \bm \theta)$ are fixed for all $ \bm \theta \in \FMset_n(\vec x_1, \dots, \vec x_n)$. Hence:
		\[
		\FMset_n(\vec x_1, \dots, \vec x_n) \subseteq \textstyle \bigcap_{i \leq n} \{d_i\partial B + \vec x_i\} \subseteq \bigcap_{i \leq n} \{d_i B + \vec x_i\} \subseteq \FMset_n(\vec x_1, \dots, \vec x_n).
		\]
		This follows from \Cref{lemma:FM_dists}, nested intersections, and the definition of the $d_i$. When $B$ is a polytope, $\FMset_n$ is therefore an intersection of polytopes and so a polytope itself.
		
		We prove the last claim by contradiction. Fix $i \leq n$. Suppose there is no constraint $\vec a \in A$ which is active for all $\vec u \in \vec x_i - \FMset_n$. Then for all $\vec a_j \in A$, we have some $ \bm \theta_j \in \FMset_n$ such that $\vec a_j \in A$ is not in $\A(\vec x_i -  \bm \theta_j)$. By the convexity of $\FMset_n$, $\tfrac{1}{r}\sum_{j \leq r}  \bm \theta_j$ is a Fr\'echet mean and must be in the set $d_i\partial B + \vec x_i$. For all $\vec a_j$, we have
		\[
            \vec a_j^\T(\vec x_i -  \bm \theta_j) < \Vert \vec x_i -  \bm \theta_j\Vert_B = d_i, \qquad \forall l \neq j: \;\vec a_j^\T(\vec x_i -  \bm \theta_l) \leq \Vert \vec x_i -  \bm \theta_l\Vert_B = d_i.
		\]
		We obtain that
		\begin{align*}
			\left \Vert \vec x_i - \tfrac{1}{r}\textstyle\sum_{j \leq r}  \bm \theta_j \right \Vert_B = \textstyle\max_{j'} \vec a_{j'}^{\top} \left(\tfrac{1}{r}\sum_{j \leq r} (\vec x_i -   \bm \theta_j) \right) < \tfrac{1}{r}\sum_{j \leq r} \|\vec x_i -   \bm \theta_j\|_B = d_i,
		\end{align*}
		which contradicts $\tfrac{1}{r}\sum_{j \leq r}  \bm \theta_j \in \FMset_n$. So for all $\vec x_i$, there is some proper face $F_i$ of $B$ with $\FMset_n \subseteq d_i F_i + \vec x_i$. Hence:
		\[
		\textstyle \FMset_n(\vec x_1, \dots, \vec x_n) \subseteq \bigcap_{i \leq n} \{d_iF_i + \vec x_i\} \subseteq \bigcap_{i \leq n} \{d_i\partial B + \vec x_i\} = \FMset_n(\vec x_1, \dots, \vec x_n).
		\]
		So the result is proved.
	\end{proof}

    We have proved a more general result in the final part of this corollary, which will be useful later:
	\begin{corollary}\label{cor:convex_disjoint_from_relint}
		Suppose $S$ is a convex set contained in a polytope $Q$. If $S$ and $\relint Q$ are disjoint, then $S$ is contained in some proper face of $Q$.
	\end{corollary}

	\cref{cor:FMisPolytope} highlights why Fr\'echet means with respect to polytope balls are particularly interesting to study, as such balls often have intersections that are positive dimensional, as is illustrated in the left hand side of \cref{fig:FM_sets}.

	\begin{figure}[!ht]
		\centering
		\begin{subfigure}{0.49\textwidth}
			\centering
			\begin{tikzpicture}[scale=1.6]
				\draw (1,1) -- (-1, 1) -- (-1,-1) -- (1,-1) -- (1,1);
				\draw (1,0) -- (1,2) -- (3,2) -- (3,0) -- (1,0);
				\filldraw (0,0) circle (1pt) node[anchor=north] {\small$\vec x_1 = (0,0)$};
				\filldraw (2,1) circle (1pt)  node[anchor=south] {\small$\vec x_2 = (2,1)$};
				\draw[red, line width = 2pt]  (1,0) -- (1,1);
			\end{tikzpicture}
		\end{subfigure}
		\hfill
		\begin{subfigure}{0.49\textwidth}
			\centering
			\begin{tikzpicture}[scale=0.8]
				\draw (0,0) -- (-4, 0) -- (-4,-4) -- (0,-4) -- (0,0);
				\draw (-0.5,0) -- (1.5, 0) -- (1.5,2) -- (-0.5,2) -- (-0.5,0);
				\draw (0,0.5) -- (0,-1.5) -- (2,-1.5) -- (2,0.5) -- (0,0.5);
				\filldraw (-2,-2) circle (1pt) node[anchor=north east] {\small$\vec x_1 = (-4,-4)$};
				\filldraw (0.5,1) circle (1pt)  node[anchor=south] {\small$\vec x_2 = (1,2)$};
				\filldraw (1,-0.5) circle (1pt)  node[anchor=north west] {\small$\vec x_3 = (2,-1)$};
				\filldraw[red] (0,0) circle (3pt);
			\end{tikzpicture}  
		\end{subfigure}
		\caption{Fr\'echet mean sets (red) are intersections of boundaries of balls centred about each datapoint. In the case of Fr\'echet mean sets induced by polytope norms, the Fr\'echet mean sets are an intersection of faces of balls that are centred about each data point (\Cref{cor:FMisPolytope}).}
		\label{fig:FM_sets}
	\end{figure}
	
	By contrast, when $B$ is strictly convex, the Fr\'echet mean of two points $\vec x_1$ and $\vec x_2$ will always be the single contact point of the two balls of radius $\frac{1}{2}\|\vec x_1 - \vec x_2\|_B$ around $\vec x_1$ and $\vec x_2$. In fact, if the unit ball of a norm is strictly convex, then sample Fr\'echet means must be unique, as we show in the following corollary. This is in contrast with the polytope norm setting.

    	\begin{corollary}
        \label{cor:StrictConvexBallUnique}
		Fix a normed space $(\R^k, \| \cdot \|_B)$ with unit ball $B$. If $B$ is strictly convex, in the sense that the relative interior of a segment between two points in $\partial B$ is contained in $\mathrm{int}(B)$, then $\FMset_n (\vec x_1,\ldots,\vec x_n)$ is a single point. In particular, sample Fr\'echet means with respect to $\ell_p$ norms are always unique when $p \in (1,\infty)$.
	\end{corollary}
	\begin{proof}
		The norm of a strictly convex ball has the property that whenever $\vec x$ is not proportional to $\vec y$ with respect to a non-negative scaling, we have for all $t \in (0,1)$ that 
		\begin{align*}
			\Vert (1-t)\vec x + t \vec y\Vert_B &= \big((1-t)\Vert \vec x \Vert_B + t\Vert \vec y \Vert_B\big)     \bigg\Vert \frac{(1-t)\Vert \vec x \Vert_B}{(1-t)\Vert \vec x \Vert_B + t\Vert \vec y \Vert_B} \frac{\vec x}{\Vert \vec x \Vert_B} + \frac{t\Vert \vec y \Vert_B}{(1-t)\Vert \vec x \Vert_B + t\Vert \vec y \Vert_B} \frac{\vec y}{\Vert \vec y \Vert_B} \bigg\Vert_B 
			\\
			&< (1-t)\Vert \vec x\Vert_B + t \Vert \vec y\Vert_B,
		\end{align*}
		since the above expression is a convex combination of the points $\vec x/\Vert \vec x \Vert_B, \vec y/\Vert \vec y \Vert_B \in \partial B$. Consider two Fr\'echet means, $ \bm \theta_1, \bm \theta_2 \in \FMset_n(\vec x_1, \dots, \vec x_n)$. By the convexity of $\FMset_n$, and \cref{lemma:FM_dists}, $(1-t) \bm \theta_1 + t \bm \theta_2$ has the same distance to $\vec x_1$ as $ \bm \theta_1$ and $ \bm \theta_2$. So by the result above we must have that $\vec x_1 -  \bm \theta_1 = c(\vec x_1 -  \bm \theta_2)$ for some $c \geq 0$, as well as $d(\vec x_1,  \bm \theta_1) = d(\vec x_1,  \bm \theta_2)$. The only way this can occur is if $c = 1$ and $ \bm \theta_1 =  \bm \theta_2$. 
		
		The $\ell_p$ balls are strictly convex for $p \in (1,\infty)$ since if $\Vert \vec x \Vert_p = \Vert \vec y \Vert_p = 1$ with $\vec x \neq \vec y$, then  
		\begin{align*}
			\Vert (1-t)\vec x + t \vec y \Vert_p^p &=   \textstyle\sum_{i \leq k} \vert (1-t)x_i + ty_i\vert^p \\
            & \leq \textstyle\sum_{i \leq k}  ((1-t)\vert x_i \vert + t\vert y_i \vert)^p
			\\
			& \leq   \textstyle\sum_{i \leq k}  \big((1-t)\vert x_i \vert^p  + t\vert y_i \vert^p\big) = 1.
		\end{align*}
		The second inequality follows from Jensen's inequality and is strict if there is some $i$ such that $|x_i| \neq |y_i|$. If, for all $i$, $|x_i| = |y_i|$, then the first inequality is strict as $\vec x \neq \vec y$.
	\end{proof}

\FloatBarrier 

\section{Large Sample Uniqueness for Polytope Fr\'echet Means} \label{sec:large_sample_uniqueness}

    In \cref{sec:polytope_geometry,sec:Normed_Space_FMs}, we reviewed the necessary background on polytope geometry and Fr\'echet means in normed spaces. We can now begin to study the probability of getting a unique sample Fr\'echet mean from a polytope norm; in this section, we focus on the asymptotic behaviour of this uniqueness probability, which is particularly useful for estimating confidence intervals for sample Fr\'echet means. In this section and throughout the rest of the paper, we assume the underlying distribution of our data has a positive density with respect to the Lebesgue measure on $\R^k$, and that the population Fr\'echet mean exists and is unique. 
	
	As shown by \cite{BhatPatFMI}, sample Fr\'echet means are strongly consistent estimators for population Fr\'echet means.
	
	\begin{theorem}\label{thm:consistent_FM}
		Let $\{\vec X_i\}_{i \in \N}$ be a sequence of i.i.d. random vectors with $\vec X_i \sim \mb{P}$, where $\mb P$ has a unique population Fr\'echet mean $ \bm \theta_*$. Let  $ \bm \theta_n$ be a random vector such that $ \bm \theta_n \in \FMset_n(\vec X_1, \dots, \vec X_n)$ almost surely. Then $ \bm \theta_n$ converges almost surely to $ \bm \theta_*$.
	\end{theorem}
	
	\begin{proof}
        We note that closed and bounded subsets of $\R^k$ are compact, and we have assumed the population Fr\'echet function is finite. We have also assumed that $\mb P$ has a unique Fr\'echet mean $ \bm \theta_*$. Hence, by Theorem 2.3 of \cite{BhatPatFMI}, $ \bm \theta_n$ is a strongly consistent estimator for $ \bm \theta_*$.
	\end{proof}
		
	The consistency of sample Fr\'echet means does not directly guarantee their uniqueness. In \cref{thm:uniqueness_prob_convergence}, we show that for any polytope norm, the probability of a unique sample Fr\'echet mean tends to one as the sample size increases, but will never reach one.
	
    \begin{theorem}\label{thm:uniqueness_prob_convergence}
		Suppose $\vec X \sim \mb P$, with $\esssupp(\mb P) = \R^k$. Let $p_n$ be the probability that $\FMset_n(\vec X_1,\dots,\vec X_n)$ is a unique point. Then $p_n$ converges to 1 as $n \rightarrow \infty$, but for all $n$, $p_n < 1$.
	\end{theorem}

    The constructions we use in proving \cref{thm:uniqueness_prob_convergence} are shown in \cref{fig:p_n_configurations}.

    \begin{figure}[!ht]
		\centering
		\begin{subfigure}{0.49\textwidth}
			\centering
			\begin{tikzpicture}
				\draw (0,-1) -- (1, 0) -- (0,1) -- (-1,0) -- (0,-1);
				\draw (-4,-1) -- (4,-1);
				\draw (-4,1) -- (4,1);
				\draw (-1,-2.5) -- (-1,2.5);
				\draw (1,-2.5) -- (1,2.5);
				\draw[<->] (-1,0) -- (0,0) node[anchor=north east] {$1$};
				\filldraw (0,0) circle (2pt) node[anchor=north west] {\small$ \bm \theta_*$};
				\filldraw (-0.3,0.4) circle (2pt) node[anchor=south east] {\small$ \bm \theta_n$};
				\node at (3.5,2) {$S_2$};
				\node at (-3.5,2) {$S_1$};
				\node at (3.5,-2) {$S_3$};
				\node at (-3.5,-2) {$S_4$};
				\filldraw (2,2) circle (2pt) node[anchor=south west] {\small$\vec X_{j(2)}$};
				\filldraw (-1.8,-2) circle (2pt) node[anchor=north east] {\small$\vec X_{j(4)}$};
				\filldraw (2,-2) circle (2pt) node[anchor=north west] {\small$\vec X_{j(3)}$};
				\filldraw (-2.4,2.3) circle (2pt) node[anchor=south east] {\small$\vec X_{j(1)}$};
			\end{tikzpicture}
            \caption{A visualisation of the cones $S_i$ (defined in \cref{thm:uniqueness_prob_convergence} proof) for the $\ell_1$ norm. Each cone contains a data point, so if there is a Fr\'echet mean in the unit ball around $ \bm \theta_*$, it must be a unique Fr\'echet mean.}
		\end{subfigure}
		\hfill
		\begin{subfigure}{0.49\textwidth}
		\centering
		\begin{tikzpicture}[scale=1.5]
			
			\coordinate (A) at (0,0);
			\coordinate (B) at (-0.75,0.5);
			\coordinate (C) at (1.75,-0.125);
			\coordinate (D) at (1,0.5);            
			
			\draw ($(A) + 0.5*(1,1)$) -- ($(A) + 0.5*(-1,1)$) -- ($(A) + 0.5*(-1,-1)$) -- ($(A) + 0.5*(1,-1)$) -- cycle;
			\draw ($(B) + 1.25*(1,1)$) -- ($(B) + 1.25*(-1,1)$) -- ($(B) + 1.25*(-1,-1)$) -- ($(B) + 1.25*(1,-1)$) -- cycle;
			\draw ($(C) + 1.25*(1,1)$) -- ($(C) + 1.25*(-1,1)$) -- ($(C) + 1.25*(-1,-1)$) -- ($(C) + 1.25*(1,-1)$) -- cycle;
			\draw ($(D) + 0.5*(1,1)$) -- ($(D) + 0.5*(-1,1)$) -- ($(D) + 0.5*(-1,-1)$) -- ($(D) + 0.5*(1,-1)$) -- cycle;
			\filldraw (A) circle (1pt) node[anchor=north] {\small $\vec X_1$};
			\filldraw (B) circle (1pt) node[anchor=south ] {\small $\vec X_3$};
			\filldraw (C) circle (1pt) node[anchor= north] {\small $\vec X_4$};
			\filldraw (D) circle (1pt)  node[anchor=south] {\small $\vec X_2$};
			\draw[red, line width = 3pt]  (0.5,0) -- (0.5,0.5);
		\end{tikzpicture}
        \caption{The difference between any two datapoints is contained in the same face cone for the $\ell_\infty$ norm. This occurs with positive probability and results in a non-unique Fr\'echet mean (red) as shown in the \cref{thm:uniqueness_prob_convergence} proof.}
		\end{subfigure}
        \caption{Illustrations of the constructions used in the proof of \cref{thm:uniqueness_prob_convergence}.}\label{fig:p_n_configurations}
	\end{figure}
    
	\begin{proof}
		Suppose there is some sample Fr\'echet mean $ \bm \theta_n \in \FMset_n(\sample)$ such that for every possible active constraint $\vec a_i \in A$, there is some data point $\vec X_{j(i)}$ satisfying $\A(\bm \theta - \vec X_{j(i)}) = \vec a_i$. Then by \Cref{cor:FMisPolytope}, the Fr\'echet mean set is contained in the intersection of corresponding faces, $\bigcap_i \| \bm \theta-\vec X_{j(i)}\| F_i+\vec X_{j(i)}$ and so
		\begin{align}
			\FMset_n(\vec X_1,\dots,\vec X_n) \subseteq \{ \vec x: \vec a_i^{\T}(\vec x-\vec X_{j(i)}) = \| \bm \theta-\vec X_{j(i)}\| \} \label{eq:FM_uniqueness_linear_constraints}.
		\end{align}
		As the set of all $\vec a_i$ spans $\R^k$, the linear constraints in \eqref{eq:FM_uniqueness_linear_constraints} define a unique point in $\FMset_n(\vec X_1, \dots, \vec X_n)$. It now suffices to show that, with probability tending to 1, there is some Fr\'echet mean $ \bm \theta$ such that every constraint $\vec a \in A$ is active for some $\bm \theta - \vec X_{j}$.

        Let $ \bm \theta_*$ be the unique population Fr\'echet mean of $\vec X$, which exists by \cref{thm:unique_population_FM}.
		For each facet $F_i$, consider the intersection $S_i = \bigcap_{ \bm \theta \in B_1( \bm \theta_*)} ( \bm \theta - \relint (\R_{>0} F_i)) \subset \R^k$, where $B_1(\bm \theta_*)$ is the unit ball about $\bm\theta_*$. This is the set of points $\vec x$ with $\A(\bm \theta - \vec x) = \vec a_i$, the normal vector of $F_i$, for all $ \bm \theta \in B_1(\bm \theta_*)$. This intersection is full dimensional, so $\mb P (\vec X_j \in S_i) = c_i > 0$. Let $b_n$ be the probability that there is some $ \bm \theta_n \in \FMset_n (\vec X_1, \dots, \vec X_n)$ in $B_1( \bm \theta_*)$. By the argument above, we can bound $p_n$ below using $b_n, c_i$:
		\begin{align*}
			p_n &\geq \textstyle \mb P\left(A = \bigcup_{j \leq n} \A(\bm \theta - \vec X_j)\right) \\
			&\geq \mb P\left(\{\emptyset \neq \FMset_n(\vec X_1,\dots,\vec X_n) \cap B_1( \bm \theta_*)\} \cap \{ \forall S_i, \exists \, j \text{ s.t. } \vec X_j \in S_i \}\right) \\
			&\geq 1 - (1-b_n) - \textstyle \sum_{i \leq r} (1-c_i)^n
		\end{align*}
		By the consistency of $ \bm \theta_n$ (\Cref{thm:consistent_FM}), $b_n \rightarrow 1$. So this lower bound converges to 1.

        It remains to show that $p_n < 1$. Fix some $\vec a \in A$ and corresponding facet $F$. We first show that if,
        \begin{gather}
            \forall j_1 \neq j_2 \leq n : \, \A(\vec X_{j_1}-\vec X_{j_2}) = \vec a \text{ or } -\vec a, \label{eq:data_cones_nested} \\
            \textstyle \text{ and }\;\nexists i \text{ s.t. }\vec a^{\top} \vec X_i = n^{-1} \sum_{j \leq n}\vec a^{\top} \vec X_j,\label{eq:average_offset}
        \end{gather}
        then $\FMset_n$ is not unique. Suppose there is some $\bm \theta$ such that:
        \begin{gather}
            \forall j, \, \A(\vec X_{j}-\bm \theta) = \vec a \text{ or }-\vec a,\label{eq:theta_conditions_1}\\ 
            \textstyle\vec a^{\top} \bm \theta = n^{-1}\sum_{j\leq n} \vec a^{\top}\vec X_j.\label{eq:theta_conditions_2} 
        \end{gather}
        Then by \cref{rmk:subgradient}, the subgradient of the sample Fr\'echet function is:
        \[
            \partial f_n(\bm \theta) = 2n^{-1}\textstyle \sum_{j \leq n} \vec a^{\top}(\bm \theta - \vec X_j)\, \vec a = \vec 0,
        \]
        where $F$ is the face with normal vector $\vec a$. Such a $\bm \theta$ is therefore a sample Fr\'echet mean of $\vec X_1, \dots, \vec X_n$. 
        
        We now show that multiple such $\bm \theta$ exist. Without loss of generality, assume our data are labelled such that $\vec a^{\top}\vec X_i < \vec a^{\top} \vec X_j $ whenever $i<j$. Then there is some $i$ such that $\vec a^{\top} \vec X_i < n^{-1}\sum_{j \leq n}\vec a^{\top} \vec X_j < \vec a^{\top} \vec X_{i+1}$. Consider the polyhedron $Q = (\vec X_i + \relint \R_{>0}F) \cap (\vec X_{i+1} - \relint \R_{>0}F)$ --- this is $k$-dimensional, and every $\bm \theta \in Q$ satisfies \cref{eq:theta_conditions_1} as the cones $\vec X_j + \relint \R_{>0}F$ ($\vec X_j - \relint \R_{>0}F$) are nested by \cref{eq:data_cones_nested}. 
        
        The hyperplane $\{ \bm \theta : \textstyle\vec a^{\top} \bm \theta = n^{-1}\sum_{j\leq n} \vec a^{\top}\vec X_j \}$ intersects the relative interior of $Q$, so their intersection is a $k-1$ dimensional set of points which satisfy \cref{eq:theta_conditions_1,eq:theta_conditions_2}. Hence $\FMset_n(\vec X_1, \dots, \vec X_n)$ is a $k-1$ dimensional polytope, and not a unique point.

        It remains to show that with positive probability, we satisfy \cref{eq:data_cones_nested,eq:average_offset}. We prove this by induction. With probability one, $\vec X_2$ is in the full dimensional face cone of some $F$ with respect to $\vec X_1$; that is, $\A(\vec X_2 - \vec X_1)$ is a singleton. Furthermore, as $\vec X_2 \neq \vec X_1$ almost surely, we have $\frac{1}{2} \vec a^{\top}(\vec X_1 + \vec X_2) \neq \vec a^{\top} \vec X_1, \vec a^{\top} \vec X_2$. Hence \cref{eq:data_cones_nested,eq:average_offset} are satisfied with respect to $F$.
        
        Now assume \cref{eq:data_cones_nested,eq:average_offset} are satisfied for $\vec X_1, \dots, \vec X_n$. We have already shown there is a $k$-dimensional set of points $\vec x$ such that $\forall j, \, \A(\vec X_{j}-\vec x) = \vec a \text{ or }-\vec a$, so with positive probability $\vec X_1, \dots, \vec X_{n+1}$ satisfy \cref{eq:data_cones_nested}. Furthermore, there are a finite number of values $\vec a^{\top} \vec X_{n+1}$ cannot take to ensure that there is no $i$ with $\vec a^{\top} \vec X_i = (n+1)^{-1}\sum_{j\leq n+1} \vec a^{\top} \vec X_j$, so \cref{eq:average_offset} holds for $\vec X_1, \dots, \vec X_{n+1}$ almost surely. By induction, \cref{eq:data_cones_nested,eq:average_offset} are satisfied for $\vec X_1, \dots, \vec X_{n+1}$ with positive probability.

        We conclude that for $n \geq 2$, \cref{eq:data_cones_nested,eq:average_offset} are satisfied and $\FMset_n$ is non-unique with positive probability.
	\end{proof}

    In our proof of \cref{thm:uniqueness_prob_convergence}, we have also shown that the probability of uniqueness is always zero for samples of size 2, as our constructive example for non-unique Fr\'echet means occurs almost surely with samples of size 2.
	
	\begin{corollary}
	    \label{rmk:non_unique_for_2n}
		For any polytope norm $\|\cdot\|_B$, Fr\'echet means of two datapoints are almost surely non-unique.
	\end{corollary}
	
	We have used the consistency of sample Fr\'echet means to deduce that the probability of a sample having a unique Fr\'echet mean converges to (but never equals) 1 as our sample size increases, but this probability is zero for samples of size 2. In the following section, we study the uniqueness behaviour of Fr\'echet means for small samples.

     \FloatBarrier

    \section{Sample Threshold for Unique Fr\'echet Means}
	\label{sec:FM_uniqueness}	
	
	In this section, we identify the combinatorial conditions that dictate the sample size needed for the sample Fr\'echet mean to be unique with positive probability. This is a best-case analysis of how many samples we need in order to perform statistical methods which require unique Fr\'echet means such as hypothesis testing. It will be shown that this sample size is a threshold in that there is some value $N(B)>2$ such that the probability of a unique Fr\'echet mean is positive for a sample of size $n$ if and only if $n \geq N(B)$ or $n=1$. To simplify discussion, for the rest of the paper we ignore the trivial case where $n = 1$ and the Fr\'echet mean is unique. We begin by showing that, with probability one, no data point is a Fr\'echet mean of a full sample, so we can assume throughout this section that datapoints have positive distance to the Fr\'echet mean set. 
	\begin{lemma}\label{lemma:data_is_not_FM}
		Suppose $\mathbb P$ has a positive density everywhere and $\sample \overset{i.i.d.}{\sim} \mathbb P$. Almost surely, no sample point $\vec X_i$ is a Fr\'echet mean of the full sample --- for all $i$, $\vec X_i \notin \FMset_n(\sample)$.
	\end{lemma}
	\begin{proof}
		We claim $\vec X_n$ is a Fr\'echet mean of $\sample$ iff $\vec X_n$ is a Fr\'echet mean of $\vec X_1, \dots, \vec X_{n-1}$. Note that:
		\[
		\textstyle f_n(\vec x) = n^{-1} \sum_{i=1}^n d_B(\vec X_i, \vec x)^2, \qquad
		f_{n-1}(\vec x) = (n-1)^{-1} \sum_{i=1}^{n-1} d_B(\vec X_i, \vec x)^2.
		\]
		The subgradient of $d_B(\vec X_n, \cdot)^2$ at $\vec X_n$ is $\vec 0$, so by the additivity of subgradients \citep[Thm 23.8]{rockafellar1970convex}:
		\[
		n \partial f_n(\vec X_n) = (n-1) \partial f_{n-1}(\vec X_n).
		\]
		The point $\vec X_n$ is a Fr\'echet mean of $\sample$ if and only if $\vec 0 \in \partial f_n(\vec X_n)$, and is a Fr\'echet mean of $\vec X_1, \dots \vec X_{n-1}$ if and only if $\vec 0 \in \partial f_{n-1}(\vec X_n)$. These conditions are equivalent, so our claim is proved.
		
        Let $S_n$ be the event that $\vec X_n$ is a Fr\'echet mean of the full sample, or equivalently $\vec X_n \in \FMset(\vec X_1, \dots, \vec X_{n-1})$. By \cref{cor:FMisPolytope}, the set $\FMset(\vec X_1,\dots, \vec X_{n-1})$ has co-dimension $\geq 1$, and hence has $\mathbb P$-measure zero. Therefore $\mathbb P(S_n) = 0$. Similarly, let $S_i$ be the event that $\vec X_i \in \FMset_n(\sample)$, and each will have probability zero by symmetry in the argument above. By taking a union over $S_1, \dots, S_n$ we conclude that the probability that the Fr\'echet mean set has a non-empty intersection with the sample $\sample$ is zero.
	\end{proof}

	\subsection{Face Types}

    To understand the dimension of Fr\'echet mean sets for polytope norms, we categorise samples by their arrangement around their Fr\'echet mean set. In particular, we consider which faces of balls need to be intersected to get the Fr\'echet mean set (\cref{cor:FMisPolytope}), then we can determine the dimension of the Fr\'echet mean set from the minimal faces we need to intersect. We call the minimal such faces the \emph{face type} of a sample. In this subsection, we show that the face type of a sample is almost surely well-defined and determines the dimension of $\FMset_n$.
	
	\begin{lemma}
		\label{lem:unique_least_face_seq}
		Let $\sample$ be a finite sample with Fr\'echet mean set $ \FMset_n$ such that $\vec X_i \notin  \FMset_n$ for every $i$. Let $d_i = d_B(\vec X_i,  \FMset_n) > 0$. Then the set of tuples of faces
		\[
		S_F(\sample) = \bigg\{ (F'_1, \dots, F'_n) \ : \textstyle \bigcap_{i = 1}^n (\vec X_i - d_i F'_i) =  \FMset_n \bigg\} \subseteq \mathfrak{F}(B)^n
		\]
		has a least element $(F_1, \dots, F_n) \in S_F$ with respect to coordinate-wise containment.
	\end{lemma}
	
	\begin{proof}
		By \Cref{cor:FMisPolytope}, $S_F(\sample)$ is not empty. Define $(F_1, \dots, F_n)$ by $F_i = \bigcap_{(F'_1, \dots, F'_n) \in S_F} F'_i$.
		Then for all $(F'_1, \dots, F'_n) \in S_F(\sample)$ we have that $(F_1, \dots, F_n) \subseteq (F'_1, \dots, F'_n)$ elementwise. \\
		It suffices to show that $(F_1, \dots, F_n) \in S_F(\sample)$. As $d_i > 0$, this follows from:
		\begin{align*}
			\textstyle \bigcap_{i \leq n} (\vec X_i - d_i F_i) &= \textstyle \bigcap_{i \leq n} \bigg(\vec X_i - d_i \big( \textstyle \bigcap_{(F'_1, \dots, F'_n) \in S_F} F'_i\big)\bigg) \\
			&= \textstyle \bigcap_{(F'_1, \dots, F'_n) \in S_F} \bigcap_{i \leq n} (\vec X_i - d_i F'_i) \\
			&= \textstyle \bigcap_{(F'_1, \dots, F'_n) \in S_F}  \FMset_n =  \FMset_n.\tag*{\qedhere}
		\end{align*}
	\end{proof}
	Using \cref{lem:unique_least_face_seq}, we can define the face type of a sample:
    
	\begin{definition}[Face Type] \label{defn:face_type}
		Let $\sample$ be a sample with Fr\'echet mean set $\FMset_n$ such that for all $i$, $\vec X_i \notin \FMset_n$. This sample has face type $(F_1, \dots, F_n)$ about $ \FMset_n$ when $(F_1, \dots, F_n)$ is the least element of $S_F(\sample)$.
	\end{definition}
	
	By \cref{lemma:data_is_not_FM} and \cref{lem:unique_least_face_seq}, the face type of a sample is well-defined almost surely. Face type is not defined when there is some sample point $\vec X_i$ in $\FMset_n$; this happens with probability zero (\cref{lemma:data_is_not_FM}), but would give a unique Fr\'echet mean (\cref{lemma:FM_dists}). We ignore this case for the rest of this paper. 
    
    In the next lemma, we give a more geometric description of face type which will prove more useful. Throughout the rest of this section we set $d_i \coloneqq d_B(\vec X_i,  \FMset_n)$, which is well-defined by \cref{cor:FMisPolytope}.
	
	\begin{lemma} \label{lem:geometry_face_type}
		Let $\sample$ be a sample with face type $(F_1, \dots F_n)$ and $d_i = d_B(\vec X_i, \FMset_n)$. Then we have that 
		$
		\relint  \FMset_n = \textstyle\bigcap_{i \leq n } \relint (\vec X_i - d_i F_i).
		$
		A sample has face type $(F_1, \dots, F_n)$ if and only if there is some $ \bm \theta \in \relint  \FMset_n$ such that for all $i \leq n$, $\vec X_i -  \bm \theta \in \relint (\R_{>0}F_i)$.
	\end{lemma}
	\begin{proof}
		By definition, $ \FMset_n \subseteq (\vec X_1 - d_1 F_1)$.     Suppose that $\relint  \FMset_n$ does not intersect $\relint (\vec X_1 - d_1 F_1)$, so $\relint  \FMset_n \subseteq (\vec X_1 - d_1 \partial F_1)$. By \cref{cor:convex_disjoint_from_relint}, there is some proper face $F_1'$ of $F_1$ satisfying $ \FMset_n \subseteq \vec X_1 - d_1 F_1'$. Hence $(F_1', F_2, \dots, F_n) \in S_F(\vec X_1, \dots,\vec X_n)$, contradicting the definition of $(F_1, \dots, F_n)$. Therefore (for any $i$, by symmetry) $\relint  \FMset_n$ intersects $\relint (\vec X_i - d_i F_i)$, and by \cite[Theorem 6.5]{rockafellar1970convex}:
		\begin{align*}
			\relint \FMset_n = \relint ( \FMset_n \cap (\vec X_i - d_i F_i)) &= \relint  \FMset_n \cap \relint (\vec X_i - d_i F_i), \\
			 \Rightarrow \quad \relint  \FMset_n &\subseteq \relint (\vec X_i - d_i F_i).
		\end{align*}
		Applying \cite[Theorem 6.5]{rockafellar1970convex} again to $\vec X_1 - d_1 F_1, \dots, \vec X_n - d_n F_n$, all of which contain a common point from $\relint  \FMset_n$:
		\begin{align*}
			\relint  \FMset_n = \textstyle \bigcap_{i \leq n} \relint (\vec X_i - d_i F_i).
		\end{align*}
		Suppose our sample has face type $F_1, \dots, F_n$. Then for all $ \bm \theta \in \relint  \FMset_n$, we have $ \bm \theta \in \relint (\vec X_i - d_i F_i)$ and therefore $\vec X_i -  \bm \theta \in \relint \R_{>0} F_i$.
		
		To show the converse, consider $ \bm \theta \in \relint  \FMset_n$ such that for all $i$, $\vec X_i -  \bm \theta \in \relint (\R_{>0} F_i)$. We have $d_i = d_B(\vec X_i,  \bm \theta)$. Assume that $(\sample)$ has the face type $F_1', \dots, F_n'$. Then by the argument above, $ \bm \theta \in \relint(\vec X_i - d_i F_i')$, implying that $\vec X_i -  \bm \theta \in d_i \relint F_i' \subseteq \relint (\R_{>0} F_i')$. The fans $\{ \relint (\R_{>0}F) : F \in \mathfrak F (B) \}$ partition $\R^k \setminus \{ \vec 0 \}$, so it must be the case that $F_i' = F_i$ and our sample has type $(F_1, \dots, F_n)$. 
	\end{proof}
	
	As the relative interior of the Fr\'echet mean set is the intersection of the relative interiors of the faces $\vec X_i - d_i F_i$, we can deduce the dimension of $ \FMset_n$ from the face type of the sample by the following lemma.
	
	\begin{lemma}\label{lemma:affine_intersection}
		If $W_1,\ldots,W_m$ are relatively open and $\textstyle \bigcap_{i\leq m} W_i \neq \emptyset$, then
		$$
		\textstyle\bigcap_{i\leq m}\mathrm{aff}(W_i) = \mathrm{aff} \left( \textstyle\bigcap_{i\leq m} W_i \right).
		$$
	\end{lemma}
	\begin{proof}
		It suffices to consider the $m=2$ case. The space $\text{aff}(W_1) \cap \text{aff}(W_2)$ is affine and contains $W_1 \cap W_2$, so we get that $\text{aff}(W_1 \cap W_2) \subseteq \text{aff}(W_1) \cap \text{aff}(W_2)$.\\
		Now fix $\vec y \in \text{aff}(W_1) \cap \text{aff}(W_2)$, $\vec x \in W_1 \cap W_2$. As $W_1,W_2$ are relatively open, there exists some $\epsilon > 0$ such that $(1-\epsilon)\vec x + \epsilon \vec y \in W_1 \cap W_2$. Hence $\vec y \in \text{aff}(\{ \vec x, (1-\epsilon)\vec x + \epsilon \vec y \}) \subseteq \text{aff}(W_1 \cap W_2)$.     
	\end{proof}

    We can reframe our proof of \cref{thm:uniqueness_prob_convergence} in terms of the language of face types. We proved that if every facet appears in the face type of a sample, we will get a unique Fr\'echet mean, and the probability of this occurring tends to one. We also showed that with positive probability, the face type of a sample will consist only of $F$ and $-F$ for some facet $F$, which cannot give rise to a unique Fr\'echet mean. In the next subsections, we show how to make similar statements for general face types --- we identify which face types occur with positive probability, and which produce a unique Fr\'echet mean.
	
	\subsection{Type Events}
	
	By \Cref{lemma:affine_intersection}, we can determine the dimension of a Fr\'echet mean set from the face type of a sample. This can tell us whether a unique Fr\'echet mean can occur if we know which face types occur with positive probability, and identify which face types correspond to a unique Fr\'echet mean. To determine which face types occur with positive probability, we define the (lifted) type event. Throughout this section, we use $C_i$ and $G_i$ to denote $\R_{>0} F_i$ and $F_i^{\Diamond}$ respectively.
	
	\begin{definition}[(Lifted) Type Event]
		Consider the random vector $(\sample) \in \R^{nk}$. For a sequence of proper faces $F_1, \dots, F_n \in \mathfrak{F}(B)^n$, the \emph{type event} $U_{\CType}$ is given by:
		\[
		U_{F_1, \dots, F_n} = \{ (\sample) \in \R^{nk}: (\sample) \text{ has face type }(F_1, \dots F_n) \}.
		\]
		Then we define the \emph{lifted type event} as the set
		\[
		U_{F_1, \dots, F_n, \Theta} = \{ (\sample,  \bm \theta) \in \R^{k(n+1)} :  \bm \theta \in \relint  \FMset_n, 
		\; \vec X_i -  \bm \theta \in \relint C_i \},
		\]
        where $C_i$ denotes the cone $\R_{>0} F_i$.
	\end{definition}
	
	Note that by \Cref{lem:geometry_face_type}, the type event is the image of the lifted type event under the projection $\proj: \R^{k(n+1)} \rightarrow \R^{kn}$ given by $(\sample,  \bm \theta) \mapsto (\sample)$.
	
	Our goal is to classify the face types $(\CType)$ for which $U_{\CType}$ has positive Lebesgue measure. Then as $(\sample)$ has a positive density, such a face type occurs with positive probability. It is difficult to calculate the measure of the type events directly, so we will translate this condition into a linear relation on subgradients. For this, we define the \emph{zero subgradient event}.
	
	\begin{definition}[(Projected) Zero Subgradient Event]
		The zero subgradient event is given by:
		\[
		V_{\CType, \Theta} = \left\{ (\sample,  \bm \theta) \in \R^{k(n+1)} : \vec 0 \in \relint \left( \textstyle \sum_{i\leq n} \|\vec X_i- \bm \theta\| G_i \right),\; \vec X_i -  \bm \theta \in \relint C_i \right\},
		\]
        where $G_i$ denotes the polar face $F_i^{\Diamond}$. The projected zero subgradient event is given by $V_{\CType} \coloneqq \proj V_{\CType,\Theta}$, the projection to its first $nk$ coordinates.
	\end{definition}
	
	Note that when we have $\vec X_i -  \bm \theta \in \relint C_i$ for all $i$, the subgradient of the Fr\'echet function is given by $\partial f_n( \bm \theta) = -2n^{-1}\sum_{i \leq n} \| \vec X_i -  \bm \theta \|G_i$. Hence, by the definition of subgradients, $ \bm \theta \in  \FMset_n$ if and only if $\vec 0 \in \sum_{i\leq n} \| \vec X_i -  \bm \theta \|G_i$ (\cref{rmk:subgradient}). From this, we might hope that the lifted type event and the zero subgradient event are equivalent, but this is not true. However, they do coincide almost everywhere.
	
	\begin{restatable}{proposition}{EventEquivalence}\label{prop:lifted_type_event}
		The zero subgradient event $V_{\CType, \Theta} \subset \R^{k(n+1)}$ is a relatively open polyhedron; it is an intersection of finitely many hyperplanes and open half spaces. Furthermore, it coincides with the lifted type event everywhere other than on its boundary:
		\begin{align*}
			\relint V_{\CType, \Theta} &= \relint U_{\CType, \Theta},\\
			\closure V_{\CType, \Theta} &= \closure U_{\CType, \Theta}.
		\end{align*}
	\end{restatable}
	
	The proof of this proposition is built on the following three technical lemmas, whose proofs can be found in \cref{appsec:lemma_proofs}.

    \begin{restatable}{lemma}{polyhedronW}\label{lemma:polyhedral_cone_W}
		The set $W = \{\vec w \in \R^n_{>0} : \vec 0 \in \sum_{i \leq n} w_i \relint G_i\}$ is a relatively open polyhedral cone.
	\end{restatable}

    The set $W$ defined in \cref{lemma:polyhedral_cone_W} is the set of possible distances $(d_B(\vec X_1, \FMset_n), \dots, d_B(\vec X_n, \FMset_n))$ which balance the subgradients and ensure that $\vec 0 \in \partial f_n(\bm \theta)$.

    \begin{restatable}{lemma}{Containment}\label{lemma:RelintGrad_to_RelintFMset}
		Given any points  $(\sample,  \bm \theta) \in \R^{k(n+1)}$, if 
		$\vec 0 \in \relint(\partial f_n( \bm \theta))$ then $ \bm \theta \in \relint  \FMset_n$.
	\end{restatable}

    \begin{restatable}{lemma}{Limit}\label{lemma:U_in_ClV}
		For all $\epsilon > 0$, $(\vec X_1,\dots, \vec X_n, \bm \theta) \in U_{F_1,\dots, F_n,\Theta}$, there is some $(\vec Y_1, \dots, \vec Y_n, \bm \theta) \in V_{\CType, \Theta}$ such that $\|\vec X_i - \vec Y_i\|_B < \epsilon$.
	\end{restatable}

    	\begin{proof}[Proof of \cref{prop:lifted_type_event}]
		For each $i \leq n$, fix some $\vec a_i$ in $\A(F_i)$. Then for all $\vec X_1,\dots \vec X_n,  \bm \theta$ such that $\vec X_i -  \bm \theta \in \relint C_i$, we have that $\| \vec X_i - \bm \theta \|_B = \vec a_i^{\T}(\vec X_i -  \bm \theta)$. By the definition of $W$ (\cref{lemma:polyhedral_cone_W}), we conclude that
		\[
		\vec 0 \in \relint \left( \textstyle \sum_{i \leq n} \|\vec X_i-  \bm \theta\|G_i \right) \;\;\text{ if and only if }\;\; \left(\vec a_1^{\T} (\vec X_1- \bm \theta), \dots, \vec a_n^{\T} (\vec X_n- \bm \theta)\right) \in W.
		\]
		Let $L_{\CType, \Theta}:\R^{k(n+1)}\rightarrow \R^n$ be the linear map $(\sample, \bm \theta) \mapsto \left(\vec a_1^{\T} (\vec X_1- \bm \theta), \dots, \vec a_n^{\T} (\vec X_n- \bm \theta)\right)$, so:
		\[
		V_{\CType, \Theta} =   L_{\CType, \Theta}^{-1}(W) \; \textstyle\bigcap \; \{ (\sample,  \bm \theta) : \vec X_i -  \bm \theta \in \relint C_i, \; i \leq n \}.
		\]
		The set $\{(\vec X_1,\ldots,\vec X_n, \bm \theta): \vec X_i -  \bm \theta \in \relint C_i \}$ is a relatively open polyhedral cone, as is $L_{\CType, \Theta}^{-1}(W)$ because $W$ is a relatively open polyhedral cone. So $V_{\CType, \Theta}$ is a relatively open polyhedral cone. 
		
		By \cref{lemma:RelintGrad_to_RelintFMset,lemma:U_in_ClV}, we deduce that $V_{\CType,\Theta} \subseteq U_{\CType,\Theta} \subseteq \closure V_{\CType,\Theta}$. Taking the closure for each containment gives $\closure V_{\CType\Theta} = \closure U_{\CType\Theta}$. They therefore have the same affine span, and so taking interiors within their affine span gives:
		\begin{align*}
			\relint V_{\CType,\Theta} \subseteq \relint U_{\CType\Theta} \subseteq \relint \closure V_{\CType\Theta}.
		\end{align*}
		As $V_{\CType,\Theta}$ is convex, Theorem 6.3 of \cite{rockafellar1970convex} ensures $\relint V_{\CType,\Theta} = \relint U_{\CType\Theta}$.
	\end{proof}
	
	\subsection{Polytope Conditions for Unique Fr\'echet Means}
	
	We can now state and prove the main result of this section (\cref{thm:wwp_dim_condition}), which establishes the conditions for $\sample$ to have face type $(F_1, \dots, F_n)$ with positive probability and to produce a unique Fr\'echet mean. We do this by computing the dimension of the polyhedron $V_{\CType,\Theta}$, which then allows us to compute the dimension of $ V_{\CType}$ and $U_{\CType}$. This result is crucial for not only computing the sample threshold for uniqueness, but also for proving that the sample threshold for uniqueness is well-defined.
	
	\begin{theorem}\label{thm:wwp_dim_condition}
		The type event $U_{\CType}$ and projected zero subgradient event $V_{\CType}$ have the same Lebesgue measure, 
		$
		\L_{}(U_{\CType}) = \L(V_{\CType}).
		$ Therefore $U_{\CType}$ has positive probability if and only if
		\begin{gather}
			\vec 0 \in \relint \left( \conv \left(\textstyle \bigcup_{i \leq n} G_i \right) \right) \label{eq:wpp_cond_possible}, \\
			\text{and} \qquad \textstyle \sum_{i \leq n} \dim G_i = \dim \textstyle \sum_{i \leq n} \aff G_i. \label{eq:wpp_cond_probable}
		\end{gather}
		The Fr\'echet mean of a sample in $U_{\CType}$ is unique if and only if 
		\begin{align}\label{eq:wpp_cond_unique}
			\dim \conv \left( \textstyle \bigcup_{i \leq n} G_i \right) = k.
		\end{align}
	\end{theorem}
	
	To understand this result, we return to our 3 point sample from \cref{fig:FM_sets}, displayed again in \cref{fig:Face_Type_Example}.

    \begin{example}
    \cref{fig:Face_Type_Example} shows a sample with face type $(F_1,F_2,F_3) = (\vec e_1+\vec e_2, \conv \{ -\vec e_2 \pm \vec e_1 \}, \conv \{ -\vec e_1 \pm \vec e_2 \})$. We can test whether this face type satisfies the three conditions in \cref{thm:wwp_dim_condition}; for this face type, note that $(G_1, G_2, G_3) = (\conv \{\vec e_1, \vec e_2 \}, -\vec e_2, -\vec e_1)$. \\
	The first condition \cref{eq:wpp_cond_possible} verifies whether a sample can have the face type $(F_1,F_2, F_3)$. This is shown to be true in \cref{fig:Face_Type_Example}, and in fact $\conv \left(\textstyle \bigcup_{i \leq 3} G_i \right) = \conv \{\pm\vec e_1,\pm\vec e_2 \}$ which contains $\vec 0$ in its relative interior.\\
	The second condition \cref{eq:wpp_cond_probable} verifies whether a sample has face type $(F_1,F_2,F_3)$ with positive probability. The dimensions of $G_1,G_2,G_3$ are 1,0,0 respectively, while $\sum_{i \leq 3} \aff G_i = \aff \{ -\vec e_1, -\vec e_2 \}$, which has dimension $1=1+0+0$. This verifies that such a face type can occur with positive probability; this configuration of data points around a Fr\'echet mean is generic.\\
	Finally, the third condition \cref{eq:wpp_cond_unique} verifies whether having face type $(\CType)$ implies that our sample has a unique Fr\'echet mean. We know $\conv \left(\textstyle \bigcup_{i \leq 3} G_i \right) = \conv \{\pm\vec e_1,\pm\vec e_2 \}$ which has dimension 2, and therefore any sample with this face type will have unique Fr\'echet mean. \\
	
	\begin{figure}
		\centering
		\begin{tikzpicture}[scale=0.9]
			\draw (0,0) -- (-4, 0) -- (-4,-4) -- (0,-4) -- (0,0);
			\draw (-0.5,0) -- (1.5, 0) -- (1.5,2) -- (-0.5,2) -- (-0.5,0);
			\draw (0,0.5) -- (0,-1.5) -- (2,-1.5) -- (2,0.5) -- (0,0.5);
			\filldraw (-2,-2) circle (1pt) node[anchor=north east] {\small$\vec x_1 = (-4,-4)$};
			\filldraw (0.5,1) circle (1pt)  node[anchor=south] {\small$\vec x_2 = (1,2)$};
			\filldraw (1,-0.5) circle (1pt)  node[anchor=north west] {\small$\vec x_3 = (2,-1)$};
			\filldraw[red] (0,0) circle (3pt);
		\end{tikzpicture}  
		\caption{A sample with face type $(\CType) = (\vec e_1+\vec e_2, \conv \{ - \vec e_2 \pm \vec e_1 \}, \conv \{ - \vec e_1 \pm \vec e_2 \})$, and its unique Fr\'echet mean highlighted in red.}
		\label{fig:Face_Type_Example}
	\end{figure}
    \end{example}
    
	To prove \cref{thm:wwp_dim_condition}, we first prove the following three lemmas.
	
	\begin{lemma}\label{lemma:possible}
		The zero subgradient event $V_{\CType,\Theta}$ is non-empty if and only if $\vec 0 \in \relint \left( \conv \left(\textstyle \bigcup_{i \leq n} G_i \right) \right)$.
	\end{lemma}
	
	\begin{proof}
		First, we have shown in the proof of \cref{prop:lifted_type_event} that
		\begin{equation}
			V_{\CType, \Theta} = L_{\CType}^{-1}(W) \textstyle \bigcap \bigcap_{i \leq n} \{ (\sample,  \bm \theta) : \vec X_i -  \bm \theta \in \relint C_i \},\label{eq:zero_subgrad_event_as_intersection}
		\end{equation}
        where $W$ is the relatively open polyhedral cone $\{\vec w \in \R^n_{>0} : \vec 0 \in \sum_{i \leq n} w_i \relint G_i\}$ (\cref{lemma:polyhedral_cone_W}). 
        
		We will first show that $V_{\CType,\Theta}$ is not empty if and only if $W$ is non-empty. Suppose $(\vec X_1,\dots, \vec X_n, \bm \theta) \in V_{\CType,\Theta}$; then by \cref{eq:zero_subgrad_event_as_intersection}, $W$ is non-empty. Conversely, suppose $\vec w \in W$. Fix an arbitrary $ \bm \theta$ and pick some $\vec X_i \in  \bm \theta + w_i \relint F_i$. By construction, $\vec X_i -  \bm \theta \in \relint C_i$ and $\| \vec X_i -  \bm \theta \|_B = w_i$. By the definition of $W$ we have $
		\vec 0 \in \relint \left( \textstyle \sum_{i \leq n} \| \vec X_i -  \bm \theta \|_B G_i \right),
		$ so $(\sample, \bm \theta) \in V_{\CType,\Theta}$.
		
		Next, note that \cite[Theorem 6.9]{rockafellar1970convex} states that for non-empty convex $G_1, \dots, G_n \subset \R^k$:
		\begin{equation}\label{eq:rockafellar_thm6.9}
			\relint \left( \conv \left( \textstyle \bigcup_{i \leq n} G_i \right) \right) = \left \{ \textstyle \sum_{i \leq n} \lambda_i \relint G_i : \lambda_i > 0, \sum_{i \leq n} \lambda_i = 1 \right \}.
		\end{equation}
		From this we see that $\vec 0 \in \relint \left( \conv \left(\textstyle \bigcup_{i \leq n} G_i \right) \right)$ if and only if $W$ is non-empty. We conclude that $V_{\CType, \Theta}$ is non-empty if and only if $\vec 0 \in \relint \left( \conv \left(\textstyle \bigcup_{i \leq n} G_i \right) \right)$. 
	\end{proof}
	
	\begin{lemma}\label{lemma:probable}
		The type event $U_{\CType}$ has positive probability if and only if the polyhedron $V_{\CType}$ has dimension $nk$.
	\end{lemma}
	
	\begin{proof}
		The zero subgradient event $V_{\CType,\Theta}$ is Borel measurable as it is a relatively open polyhedron (\cref{prop:lifted_type_event}). Consequently, $V_{\CType}$ is also a relatively open polyhedron and so also Borel measurable.
		We also have by \cref{prop:lifted_type_event} and Theorem 6.6 of \cite{rockafellar1970convex}:
		\begin{align*}
			U_{\CType} &= (\proj U_{\CType,\Theta}),        &\quad U_{\CType} &= \proj U_{\CType,\Theta},\\
            &\supseteq \proj (\relint U_{\CType,\Theta}),   & & \subseteq \proj (\closure U_{\CType,\Theta}),\\
            &= \proj (\relint V_{\CType,\Theta}),           & &= \proj (\closure V_{\CType,\Theta}),\\
            &= \relint (\proj V_{\CType,\Theta}),           & &\subseteq \closure (\proj V_{\CType,\Theta}),\\
            &= \relint (V_{\CType}).                        & &= \closure (V_{\CType}).
		\end{align*}
		Hence $V_{\CType}$ is a measurable set which differs from $U_{\CType}$ only on its boundary. The boundary of $V_{\CType}$ is a measure zero set as it is a relatively open polyhedron, so by the completeness of the Lebesgue measure, $U_{\CType}$ is measurable and $\L(U_{\CType}) = \L(V_{\CType})$.
		
		Being a polyhedron, $V_{\CType}$ has positive Lebesgue measure if and only if it has full dimension $nk$. 
	\end{proof}

	\begin{lemma} \label{lemma:unique}
		A sample with face type $(\CType)$ will have a Fr\'echet mean set of dimension 
        \begin{equation}\label{eq:expr_for_d}
            k - \dim \conv \left(\textstyle \bigcup_{i \leq n} G_i \right).
        \end{equation}
	\end{lemma}
	
	\begin{proof}
		We recall \cref{lemma:affine_intersection}; we have that $\relint  \FMset_n = \bigcap_{i \leq n} \relint (\vec X_i - d_i F_i)$, so fixing some $\bm \theta \in \FMset_n$:
		\begin{align*}
			\dim \FMset_n  &= \dim \textstyle\bigcap_{i \leq n} \aff (\vec X_i - d_i F_i), \nonumber\\
			&=\dim \left\{  \bm \theta + \vec v :  \bm \theta \in \textstyle\bigcap_{i \leq n} (\vec X_i - d_i F_i), \vec v \in G_i^{\perp} \text{ for all }i \leq n \right\}, \nonumber\\
			&=\dim \left(\textstyle \bigcup_{i \leq n} G_i \right)^{\perp}, \nonumber\\
			&=k - \dim \conv \left(\textstyle \bigcup_{i \leq n} G_i \right).\tag*{\qedhere} 
		\end{align*}
	\end{proof}

    We can now prove \cref{thm:wwp_dim_condition} using \cref{lemma:possible,lemma:probable,lemma:unique}. \cref{lemma:possible} leads directly to condition \cref{eq:wpp_cond_possible}, and dictates whether any sample can have the specified face type. \cref{lemma:unique} directly leads to condition \cref{eq:wpp_cond_unique}, which dictates whether the induced Fr\'echet mean set is a unique point. Most of the proof below involves computing the dimension of $V_{\CType,\Theta}$ in order to translate \cref{lemma:probable} into a more direct condition on the polar faces $G_1, \dots, G_n$ (as in condition \cref{eq:wpp_cond_probable}).
	
	\begin{proof}
		By \cref{lemma:possible} and \cref{prop:lifted_type_event}, if \cref{eq:wpp_cond_possible} does not hold, then $V_{\CType,\Theta}$ and $U_{\CType}$ are empty. We therefore assume that \cref{eq:wpp_cond_possible} holds, implying that $V_{\CType,\Theta}$ is non-empty, and then compute its dimension. 
		
		As $V_{\CType,\Theta}$ is non-empty, we can use our expression for $V_{\CType,\Theta}$ from the proof of \cref{prop:lifted_type_event} and apply \cref{lemma:affine_intersection}. This gives:
		\begin{align*}
			\text{aff}(V_{\CType,\Theta}) &= \aff L_{\CType, \Theta}^{-1}(W) \textstyle \bigcap \bigcap_{i \leq n} \text{aff} \{ (\sample,  \bm \theta) : \vec X_i -  \bm \theta \in \relint C_i \},
		\end{align*}
		where $W$ is as defined in \cref{lemma:polyhedral_cone_W} and $L_{\CType,\Theta}$ is the linear map given by $L_{\CType,\Theta}: (\sample, \bm \theta) \mapsto \left(\vec a_1^{\T} (\vec X_1- \bm \theta), \dots, \vec a_n^{\T} (\vec X_n- \bm \theta)\right)$, where $\vec a_i$ is a fixed element of $\A(F_i)$. As $L_{\CType,\Theta}$ is surjective, taking the affine span commutes with the inverse $L_{\CType, \Theta}^{-1}$, so we have that:
		\begin{align*}
			\text{aff}(V_{\CType,\Theta}) &= L_{\CType,\Theta}^{-1}(\aff (W)) \textstyle \bigcap \bigcap_{i \leq n} \text{aff} \{ (\sample,  \bm \theta) : \vec X_i -  \bm \theta \in \relint C_i \}.
		\end{align*}
		The affine hull of $C_i$ contains those points $\vec x$ for which $\vec a^{\T}\vec x$ is constant for all $\vec a \in \A(F_i)$; it is the orthogonal to $\aff \A(F_i) = \aff G_i$. Hence $\aff C_i$ has dimension 
		\begin{align*}
			\dim \aff C_i &= k - \dim \aff \A(F_i) = k - \dim G_i.
		\end{align*}
		Noting this, we can pick $\vec a_i^1 \dots, \vec a_i^{\dim G_i}$ in $\A(F_i)$ such that $\vec a_i, \vec a_i^1 \dots, \vec a_i^{\dim G_i}$ form an affine basis of $\aff \A(F_i)$, and $\aff C_i$ contains those $\vec x$ such that for all $j \leq \dim G_i$, we have $\vec a_i^{j\T}\vec x = \vec a_i^{\T}\vec x$. Then \begin{align*}
			\aff \{ (\sample, \bm \theta) : \vec X_i -  \bm \theta \in \relint C_i \} &= \{ (\sample, \bm \theta) : \vec X_i -  \bm \theta \in \aff C_i \}, \\
			&= \{ (\sample, \bm \theta) : \forall \, j \leq \dim G_i, 
			\; (\vec a_i^j - \vec a_i)^{\T}(\vec X_i -  \bm \theta) = 0 \}.
		\end{align*}
		We define $M_i$ to be the linear map sending $(\sample, \bm \theta)$ to $\left((\vec a_i^{j}-\vec a_i)^{\T}(\vec X_i- \bm \theta)\right)_{j \leq \dim G_i}$, such that $\aff \{ (\sample, \bm \theta) : \vec X_i -  \bm \theta \in \relint C_i \} = M_i^{-1}(\vec 0)$. Then 
		\begin{align*}
			\text{aff}(V_{\CType,\Theta}) &= L_{\CType,\Theta}^{-1}(\aff(W)) \cap M_1^{-1}(\vec 0) \cap \dots \cap M_n^{-1}(\vec 0), \\
			&= (L_{\CType,\Theta},M_1,\dots,M_n)^{-1}(\aff(W),\vec 0,\dots,\vec 0).
		\end{align*}
		The linear map $(L_{\CType,\Theta},M_1,\dots,M_n) : \R^{k(n+1)} \rightarrow \R^{n + \sum_{i \leq n} \dim G_i}$ is given by the matrix:
		$$
		M = \begin{pmatrix}
			\vec a_1^{\T} & \dots & \vec 0^\T & -\vec a_1^{\T} \\
			\vdots & \ddots & \vdots & \vdots \\
			\vec 0^\T & \hdots & \vec a_n^{\T} & -\vec a_n^{\T} \\
			(\vec a_1^1 - \vec a_1)^{\T}& \hdots & \vec 0^\T & (\vec a_1 - \vec a_1^1)^{\T} \\
			\vdots & \vdots & \vdots & \vdots \\
			(\vec a_1^{\dim G_1}- \vec a_1)^\T& \hdots & \vec 0^\T & (\vec a_1 - \vec a_1^{\dim G_1} )^\T \\
			\vdots & \ddots & \vdots & \vdots \\
			\vec 0^\T & \hdots & (\vec a_n^1 - \vec a_n)^\T & (\vec a_n - \vec a_n^1)^{\T} \\
			\vdots & \vdots & \vdots & \vdots \\
			\vec 0^\T & \hdots & (\vec a_n^{\dim G_n} - \vec a_n)^{\T} & (\vec a_n - \vec a_n^{\dim G_n})^{\T} \\
		\end{pmatrix}.
		$$
		Subtracting row $i$ from the rows corresponding to $M_i$, we see the rank of this matrix is given by:
		\begin{align*}
			\textstyle \sum_{i \leq n} \dim \text{span} (\vec a_i, \vec a_i^1, \dots, \vec a_i^{\dim G_i}) &= \textstyle \sum_{i \leq n} 1+ \dim \aff (\vec a_i, \vec a_i^1, \dots, \vec a_i^{\dim G_i}), \\
			&= n + \textstyle \sum_{i \leq n} \dim G_i,
		\end{align*}
		where the first line is because $\aff (\vec a_i, \vec a_i^1, \dots, \vec a_i^{\dim G_i}) = \aff \A(F_i)$ doesn't contain the origin; it is a proper face of $B^{\Delta}$, and $\vec 0$ is in the relative interior of $B^{\Delta}$.\\    
		It follows that the matrix $M$ is surjective as its rank is equal to the dimension of its codomain.  Let $T: \R^{n + \sum_{i \leq n} \dim G_i} \rightarrow \R^{n + \sum_{i \leq n} \dim G_i-\dim W}$ be the linear projection onto the orthogonal complement of $(\aff(W),\vec 0,\dots,\vec 0)$ such that the kernel of $T$ is exactly $(\aff(W),\vec 0,\dots,\vec 0)$. 
		Then by construction, the map $T \circ M: \R^{k(n+1)} \rightarrow \R^{n + \sum_{i \leq n} \dim G_i-\dim W}$ is surjective with kernel $\aff(V_{\CType,\Theta})$ and by the rank-nullity theorem:
		\begin{align*}
			\dim V_{\CType,\Theta} &= k(n+1) - \rank (T \circ M) = k(n+1) - n - \textstyle \sum_{i \leq n} \dim G_i + \dim W.
		\end{align*}
		Next, we compute the dimension of $V_{\CType}$. For this, we will use \cite[Theorem 2.4]{balas1998dimension}, which states:
		\begin{theorem*}[Theorem 2.4 of \cite{balas1998dimension}]
			Consider a polytope $Q \subset \R^p \times \R^q$ with an affine hull that is given by $\{ (\vec u,\vec x) \in \R^p \times \R^q : A\vec u+B\vec x = \vec b\}$. Then $\dim \proj_{\R^q}(Q) = \dim Q - \dim \mathrm{Ker}(A)$.
		\end{theorem*}
		We apply this theorem to $V_{\CType,\Theta}$; define matrices $M^1, M^2$ and vector $\vec b$ to be such that
		\[
		\aff V_{\CType,\Theta} = \{(\vec x_1, \dots, \vec x_n, \bm \theta) \in \R^{nk} \times \R^{k} :  M^1 (\vec x_1, \dots, \vec x_n) + M^2  \bm \theta = \vec b \}.
		\]
		As $U_{\CType,\Theta}$ and $V_{\CType,\Theta}$ have the same relative interior (\cref{prop:lifted_type_event}), their affine spans are the same. So by the definition of $U_{\CType,\Theta}$ and \cref{lem:geometry_face_type}, we can infer that for $(\vec X_1, \dots, \vec X_n) \in V_{\CType}$, the set $\{ \bm \theta \in \R^k : M^2  \bm \theta = \vec b - M^1(\sample) \}$ is the affine hull of the Fr\'echet mean set $\aff  \FMset_n(\sample)$. Therefore $\dim \Ker M^2 = \dim \FMset_n(\sample)$. Using \cite[Theorem 2.4]{balas1998dimension} and \cref{lemma:unique}, we see that
		\begin{align}
			\dim V_{\CType} &= \dim V_{\CType,\Theta} - \dim \FMset_n, \nonumber\\
			&= nk - n - \textstyle\sum_{i \leq n} \dim G_i + \dim W + \dim \conv (\bigcup_{i \leq n} G_i).\label{eq:init_U_dim}
		\end{align}
		We now look to express $\dim W$ in terms of the polytope unit ball $B$. Recall that we have assumed $W$ to be non-empty. We will express the dimension of $W$ in terms of the \emph{Cayley polytope} of $G_1, \dots, G_n$ which we denote by $P_{G_1, \dots, G_n}$. This is given by
		\begin{align*}
			P_{G_1,\dots,G_n} &= \conv \{ \vec e_i \times G_i \subset \R^n \times \R^k : i \leq n \}, \\
			&= \left \{ (\vec w, \vec g) \in \R^n \times \R^k: \vec w \in \Delta_{n-1}, \, \vec g = \textstyle \sum_{i \leq n} w_i \vec g_i, \, \vec g_i \in G_i \right \},
		\end{align*}
		where $\Delta_{n-1}$ denotes the $n-1$ dimensional simplex. By examining this definition, we see that $\vec w \in W \cap \Delta_{n-1}$ if and only if $(\vec w, \vec 0) \in \relint P_{G_1, \dots, G_n}$. As $W$ is a cone in the positive orthant (\cref{lemma:polyhedral_cone_W}) and non-empty, $W \cap \Delta_{n-1} = \relint P_{G_1, \dots, G_n} \cap (\R^n \times \vec 0)$ is also non-empty.
		
		We once again use \cite[Theorem 2.4]{balas1998dimension}; we now take $M^1,M^2, \vec b$ to be such that 
		$$
		\aff P_{G_1, \dots, G_n} = \{(\vec w,\vec g) \in \R^n \times \R^k : M^1 \vec w + M^2 \vec g = \vec b \}.
		$$
		Note that $\vec b - M^2 \vec 0$ is in the image of $M^1$ because $\relint P_{G_1, \dots, G_n}$ intersects $(\R^n \times \vec 0)$, so by \cite[Theorem 2.4]{balas1998dimension}: 
		\begin{align}
			\dim P_{G_1,\dots,G_n} - \dim \proj_{\R^k}(P_{G_1,\dots,G_n}) &= \dim (M^1)^{-1}(\vec b - M^2 \vec 0), \nonumber\\
			&= \dim \{ \vec w \in \R^n : (\vec w,\vec 0) \in \aff P_{G_1, \dots, G_n} \}, \nonumber\\
			&= \dim \left( \aff P_{G_1, \dots, G_n} \textstyle\cap (\R^n \times \vec 0) \right) \nonumber\\
            &= \dim \left(W \textstyle\cap \Delta_{n-1}\right).\label{eq:cayley_proj_dim}
		\end{align}
        We can also calculate $\dim P_{G_1,\dots,G_n}$ and $\dim \proj_{\R^k}(P_{G_1,\dots,G_n})$ directly. For any $\vec w \in \relint \Delta_{n-1}$, we have that $\vec b - M^1 \vec w$ is in the image of $M^2$. We pick $\vec w = (1/n,\dots, 1/n)$:
		\begin{align*}
			\dim P_{G_1,\dots,G_n} &= \dim \proj_{\R^n}(P_{G_1,\dots,G_n})+ \dim (M^2)^{-1}(\vec b - n^{-1}M^1\vec 1) \\
			&= \dim \Delta_{n-1} + \dim \{ \vec g \in \R^k : ((1/n,\dots, 1/n),\vec g) \in \aff P_{G_1,\dots,G_n} \} \\
			&= n-1+ \dim \aff \left\{ 1/n \textstyle\sum_{i \leq n} \vec g_i: \vec g_i \in \aff G_i \right\} \\
			&= n-1+ \dim \textstyle \sum_{i \leq n} \aff G_i.
		\end{align*}
		Writing out $\dim \proj_{\R^k}(P_{G_1,\dots,G_n})$ explicitly:
		\begin{align*}
			\dim \proj_{\R^k}(P_{G_1,\dots,G_n}) &= \dim \left\{ \textstyle \sum_{i \leq n} w_i \vec g_i \in \R^k : \vec g_i \in G_i, \vec w \in \Delta_{n-1} \right\}, \\
			&= \dim \conv \left( \textstyle \bigcup_{i \leq n} G_i \right).
		\end{align*}
		Combining the computations above with \cref{eq:cayley_proj_dim}, we see that
		\begin{align*}
			\dim W &= 1 + \dim (W \textstyle\bigcap \Delta_{n-1}), \\
			&= 1 + \dim P_{G_1,\dots,G_n} - \dim \proj_{\R^k}(P_{G_1,\dots,G_n}), \\
			&= n+\dim \textstyle \sum_{i \leq n} \aff G_i - \dim \conv \left( \textstyle \bigcup_{i \leq n} G_i \right),
		\end{align*}
		where the first line follows from $W$ being a cone in the positive orthant. Putting this expression into \cref{eq:init_U_dim}:
		\begin{align*}
			\dim U_{\CType} &= nk - n - \textstyle\sum_{i \leq n} \dim G_i + \dim W + \dim \conv (\bigcup_{i \leq n} G_i), \\
			&= nk - \textstyle \sum_{i \leq n} \dim G_i + \dim \sum_{i \leq n} \aff G_i.
		\end{align*}
		Therefore $\L(U_{\CType}) > 0$ if and only if $\textstyle \sum_{i \leq n} \dim G_i = \dim \sum_{i \leq n} \aff G_i$ as required. \\
		The final claim of \cref{thm:wwp_dim_condition} follows directly from \cref{lemma:unique}; $\dim \FMset_n =0 $ if and only if $ k = \dim \conv \left( \bigcup_{i \leq n} G_i \right)$.
	\end{proof}
	
	\cref{thm:wwp_dim_condition} is the primary theoretical result of this section, explicitly stating conditions on the face type that are needed for a Fr\'echet mean to be unique. The question of whether a sample of size $n$ has a positive probability of producing a unique Fr\'echet mean now equates to whether there is a face type $(\CType)$ that satisfies the conditions in \Cref{thm:wwp_dim_condition}. From these polytope conditions, it is quick to show that there is some finite sample threshold below which Fr\'echet means are almost surely not unique, and above which Fr\'echet means have a positive probability of uniqueness. This is proven by the following corollary.
	
	\begin{corollary}\label{cor:inductive_result}
		Suppose $U_{\CType}$ has positive probability, and $F_{n+1}$ is a facet of $B$ with $F_n \subseteq F_{n+1}$. Then $U_{F_1,\dots,F_{n+1}}$ has positive probability. In particular, if $U_{\CType}$ has positive probability and induces a unique Fr\'echet mean, then $U_{F_1,\dots,F_{n+1}}$ has positive probability and induces a unique Fr\'echet mean.
	\end{corollary}
	
	\begin{proof}
        If $F_{n+1}$ is a facet, then $G_{n+1}$ and $\aff G_{n+1}$ are a single point. Then as $G_1, \dots, G_n$ satisfy \cref{eq:wpp_cond_probable}:
		\begin{align*}
			\textstyle \sum_{i \leq n+1} \dim G_i &= \textstyle\sum_{i \leq n} \dim G_i,\\
			&= \textstyle\dim \sum_{i \leq n} \aff G_i,\\
			&= \textstyle\dim \sum_{i \leq n+1} \aff G_i. 
		\end{align*}
		Next note that as $G_{n+1}$ is a vertex of $G_n$, we have that $\conv \bigcup_{i \leq n} G_i = \conv \bigcup_{i \leq n+1} G_i$, and so:
		\begin{align*}
			\vec 0 \in \relint \left( \conv \left(\textstyle \bigcup_{i \leq n} G_i \right) \right) &\implies \vec 0 \in \relint \left( \conv \left(\textstyle \bigcup_{i \leq n+1} G_i \right) \right), \\
			\dim \conv \left(\textstyle \bigcup_{i \leq n} G_i \right) = k & \implies \dim \conv \left(\textstyle \bigcup_{i \leq n+1} G_i \right) = k.
		\end{align*}
		We conclude that if $U_{\CType}$ satisfies \cref{eq:wpp_cond_possible,eq:wpp_cond_probable}, then so does $U_{F_1,\dots,F_{n+1}}$, and if $U_{\CType}$ satisfies \cref{eq:wpp_cond_unique}, then so does $U_{F_1,\dots,F_{n+1}}$. The result then follows from \cref{thm:wwp_dim_condition}.
	\end{proof}
	
	As the probability of a unique Fr\'echet mean converges to one (\cref{thm:uniqueness_prob_convergence}), there is some finite sample size that gives a positive probability of a unique Fr\'echet mean. By \cref{cor:inductive_result}, we can then conclude that for any polytope norm, there is a finite sample threshold $N(B)$ such that a sample has a positive probability of producing a unique Fr\'echet mean if and only if $n\geq N(B)$. 
	
	\begin{definition}[Unique Fr\'echet mean sample threshold]\label{defn:sample_threshold}
		Fix a polytope norm on $\R^k$ with ball $B$. The \emph{unique Fr\'echet mean sample threshold}, $N(B)$, is the smallest $n \geq 2$ such that there is a positive probability of $\sample$ having a unique Fr\'echet mean.
	\end{definition}
	
	In the following section, we prove an upper bound for the unique Fr\'echet mean sample threshold $N(B)$ of a general polytope norm, and compute it explicitly for the $\ell_\infty$, $\ell_1$ and tropical norms.

     \FloatBarrier

    \section{Calculating Uniqueness Sample Thresholds} \label{sec:calculating_examples}

    In the previous section, we showed that there is a well-defined sample threshold at which the probability of a unique Fr\'echet mean becomes positive (\cref{cor:inductive_result}), and characterised this threshold by geometric conditions on the polytope $B$ (\cref{thm:wwp_dim_condition}). In this section, we demonstrate the theoretical power of \cref{thm:wwp_dim_condition} by establishing general upper bounds for the uniqueness sample threshold and calculating the exact sample threshold for specific examples.

    \subsection{Upper Bounds on Uniqueness Sample Thresholds}
	
	Having defined the unique Fr\'echet mean sample threshold, we will prove that for any polytope norm on $\R^k$, this threshold is at most $k+1$. As samples of size 2 can almost never produce a unique Fr\'echet mean (\cref{rmk:non_unique_for_2n}), this upper bound is tight in $\R^2$, but it is not known whether this is a tight upper bound for arbitrary dimension. We also prove an additive upper bound for polytopes whose unit ball is a Cartesian product, which often proves to be a tighter bound than our general upper bound.
    
	\begin{proposition}\label{prop:threshold_bound}
		Consider a polytope normed space $(\R^k, \|\cdot\|_B)$ with unit ball $B$. The unique Fr\'echet mean sample threshold $N(B)$ is at most $k+1$.
	\end{proposition}
	
	\begin{proof}
		It suffices to find some $G_1,\dots, G_{k+1}$, each a face of $B^{\Delta}$, which together satisfy the conditions of \cref{thm:wwp_dim_condition}. We will construct $G_1,\dots, G_{k+1}$ iteratively, using the following result:
		
		\begin{lemma}[Corollary 11.7 of Chapter 2,  \cite{brondsted2012introduction}]\label{lem:vertex_neighbours}
			Let $Q \subset \R^k$ be some $k$-dimensional polytope with a vertex $\vec x_0$. Let $\vec x_1, \dots, \vec x_r$ be the vertices adjacent to $\vec x_0$. Then
			\[
			\aff \{ \vec x_0,\dots, \vec x_r \} = \R^k.
			\]
		\end{lemma}
		We first fix $\vec v_1$ to be any vertex of the polar polytope $B^{\Delta}$. Let $G_1$ be $\vec v_1$, then we define vertices $\vec v_2, \dots, \vec v_k$ and faces $G_2, \dots, G_{k}$ iteratively as follows:
		\begin{enumerate}
			\item Using \cref{lem:vertex_neighbours}, we fix $\vec v_{i+1}$ to be some neighbour of $-\vec v_i$ which is not in $\text{span} \{\vec v_1,\dots, \vec v_i\}$.
			\item Let $G_{i+1}$ be the edge connecting $(-\vec v_i, \vec v_{i+1})$.
		\end{enumerate}
        Finally, we set $G_{k+1}$ be the vertex $-\vec v_{k}$. Note that by construction, $\vec v_1,\dots \vec v_k$ are linearly independent.
        
		Having constructed $G_1, \dots, G_{k+1}$, it remains to prove that they satisfy \cref{eq:wpp_cond_possible,eq:wpp_cond_probable,eq:wpp_cond_unique}. We can first prove \cref{eq:wpp_cond_unique} directly, as $\vec v_1, \dots \vec v_k$ are linearly independent:
		\begin{align*}
			\dim \conv \left( \textstyle\bigcup_{i \leq k+1} G_i  \right) &= \dim \aff \{ \pm \vec v_1, \dots, \pm \vec v_k \}, \\
			&= \dim \aff \{\vec 0, \vec v_1, \dots, \vec v_k \}, \\
			&= k.
		\end{align*}
		Next, we prove \cref{eq:wpp_cond_probable}. The edges $G_2, \dots, G_k$ each have dimension $1$, while $G_1, G_{k+1}$ are single vertices, so $\sum_{i \leq k+1} \dim G_i = k-1$. Then:
		\begin{align*}
			\dim \left(\textstyle\sum_{i \leq k+1} \aff G_i \right) &= \dim \left(-\vec v_{k} + \vec v_1 + \textstyle\sum_{i \leq k-1} (-\vec v_i + \text{span}(\vec v_i + \vec v_{i+1})) \right),\\
			&= \dim \text{span}(\vec v_1 + \vec v_2, \dots, \vec v_{k-1} + \vec v_{k}),\\
			&= k-1 = \textstyle \sum_{i \leq k+1} \dim G_i,
		\end{align*}
		where the final line follows the fact that $\vec v_1 + \vec v_2, \dots, \vec v_{k-1} + \vec v_{k}$ must be independent as $\vec v_1, \dots, \vec v_k$ are independent. We have proved \cref{eq:wpp_cond_probable}. 
        
        It remains to prove \cref{eq:wpp_cond_possible}, which follows from the fact that $\vec 0$ can be expressed as a positive sum of all the vertices $\pm \vec v_1, \dots, \pm \vec v_{k}$.
        
        We have shown that $G_1, \dots, G_{k+1}$ satisfy \cref{eq:wpp_cond_possible,eq:wpp_cond_probable,eq:wpp_cond_unique}, so it follows by \cref{thm:wwp_dim_condition} that the unique Fr\'echet mean sample threshold $N(B)$ is at most $k+1$.
	\end{proof}

    The upper bound of $k+1$ in \cref{prop:threshold_bound} is very general, but we can get a potentially tighter bound for norms whose unit ball is a Cartesian product; that is, $\|(\vec u,\vec v)\|_B$ is given by $\max \{ \|\vec u\|_{B_1}, \|\vec v\|_{B_2} \}$ and $B = B_1 \times B_2$. \cref{prop:C_prods_additive} shows that the uniqueness sample threshold is at worst additive with respect to Cartesian products.

    \begin{proposition}\label{prop:C_prods_additive}
        Let $\|\cdot \|_{B_1}$ and $\|\cdot \|_{B_2}$ be polytope norms on $\R^{k_1}$ and $\R^{k_2}$ respectively, then consider the norm $\|(\vec u,\vec v)\|_B = \max \{ \|\vec u\|_{B_1}, \|\vec v\|_{B_2} \}$. Then the unique Fr\'echet mean sample threshold of $B$ is at most the sum of the thresholds for $B_1$ and $B_2$; $N(B) \leq N(B_1) + N(B_2)$.
    \end{proposition}

    \begin{proof}
        Let $n_1, n_2$ be $N(B_1), N(B_2)$ respectively, and let $A_1,A_2$ be matrices such that $\|\vec u \|_{B_i} = \max \{ A_i \vec u \}$ for $i=1,2$. By \cref{thm:wwp_dim_condition}, there are faces $G_1^i,\dots,G_{n_i}^i$ of $B_i^{\triangle}$ satisfying \cref{eq:wpp_cond_possible,eq:wpp_cond_probable,eq:wpp_cond_unique} for $i = 1,2$. \\
        The norm $\|\cdot\|_B$ is given by $\|\vec x \|_B = \max \{ A\vec x \}$, where
        \[
            A = \begin{pmatrix}
                A_1 & 0 \\
                0 & A_2
            \end{pmatrix},
        \]
        so, denoting the origin of $\R^k$ by $\vec 0_k$, the polar polytope of the unit ball $B$ is given by: 
        \[
            B^{\triangle} = \conv \{B_1^{\triangle} \times \vec 0_{k_2}, \vec 0_{k_1} \times B_2^{\triangle} \}.
        \]
        From this, we can see that for any face $F^1$ of $B_1^{\triangle}$, $F^1 \times \vec 0_{k_2}$ is a face of $B^{\triangle}$, and for any face $F^2$ of $B_2^{\triangle}$, $\vec 0_{k_1} \times F^2$ is a face of $B^{\triangle}$. We will show that the faces $G_1^1 \times \vec 0_{k_2}, \dots, G_{n_1}^1 \times \vec 0_{k_2}, \vec 0_{k_1} \times G_1^2, \dots, \vec 0_{k_1} \times G_{n_2}^2$ satisfy \cref{eq:wpp_cond_possible,eq:wpp_cond_probable,eq:wpp_cond_unique}. \\
        Condition \cref{eq:wpp_cond_possible} tells us there exist $\lambda \in \R^{n_1}_{>0}, \mu \in \R^{n_2}_{>0}$ and $\vec g_j^i \in \relint G_j^i$ ($j \leq n_i, i=1,2$) such that
        \begin{align*}
            \vec 0_{k_1} = \textstyle\sum_{j \leq n_1} \lambda_j \vec g_j^1, \qquad \vec 0_{k_2} = \textstyle\sum_{j \leq n_2} \mu_j \vec g_j^2.
        \end{align*}
        Then
        \begin{align*}
            \vec 0_{k_1+k_2} = \textstyle\sum_{j \leq n_1} \lambda_j (\vec g_j^1 \times \vec 0_{k_2}) + \textstyle\sum_{j \leq n_2} \mu_j (\vec 0_{k_1} \times \vec g_j^2),
        \end{align*}
        so $\vec 0_{k_1+k_2} \in \relint \conv \{ G_1^1 \times \vec 0_{k_2}, \dots, G_{n_1}^1 \times \vec 0_{k_2}, \vec 0_{k_1} \times G_1^2, \dots, \vec 0_{k_1} \times G_{n_2}^2 \}$. \\
        We next prove that $G_1^1 \times \vec 0_{k_2}, \dots, G_{n_1}^1 \times \vec 0_{k_2}, \vec 0_{k_1} \times G_1^2, \dots, \vec 0_{k_1} \times G_{n_2}^2$ satisfy \cref{eq:wpp_cond_probable}.
        \begin{align*}
            \textstyle\sum_{j \leq n_1} \dim (G_j^{1} \times \vec 0_{k_2}) + \sum_{j \leq n_2} \dim (\vec 0_{k_1} \times G_j^2) &= \textstyle\sum_{j \leq n_1} \dim G_j^{1} + \sum_{j \leq n_2} \dim G_j^2, \\
                &= \dim \textstyle \sum_{j \leq n_1} \aff G_j^1 + \dim \sum_{j \leq n_2} \aff G_j^{2}, \\
                &= \dim \textstyle \sum_{j \leq n_1} \aff(G_j^1 \times \vec 0_{k_2}) + \dim \sum_{j \leq n_2} \aff(\vec 0_{k_1} \times G_j^{2}), \\
                &= \dim \left ( \textstyle \sum_{j \leq n_1} \aff(G_j^1 \times \vec 0_{k_2}) + \sum_{j \leq n_2} \aff(\vec 0_{k_1} \times G_j^{2}) \right),
        \end{align*}
        where the last line follows from the independence of $G_1^1 \times \vec 0_{k_2}, \dots, G_{n_1}^1 \times \vec 0_{k_2}$ and $\vec 0_{k_1} \times G_1^2, \dots, \vec 0_{k_1} \times G_{n_2}^2$. Hence \cref{eq:wpp_cond_probable} is satisfied. \\
        Finally, we prove \cref{eq:wpp_cond_unique}. Note that $\conv \bigcup_{j \leq n_1} (G_j^1 \times \vec 0_{k_2})$ is full dimensional in $\R^{k_1} \times \vec 0_{k_2}$, and $\conv \bigcup_{j \leq n_2} (\vec 0_{k_1} \times G_j^2)$ is full dimensional in $\vec 0_{k_1} \times \R^{k_2}$. The affine hull of their joint convex hull therefore contains $\R^{k_1} \times \vec 0_{k_2}$ and $\vec 0_{k_1} \times \R^{k_2}$, and so is the full space $\R^{k_1+k_2}$. Hence
        \[
            \conv \left( \textstyle \bigcup_{j \leq n_1} (G_j^1 \times \vec 0_{k_2}) \bigcup \bigcup_{j \leq n_2} (\vec 0_{k_1} \times G_j^2) \right) = k_1 + k_2.\tag*{\qedhere}
        \]
    \end{proof}

    In general, we expect it to be easier to find upper bounds than lower bounds for the unique Fr\'echet mean sample thresholds, as a lower bound would require proving non-existence. In the following section, we compute unique Fr\'echet mean sample thresholds for the $\ell_{\infty}$, $\ell_1$ and tropical norms, but the techniques used to bound their thresholds below are quite example-specific, particularly for the $\ell_{\infty}$ and tropical norms.
	
	\subsection{Uniqueness sample thresholds for the $\ell_{\infty}$, $\ell_1$ and tropical norms} \label{subsec:UST_examples}
	
	Using the conditions of \Cref{thm:wwp_dim_condition}, we can identify the uniqueness sample threshold in the case of the $\ell_{\infty}$, $\ell_1$ and tropical norms. In each case, we provide a construction which satisfies the conditions \cref{eq:wpp_cond_possible,eq:wpp_cond_probable,eq:wpp_cond_unique}, and then we use case-specific methods to show these constructions are using the fewest possible samples.

    \subsubsection{The $\ell_{\infty}$ norm}
    
    We first show that the $\ell_\infty$ norm has a unique Fr\'echet mean sample threshold of $k$ for $k\geq 3$, and a unique Fr\'echet mean sample threshold of $3$ for $k= 2$. Our face construction is used for the $k=3$ case in \cref{fig:hypercube_face_types}.
	
	\begin{proposition}\label{prop:hypercube}
		Consider the normed space $(\R^k, \ell_{\infty})$, which has unit ball $B=[-1,1]^k$. When $k \geq 3$, the sample threshold $N([-1,1]^k)$ is $k$. In the $k = 2$ case, $N([-1,1]^2) = 3$.
	\end{proposition}
	
	\begin{proof}
        We first consider the $k=2$ case. By \cref{rmk:non_unique_for_2n}, and \cref{prop:threshold_bound},  $2 < N([-1,1]^2) \leq 3$. So  $N([-1,1]^2) = 3$. \\
        
		We now show that if $G_1, \dots, G_n$ satisfy \cref{eq:wpp_cond_possible,eq:wpp_cond_probable,eq:wpp_cond_unique}, then $n \geq k$. Suppose $G_1, \dots, G_n$ satisfy \cref{eq:wpp_cond_possible,eq:wpp_cond_probable,eq:wpp_cond_unique}. As $B = [-1, 1]^k$, the polar polytope $B^{\Delta}$ is given by the cross-polytope. Let $\mathcal V$ be the vertex set of $B^{\Delta}$; that is, $\mathcal V = \{ \pm \vec e_1,\dots, \pm \vec e_k \}$. If any vertex $\vec e_j$ is not in $\bigcup_{i \leq n} G_i$, then $\conv \left( \bigcup_{i \leq n} G_i \right)$ is in the half-space $\{ \vec x: x_j \leq 0 \}$. As $k = \dim \conv \left( \bigcup_{i \leq n} G_i\right)$ by \cref{eq:wpp_cond_unique}, this would mean $\vec 0$ is not in its relative interior; contradiction. Hence $\bigcup_{i \leq n} G_i$ must contain every $\vec e_j$, and similarly every $-\vec e_j$. Hence $\mathcal V \subseteq \bigcup_{i \leq n} G_i$.

        \begin{samepage}
		Every face of $B^{\Delta}$ is a simplex, so a face $G_i$ of dimension $m$ contains $m+1$ vertices of $B^{\Delta}$. By \cref{eq:wpp_cond_probable}:
		\begin{align*}
			k &\geq \dim \textstyle\sum_{i \leq n} \aff G_i, \\
			&= \textstyle \sum_{i \leq n} \dim G_i, \\
			&= \textstyle \sum_{i \leq n} \left(|G_i \cap \mathcal V| - 1\right), \\
			&\geq \textstyle 2k - n.
		\end{align*}
		Hence $n \geq k$. \\
        \end{samepage}

        We now consider $k=3,4,5$. We define $G_1, \dots, G_k$ by
		$$
		G_1, \dots, G_k = \begin{cases} \conv (\vec e_1,-\vec e_2), \conv (\vec e_2,-\vec e_3),  \dots, \conv (\vec e_k, -\vec e_1) & k=3,5, \\
        \conv (\vec e_1,\vec e_2,-\vec e_3), \conv (\vec e_3,-\vec e_4),  \conv (\vec e_4, -\vec e_1), \{-\vec e_2 \} & k=4.
        \end{cases}
		$$
		
		In each case, $\conv \bigcup_{i \leq k} G_i = B^{\Delta}$, so has dimension $k$ and $\vec 0$ in its interior, satisfying \cref{eq:wpp_cond_possible,eq:wpp_cond_unique}. Also:
		\begin{align*}
			\dim \textstyle \sum_{i \leq k} \aff G_i &= \begin{cases}
			    \dim \text{span} \{ \vec e_1 + \vec e_2, \dots, \vec e_k + \vec e_1 \}, & k=3,5, \\
                \dim \text{span} \{ \vec e_1 + \vec e_3, \vec e_2 + \vec e_3, \vec e_3 + \vec e_4 , \vec e_4+ \vec e_1\},& k=4,
			\end{cases} \\
            &= k = \textstyle \sum_{i \leq k} \dim G_i.
		\end{align*}
        Hence, if $k=3,4,$ or $5$, we get that $N([-1,1]^k)= k$. Any $k \geq 6$ can be expressed as a sum of threes and fours, so by \cref{prop:C_prods_additive}, $N([-1,1]^k) \leq k$. We conclude that $N([-1,1]^k) = k$ for $k \geq 3$.
	\end{proof}
	
	\begin{figure*}[ht!]
		\centering
		\begin{subfigure}[b]{0.48\textwidth}
			\centering
			\tdplotsetmaincoords{70}{50}
			\begin{tikzpicture}[scale=1.33, tdplot_main_coords]
				\coordinate (O) at (0,0,0);
				
				\draw[->] (-3,0,0) -- (3,0,0) node[anchor=north east]{$x$};
				\draw[->] (0,3,0) -- (0,-3,0) node[anchor=north west]{$y$};
				\draw[->] (0,0,-2.25) -- (0,0,2.25) node[anchor=south]{$z$};
				
				\coordinate (A) at (1,1,1);
				\coordinate (B) at (1,-1,1);
				\coordinate (C) at (-1,-1,1);
				\coordinate (D) at (-1,1,1);
				\coordinate (E) at (1,1,-1);
				\coordinate (F) at (1,-1,-1);
				\coordinate (G) at (-1,-1,-1);
				\coordinate (H) at (-1,1,-1);
				
				\draw (A) -- (B) -- (C) -- (D) -- (A);
				\draw (E) -- (F) -- (G);
				\draw[dashed] (G) -- (H);
				\draw[dashed] (E) -- (H);
				\draw (A) -- (E);
				\draw (B) -- (F);
				\draw (C) -- (G);
				\draw[dashed] (D) -- (H);
				\fill[fill=gray,opacity=1/8](A) -- (B) -- (F) -- (E) -- cycle;
				\fill[fill=gray,opacity=1/8](C) -- (D) -- (H) -- (G) -- cycle;
				\fill[fill=gray,opacity=1/8](A) -- (D) -- (H) -- (E) -- cycle;
				\fill[fill=gray,opacity=1/8](C) -- (B) -- (F) -- (G) -- cycle;
				\fill[fill=gray,opacity=1/8](A) -- (B) -- (C) -- (D) -- cycle;
				\fill[fill=gray,opacity=1/8](E) -- (F) -- (G) -- (H) -- cycle;
				
				\draw[red,line width=3,line cap=round,opacity=3/4] (A) -- (E)  node[midway,anchor=west]{$F_1$};
				\draw[red,line width=3,line cap=round,opacity=3/4] (D) -- (C)  node[midway,anchor=south]{$F_3$};
				\draw[red,line width=3,line cap=round,opacity=3/4] (G) -- (F)  node[midway,anchor=north]{$F_2$};
				
			\end{tikzpicture}
			\caption{The 3-dimensional hypercube, which is the unit ball for the $\ell_{\infty}$ norm, and a face type that induces a unique Fr\'echet mean.}
		\end{subfigure}
		\hfill
		\begin{subfigure}[b]{0.48\textwidth}
			\centering
			\tdplotsetmaincoords{70}{50}
			\begin{tikzpicture}[scale = 2,tdplot_main_coords]
				\coordinate (O) at (0,0,0);
				
				\draw[->] (-2,0,0) -- (2,0,0) node[anchor=north east]{$x$};
				\draw[->] (0,2,0) -- (0,-2,0) node[anchor=north west]{$y$};
				\draw[->] (0,0,-1.5) -- (0,0,1.5) node[anchor=south]{$z$};
				
				\coordinate (A) at (1,0,0);
				\coordinate (B) at (0,1,0);
				\coordinate (C) at (0,0,1);
				\coordinate (D) at (-1,0,0);
				\coordinate (E) at (0,-1,0);
				\coordinate (F) at (0,0,-1);
				
				\draw (D) -- (E) -- (A) -- (B);
				\draw[dashed] (D) -- (B);
				\draw (D) -- (C) -- (A) -- (F);
				\draw[dashed] (D) -- (F);
				\draw (F) -- (E) -- (C) -- (B);
				\draw[dashed] (F) -- (B);
				\fill[fill=gray,opacity=1/8](A) -- (B) -- (C) -- cycle;
				\fill[fill=gray,opacity=1/8](D) -- (E) -- (F) -- cycle;
				\fill[fill=gray,opacity=1/8](A) -- (B) -- (F) -- cycle;
				\fill[fill=gray,opacity=1/8](C) -- (D) -- (E) -- cycle;
				\fill[fill=gray,opacity=1/8](A) -- (E) -- (C) -- cycle;
				\fill[fill=gray,opacity=1/8](D) -- (B) -- (F) -- cycle;
				\fill[fill=gray,opacity=1/8](A) -- (E) -- (F) -- cycle;
				\fill[fill=gray,opacity=1/8](D) -- (B) -- (C) -- cycle;
				
				\draw[red,line width=3,line cap=round,opacity=3/4] (B) -- (A)  node[midway,anchor=west]{$G_1$};
				\draw[red,line width=3,line cap=round,opacity=3/4] (D) -- (C)  node[midway,anchor=base east]{$G_3$};
				\draw[red,line width=3,line cap=round,opacity=3/4] (F) -- (E)  node[midway,anchor=east]{$G_2$};
			\end{tikzpicture}
			\caption{The polar polytope for the 3-dimensional hypercube, and the faces polar to $F_1, F_2, F_3$ which satisfy conditions \cref{eq:wpp_cond_possible,eq:wpp_cond_probable,eq:wpp_cond_unique} in \cref{thm:wwp_dim_condition}.}
		\end{subfigure}
		\caption{A face type for the $\ell_{\infty}$ unit ball in 3 dimensions that induces a unique Fr\'echet mean. In the proof of \cref{prop:hypercube}, we use a generalisation of this example to show that the unique Fr\'echet mean sample threshold for the $\ell_{\infty}$ norm in $\R^k$ is $k$ for $k\geq 3$.}
		\label{fig:hypercube_face_types}
	\end{figure*}

    We note that the unique Fr\'echet mean sample threshold for the $\ell_{\infty}$ norm gives an example where our upper bound from \cref{prop:C_prods_additive} is tight. The cartesian product of two $\ell_{\infty}$ spaces with dimension 3 or more will have a unique Fr\'echet mean sample threshold which is the sum of their individual thresholds.

    \subsubsection{The $\ell_1$ norm}
	
	As we saw previously, when using the $\ell_{\infty}$ norm, the sample size threshold is the dimension of the space, $k$. This may seem like generic behaviour, as it seems likely for Fr\'echet means to live in facet cones with respect to datapoints, so getting a unique point would require intersecting $k$ facets. However, this behaviour is far from generic --- in fact, for the $\ell_1$ norm the sample size threshold is just 3.
	
	\begin{proposition}\label{prop:cross-polytope}
		Consider the normed space $(\R^k, \ell_1)$, which has unit ball $B=\conv \{ \pm \vec e_1, \dots, \pm \vec e_k \}$. Then $N(B) = 3$.
	\end{proposition}
	
	\begin{proof}
		The polar polytope of the $\ell_1$ ball is the hypercube. Fix $G_1$ to be the face $\{ (-1, x_2, \dots, x_k): |x_i| \leq 1 \}$. Let $G_2$ be the vertex $(1,\dots, 1)$ and let $G_3$ be the edge between $ (1, -1, \dots, -1)$ and $(-1, -1, \dots, -1)$. 
		
		The face $G_1$ has dimension $k-1$ but $G_2$ is not in the affine hull of $G_1$, so their convex hull must have dimension $k$. This proves \cref{eq:wpp_cond_unique}.\\
		The affine hulls of $G_1, G_2, G_3$ are $\{ \vec x : x_1 = -1 \}, \{ (1,\dots, 1) \}$, and $\{ (\lambda, -1, \dots, -1): \lambda \in \R \}$ respectively. The sum of their affine hulls is the entire space $\R^k$ and so has dimension $k$, which is also the sum of their individual dimensions. This proves \cref{eq:wpp_cond_probable}. \\ 
		Note that $(-1, 0,\dots, 0)$ is in the relative interior of $G_1$, $(1,\dots, 1)$ is the relative interior of $G_2$, and $(0, -1, \dots, -1)$ is in the relative interior of $G_3$. We then have:
		\[
			\vec 0 = \tfrac{1}{3} (-1, 0,\dots, 0) + \tfrac{1}{3} (1,\dots, 1) + \tfrac{1}{3} (0, -1, \dots, -1).
		\] 
		Hence $\vec 0 \in \relint \conv (G_1 \cup G_2 \cup G_3)$, and we have proved \cref{eq:wpp_cond_possible}. \\		
		By \cref{rmk:non_unique_for_2n}, we have $N(B)>2$, so $N(B) = 3$.
	\end{proof}

		\begin{figure*}[ht!]
		\centering
		\begin{subfigure}[b]{0.48\textwidth}
			\centering
			\tdplotsetmaincoords{70}{40}
			\begin{tikzpicture}[scale = 2,tdplot_main_coords]
				\coordinate (O) at (0,0,0);
				
				\draw[->] (-2,0,0) -- (2,0,0) node[anchor=north east]{$x$};
				\draw[->] (0,2,0) -- (0,-2,0) node[anchor=north west]{$y$};
				\draw[->] (0,0,-1.5) -- (0,0,1.5) node[anchor=south]{$z$};
				
				\coordinate (A) at (1,0,0);
				\coordinate (B) at (0,1,0);
				\coordinate (C) at (0,0,1);
				\coordinate (D) at (-1,0,0);
				\coordinate (E) at (0,-1,0);
				\coordinate (F) at (0,0,-1);
				\draw[red,line width=3,line cap=round,opacity=3/4] (B) -- (F)  node[anchor=north west]{$F_3$};
				
				\draw (D) -- (E) -- (A) -- (B);
				\draw[dashed] (D) -- (B);
				\draw (D) -- (C) -- (A) -- (F);
				\draw[dashed] (D) -- (F);
				\draw (F) -- (E) -- (C) -- (B);
				\draw[dashed] (F) -- (B);
				\fill[fill=gray,opacity=1/8](A) -- (B) -- (C) -- cycle;
				\fill[fill=gray,opacity=1/8](D) -- (E) -- (F) -- cycle;
				\fill[fill=gray,opacity=1/8](A) -- (B) -- (F) -- cycle;
				\fill[fill=gray,opacity=1/8](C) -- (D) -- (E) -- cycle;
				\fill[fill=gray,opacity=1/8](A) -- (E) -- (C) -- cycle;
				\fill[fill=gray,opacity=1/8](D) -- (B) -- (F) -- cycle;
				\fill[fill=gray,opacity=1/8](A) -- (E) -- (F) -- cycle;
				\fill[fill=gray,opacity=1/8](D) -- (B) -- (C) -- cycle;

				\filldraw[red,opacity=3/4] (D) circle (1pt) node[anchor=east]{$F_1$};
				\fill[fill=red,opacity=3/4](A) -- (E) -- (C) node[anchor=south west,red]{$F_2$} -- cycle;
			\end{tikzpicture}
			\caption{The 3-dimensional cross-polytope, which is the unit ball for the $\ell_{1}$ norm, and a face type which induces a unique Fr\'echet mean.}
		\end{subfigure}
		\hfill
		\begin{subfigure}[b]{0.48\textwidth}
			\centering
			\tdplotsetmaincoords{70}{40}
			\begin{tikzpicture}[scale=1.33, tdplot_main_coords]
				\coordinate (O) at (0,0,0);
				
				\draw[->] (-3,0,0) -- (3,0,0) node[anchor=north east]{$x$};
				\draw[->] (0,3,0) -- (0,-3,0) node[anchor=north west]{$y$};
				\draw[->] (0,0,-2.25) -- (0,0,2.25) node[anchor=south]{$z$};
				
				\coordinate (A) at (1,1,1);
				\coordinate (B) at (1,-1,1);
				\coordinate (C) at (-1,-1,1);
				\coordinate (D) at (-1,1,1);
				\coordinate (E) at (1,1,-1);
				\coordinate (F) at (1,-1,-1);
				\coordinate (G) at (-1,-1,-1);
				\coordinate (H) at (-1,1,-1);
				
				\draw (A) -- (B) -- (C) -- (D) -- (A);
				\draw (E) -- (F) -- (G);
				\draw[dashed] (G) -- (H);
				\draw[dashed] (E) -- (H);
				\draw (A) -- (E);
				\draw (B) -- (F);
				\draw (C) -- (G);
				\draw[dashed] (D) -- (H);
				\fill[fill=gray,opacity=1/8](A) -- (B) -- (F) -- (E) -- cycle;
				\fill[fill=gray,opacity=1/8](C) -- (D) -- (H) -- (G) -- cycle;
				\fill[fill=gray,opacity=1/8](A) -- (D) -- (H) -- (E) -- cycle;
				\fill[fill=gray,opacity=1/8](C) -- (B) -- (F) -- (G) -- cycle;
				\fill[fill=gray,opacity=1/8](A) -- (B) -- (C) -- (D) -- cycle;
				\fill[fill=gray,opacity=1/8](E) -- (F) -- (G) -- (H) -- cycle;
				
				\filldraw[color=red,fill=red,opacity=3/4] (B) circle (1pt) node[anchor=west,red]{$G_2$};
				\fill[fill=red,opacity=3/4](C) node[anchor=south east,red]{$G_1$} -- (D) -- (H) -- (G) -- cycle;
				\draw[red,line width=3,line cap=round,opacity=3/4] (E) -- (H)  node[midway,anchor=north]{$G_3$};
			\end{tikzpicture}
			\caption{The polar polytope for the 3-dimensional cross-polytope, and the faces polar to $F_1, F_2, F_3$ which satisfy conditions \cref{eq:wpp_cond_possible,eq:wpp_cond_probable,eq:wpp_cond_unique} in \cref{thm:wwp_dim_condition}.}
		\end{subfigure}
		\caption{A face type for the $\ell_{1}$ unit ball in 3 dimensions that induces a unique Fr\'echet mean. In the proof of \cref{prop:cross-polytope}, we use a generalisation of this example to show that the unique Fr\'echet mean sample threshold for the $\ell_{1}$ norm in $\R^k$ is $3$.}
		\label{fig:cross-polytope_face_types}
	\end{figure*}

    \cref{fig:cross-polytope_face_types} shows our construction of a 3-sample face type that occurs with positive probability and yields a unique Fr\'echet mean in the case of the $\ell_1$ norm in 3 dimensions.

    In our proof of \cref{prop:cross-polytope}, we selected $G_1, G_2,$ and $G_3$ such that $G_1$ is a facet, $G_2$ is a vertex of $-G_1$, and $G_3$ is an edge with one vertex not in $-G_1$ and one vertex in $-G_1$ which does not share a co-dimension 2 face with $G_2$. We can replicate this construction for any polytope whose polar $B^{\Delta}$ has a facet $G$ with vertices $\vec u$ and $\vec v$ whose mid-point is in relative interior of $G$. In particular, this construction generalises to any norm whose polar has zonotope facets, as every vertex of a zonotope has an antipodal vertex.  	
	
    \subsubsection{The tropical norm}

    The unique Fr\'echet mean sample threshold for the tropical norm is heavily motivated by phylogenetic data analysis, as phylogenetic tree data can be treated as a point in the tropical projective torus (\cref{ex:tropical-ball_representation}). For this norm, calculation of the sample threshold is much more complicated than for the $\ell_1$ or $\ell_{\infty}$ examples --- we require highly case-specific methods to bound the sample threshold below, and an upper bound on the sample threshold requires slightly different constructions depending on the parity of $k$. The resulting threshold in this case is $\lfloor (k+3)/2 \rfloor$. This is a particularly interesting result when compared to our $\ell_{\infty}$ result; both the tropical and $\ell_{\infty}$ unit balls are zonotopes, but with $\binom{k+1}{2}$ and $2k$ facets respectively, and their thresholds are both linear but on the order of $k/2$ and $k$ respectively.

    \begin{proposition}\label{prop:tropical_ball}
		Consider the normed space $(\R^k, \| \cdot \|_{\tr})$, which has unit ball $B= \{ \vec x : |x_i| \leq 1, |x_i-x_j| \leq 1 \}$. Then for $k \geq 3$, we have $N(B) = \lfloor (k+3)/2 \rfloor$. If $k = 2$, then $N(B) = 3$.
	\end{proposition}
    \begin{proof}
        First, we resolve the $k=2$ case. By \cref{rmk:non_unique_for_2n}, $N(B) > 2$, but \cref{prop:threshold_bound} tells us $N(B) \leq 3$. Hence  $N(B) = 3$. Throughout the rest of the proof, we assume $k \geq 3$.
        
        Suppose $G_1, \dots, G_n$ satisfy \cref{eq:wpp_cond_possible,eq:wpp_cond_probable,eq:wpp_cond_unique}. Then $\conv \bigcup_{l \leq n} G_l$ is full dimensional with $\vec 0$ in its interior; there is no half space containing all $G_l$ with $\vec 0$ on the boundary. Furthermore, letting $L_l$ be the linear space parallel to $\aff G_l$, \cref{eq:wpp_cond_probable} tells us that $L_1,\dots, L_n$ are linearly independent. \\
        Let $\vec e_0= \vec 0$, so the tropical norm can be written as $\|\vec x\|_{\tr} = \max_{0\leq i\leq k} \{ \vec e_i^{\top} \vec x\} - \min_{0\leq i\leq k} \{ \vec e_i^{\top} \vec x\}$. The faces of $B$ are then characterised by which indices achieve this max and min. For $F_l$, let $I_l$ be the set of indices which achieve the maximum on $F_l$, and let $J_l$ be the set of indices which achieve the minimum on $F_l$. Note $I_l, J_l$ must be disjoint as $F_l$ is a proper face of $B$. Then $G_l, \aff G_l$ and $L_l$ are given by:
        \begin{align*}
            G_l &= \conv \{ \vec e_i - \vec e_j : i \in I_l, j \in J_l \}, \\
            \aff G_l &= \{ \vec x : \textstyle\sum_{i \in I_l} x_i = 1, \textstyle\sum_{j \in J_l} x_j = -1, \text{ and for }i \notin I_l \cup J_l, x_i = 0   \}, \\
            L_l &= \text{span} \left( \{ \vec e_a - \vec e_b : a,b \in I_l \} \cup \{ \vec e_a - \vec e_b : a,b \in J_l \} \right).
        \end{align*}
        For all $i \leq k$ there must be some $G_l$ and $j \in \{0, \dots, k \}$ such that $\vec e_i - \vec e_j \in G_l$. Otherwise $\conv \cup_i G_i$ is contained in the half space $\{ \vec x : x_i \leq 0 \}$. Also, there must be some $G_l$ and $j$ such that $\vec e_0-\vec e_j \in G_l$, otherwise $\conv \cup_i G_i$ is contained in the half-space $\{ \vec x:\sum_i x_i \geq 0\}$. Similarly, for all $i \in \{0, \dots, k\}$ there must be some polar face $G_l$ such that for some $j \in \{0, \dots, k\}$, $\vec e_j - \vec e_i \in G_l$. Therefore $\cup_{l} I_l = \cup_l J_l = \{0, \dots, k \}$. \\
        
        \textbf{Lower bound on $N(B):$} Suppose $2n \leq k+1$. We define a graph $\mathcal G = (\mathcal V, \mathcal E)$ with $2n+k+1$ vertices which we denote by $\mathcal V = \{ v^I_1, \dots, v^I_n, v^J_1, \dots, v^J_n, u_0, \dots, u_k\}$. We refer to the vertices $v^I_1, \dots, v^I_n$ as the $I$-vertices, $v^J_1, \dots, v^J_n$ as the $J$-vertices, and $u_0,\dots, u_k$ as the dimension vertices. The edges $\mathcal E$ of $\mathcal G$ are given by:
        \begin{align*}
            \mathcal E = \{ (u_i,v_l^I): i \in I_l \} \cup \{ (u_j, v_l^J): i \in J_l \}.
        \end{align*}
        This is a bipartite graph between $I/J$ vertices and dimension vertices. As $\cup I_l = \cup J_l = \{ 0 ,\dots, k \}$, each dimension vertex has degree at least 2. Hence this graph has at least $2k+2$ edges, and at most $2k+2$ vertices as $2n \leq k+1$, and therefore contains a cycle, which we denote $C$. \\
        Let $u_{i_1}, \dots, u_{i_S}$ be the dimension vertices in $C$ in order and let $l_1, \dots, l_S$ be the indices of the $I/J$ vertices of $C$ in order, such that for $s \leq S$: $u_{i_s}, u_{i_{s+1}}$ both neighbour $v^I_{l_s}$ or both neighbour $v^J_{l_s}$. Similarly, $u_{i_1}, u_{i_S}$ both neighbour $v^I_{l_S}$ or both neighbour $v^J_{l_S}$. Then we have that:
        \[
            \vec e_{i_s} - \vec e_{i_{s+1}} \in L_{l_s}\text{ and }\vec e_{i_S} - \vec e_{i_1} \in L_{l_S}.
        \]
        But we have that:
        \[
            (\vec e_{i_1} - \vec e_{i_{2}}) + \dots + (\vec e_{i_{S-1}} - \vec e_{i_{S}}) + (\vec e_{i_S} - \vec e_{i_{1}}) = \vec 0. 
        \]
        Note that $L_{l_1}, \dots, L_{l_S}$ cannot all be the same; if they were, then $u_{i_1}$ would have to be connected to $v^I_l$ and $v^J_l$ for some $l$, which contradicts the disjointness of $I_l, J_l$. The equality above therefore contradicts the linear independence of $L_{l_1}, \dots, L_{l_S}$, and so $G_1, \dots, G_n$ cannot satisfy \cref{eq:wpp_cond_probable}. Hence $N(B) > (k+1)/2$.\\
                
        \textbf{Even dimension case:} We show for even $k$, $N(B) = (k+2)/2$, so set $n$ such that $k = 2n-2$. As $k \geq 3$, we have that $n \geq 3$. We set:
        \begin{align*}
            G_1 &= \conv \{ \vec e_0 - \vec e_{2n-3}, \vec e_0 - \vec e_{2n-2}, \vec e_1 - \vec e_{2n-3}, \vec e_1 - \vec e_{2n-2} \}, \\
            G_2 &= \conv \{ \vec e_2 - \vec e_0, \vec e_3 - \vec e_0 \}, \\
            G_l &= \conv \{ \vec e_{2l-2} - \vec e_{2l-4}, \vec e_{2l-2} - \vec e_{2l-5}, \vec e_{2l-1} - \vec e_{2l-4}, \vec e_{2l-1} - \vec e_{2l-5} \} \quad \text{ for }2<l<n\text{ if }n \geq 4,  \\
            G_n &= \conv \{ \vec e_{2n-2} - \vec e_{2n-4}, \vec e_{2n-2} - \vec e_{2n-5} \}.
        \end{align*}
        To show that $G_1, \dots, G_n$ satisfy \cref{eq:wpp_cond_possible}, we prove $\vec 0$ is in the positive span of their relative interiors. We set:
        \begin{align*}
            \vec g_1 \coloneqq \tfrac{1}{3}(2\vec e_0 + \vec e_1 - \vec e_{2n-3} - 2\vec e_{2n-2}) &\in \relint G_1, \\
            \vec g_2 \coloneqq\tfrac{1}{2}(\vec e_2 + \vec e_3 - 2\vec e_0) &\in \relint G_2, \\
            \vec g_l \coloneqq\tfrac{1}{2}(\vec e_{2l-2} + \vec e_{2l-1} - \vec e_{2l-4} - \vec e_{2l-5}) &\in \relint G_l \quad \text{for }2<l<n\text{ if }n \geq 4,\\
            \vec g_n \coloneqq\tfrac{1}{2}(2\vec e_{2n-2} - \vec e_{2n-4} - \vec e_{2n-5}) &\in \relint G_n. 
        \end{align*}
        For $n = 3$ or $4$, we have:
        \begin{align*}
            3\vec g_1 + 2\textstyle\sum_{l=2}^{n}\vec g_l
                    &= \begin{cases}
                        2\vec e_0 + \vec e_1 - \vec e_5 - 2 \vec e_6 + \vec e_2 + \vec e_3 - 2 \vec e_0 + \vec e_{4}+\vec e_{5}-\vec e_1-\vec e_2 + 2 \vec e_6 - \vec e_4 - \vec e_3 & \text{ if }n=4, \\
                        2\vec e_0 + \vec e_1 - \vec e_3 - 2 \vec e_4 + \vec e_2 + \vec e_3 - 2 \vec e_0 + 2 \vec e_4 - \vec e_2 - \vec e_1 & \text{ if }n= 3,
                    \end{cases}\\
                    &= \vec 0.
        \end{align*}
        When $n\geq 5$, we get that:
        \begin{align*}
            3\vec g_1 + 2\textstyle\sum_{l=2}^{n}\vec g_l 
                    &= 3\vec g_1 + 2\vec g_2 + 2\vec g_n + \textstyle\sum_{l = 3}^{n-1} \left(\vec e_{2l-2} + \vec e_{2l-1} - \vec e_{2l-4} - \vec e_{2l-5}\right) \\
                    &= 3\vec g_1 + 2\vec g_2 + 2\vec g_n +(\vec e_{2n-5}+\vec e_{2n-4}+\vec e_{2n-3}-\vec e_1-\vec e_2 - \vec e_3) \\
                    &= \vec 0.
        \end{align*}
        Hence $G_1,\dots,G_n$ satisfy \cref{eq:wpp_cond_possible}. Now we prove \cref{eq:wpp_cond_probable}. The linear subspaces $L_1,\dots, L_n$ are given by:
        \begin{align*}
            L_1 &= \text{span}\{ \vec e_{1} - \vec e_{0}, \vec e_{2n-2} - \vec e_{2n-3} \}, \\
            L_2 &= \text{span}\{ \vec e_2 - \vec e_3 \}, \\
            L_l &= \text{span}\{ \vec e_{2l-1} - \vec e_{2l-2}, \vec e_{2l-4} - \vec e_{2l-5} \}\quad \text{ for }2<l<n \text{ if }n\geq 4, \\
            L_n &= \text{span}\{ \vec e_{2n-4} - \vec e_{2n-5} \}.
        \end{align*}
        Note that $L_1$ and $L_2$ are linearly independent. Suppose there is some $l$ with $n>l>2$ and $L_l$ is not linearly independent of $L_1 + \dots + L_{l-1}$; then there are $\mu,\lambda$ such that
        \begin{align*}
            \vec 0 \neq \mu(\vec e_{2l-1} - \vec e_{2l-2}) + \lambda (\vec e_{2l-4} - \vec e_{2l-5}) &= \gamma_{1,1}(\vec e_1 - \vec e_0) + \gamma_{1,2}(\vec e_{2n-2} - \vec e_{2n-3}) + \gamma_2(\vec e_2 - \vec e_3), \\
            &\qquad + \textstyle \sum_{i=3}^{l-1} \gamma_{i,1}(\vec e_{2i-1} - \vec e_{2i-2}) + \gamma_{i,2} (\vec e_{2i-4} - \vec e_{2i-5}).
        \end{align*}
        In the case where $l=3$ and there is no summand, considering the coefficients of $\vec e_{2}$ and $\vec e_3$ gives $\lambda = \gamma_2 = 0$. In the case where $l>3$, considering the coefficients of $\vec e_{2l-4}$ and $\vec e_{2l-3}$ gives $\lambda = \gamma_{l-1,1} = 0$. For any $l$, the $\vec e_{2l-2}$ coefficients give that $\mu = 0$. Hence for all $2<l<n$, $\mu = \lambda = 0$; contradiction. Now assume $L_n$ is not linearly independent of $L_1,\dots, L_{l-1}$, such that
        \begin{align*}
            \vec e_{2n-4} - \vec e_{2n-5} &= \gamma_{1,1}(\vec e_1 - \vec e_0) + \gamma_{1,2}(\vec e_{2n-2} - \vec e_{2n-3}) + \gamma_2(\vec e_2 - \vec e_3) \\
            &\qquad + \textstyle \sum_{i=3}^{n-1} \gamma_{i,1}(\vec e_{2l-1} - \vec e_{2l-2}) + \gamma_{i,2} (\vec e_{2l-4} - \vec e_{2l-5})
        \end{align*}
        If $n\geq 4$, then looking at the $\vec e_{2n-2}$, $\vec e_{2n-3}$ and $\vec e_{2n-4}$ coefficients, we get $0 = \gamma_{1,2} = \gamma_{n-1,1} = 1$; contradiction. If $n=3$ and there is no summand, the $\vec e_{4}$, $\vec e_{3}$ and $\vec e_{2}$ coefficients tell us $0 = \gamma_{1,2} = \gamma_2 = 1$; contradiction. Hence all $L_1,\dots,L_n$ are linearly independent, and $G_1,\dots,G_n$ satisfy \cref{eq:wpp_cond_probable}. 
        
        It remains to prove \cref{eq:wpp_cond_unique}. Consider the linear space $L$ parallel to $\aff \bigcup_{i \leq n} G_i$. As $\aff G_i \subseteq \aff \bigcup_{i \leq n} G_i$, we have that for each $i$, $L_i \subseteq L$. We have shown that $L_1, \dots, L_n$ are linearly independent, and so:
        \begin{align*}
            \dim \conv \textstyle\bigcup_{i \leq n} G_i &= \dim L \\
            &\geq \textstyle\sum_{i \leq n} \dim L_i \\
            &= 2n-2 \\
            &= k.
        \end{align*}
        Hence $\dim \conv \bigcup_{i \leq n} G_i = k$, and \cref{eq:wpp_cond_unique} is satisfied.
        
        \textbf{Odd dimension case:} Now suppose $k$ is odd, and set $n$ such that $k = 2n - 3$. As $k \geq 3$, we have that $n \geq 3$. Similar to before, we set: 
        \begin{align*}
            G_1 &= \conv \{ \vec e_0 - \vec e_{2n-3}, \vec e_1 - \vec e_{2n-3} \}, \\
            G_2 &= \conv \{ \vec e_2 - \vec e_0, \vec e_3 - \vec e_0 \}, \\
            G_l &= \conv \{ \vec e_{2l-2} - \vec e_{2l-4}, \vec e_{2l-2} - \vec e_{2l-5}, \vec e_{2l-1} - \vec e_{2l-4}, \vec e_{2l-1} - \vec e_{2l-5} \} \quad \text{ for }2<l<n\text{ if }n \geq 4,  \\
            G_n &= \conv \{ \vec e_{2n-3} - \vec e_{2n-4}, \vec e_{2n-3} - \vec e_{2n-5} \}.
        \end{align*}
        This is very similar to our previous construction, with slight changes just to $G_1$ and $G_n$. \\
        To show that $G_1, \dots, G_n$ satisfy \cref{eq:wpp_cond_possible}, we now consider:
        \begin{align*}
            \vec g_1 \coloneqq \tfrac{1}{3}(2\vec e_0 + \vec e_1 - 3\vec e_{2n-3}) &\in \relint G_1, \\
            \vec g_2 \coloneqq\tfrac{1}{2}(\vec e_2 + \vec e_3 - 2\vec e_0) &\in \relint G_2, \\
            \vec g_l \coloneqq\tfrac{1}{2}(\vec e_{2l-2} + \vec e_{2l-1} - \vec e_{2l-4} - \vec e_{2l-5}) &\in \relint G_l \quad \text{for }2<l<n\text{ if }n \geq 4,\\
            \vec g_n \coloneqq\tfrac{1}{2}(2\vec e_{2n-3} - \vec e_{2n-4} - \vec e_{2n-5}) &\in \relint G_n. 
        \end{align*}
        By following the same working as before, we see that $\vec 0 = 3 \vec g_1 + 2 \sum_{l=2}^n \vec g_l$, so $G_1, \dots, G_n$ satisfy \cref{eq:wpp_cond_possible}.
        The linear subspaces $L_1, \dots, L_n$ are given by
        \begin{align*}
            L_1 &= \text{span}\{ \vec e_{1} - \vec e_{0}\}, \\
            L_2 &= \text{span}\{ \vec e_2 - \vec e_3 \}, \\
            L_l &= \text{span}\{ \vec e_{2l-1} - \vec e_{2l-2}, \vec e_{2l-4} - \vec e_{2l-5} \}\quad \text{ for }2<l<n \text{ if }n\geq 4, \\
            L_n &= \text{span}\{ \vec e_{2n-4} - \vec e_{2n-5} \}.
        \end{align*}
        These are each contained in the corresponding linear subspace from the previous construction, so linear independence is inherited from the previous construction. Therefore $G_1, \dots, G_n$ satisfy \cref{eq:wpp_cond_probable}.
        
        It remains to prove \cref{eq:wpp_cond_unique}. As before, we consider the linear space $L$ parallel to $\aff \bigcup_{i \leq n} G_i$, and get that:
        \begin{align*}
            \dim \conv \textstyle\bigcup_{i \leq n} G_i &= \dim L \\
            &\geq \textstyle\sum_{i \leq n} \dim L_i \\
            &= 2n-3 \\
            &= k.
        \end{align*}
        Hence $\dim \conv \bigcup_{i \leq n} G_i = k$, and \cref{eq:wpp_cond_unique} is satisfied.
        
        Combining the odd and even case, we have shown that for $k \geq 3$, $N(B) = \lfloor (k+3)/2\rfloor$.
    \end{proof}

		\begin{figure*}[ht!]
		\centering
        \hfill
		\begin{subfigure}[b]{0.48\textwidth}
			\centering
			\tdplotsetmaincoords{70}{40}
			\begin{tikzpicture}[scale = 1.2,tdplot_main_coords]
				\coordinate (O) at (0,0,0);
				\draw[->] (-2,0,0) -- (2,0,0) node[anchor=north east]{$x$};
				\draw[->] (0,-2.5,0) -- (0,2.5,0) node[anchor=north west]{$y$};
				\draw[->] (0,0,-1.5) -- (0,0,1.5) node[anchor=south]{$z$};
				\coordinate (A) at (1,0,0);
				\coordinate (B) at (0,1,0);
				\coordinate (C) at (0,0,1);
				\coordinate (D) at (1,1,0);
				\coordinate (E) at (0,1,1);
				\coordinate (F) at (1,0,1);
                \coordinate (G) at (1,1,1);
				\coordinate (H) at (-1,0,0);
				\coordinate (I) at (0,-1,0);
				\coordinate (J) at (0,0,-1);
				\coordinate (K) at (-1,-1,0);
				\coordinate (L) at (0,-1,-1);
				\coordinate (M) at (-1,0,-1);
                \coordinate (N) at (-1,-1,-1);
                \fill[fill=gray,opacity=1/8](D) -- (G) -- (E) -- (C) -- (K) -- (N) -- (L) -- (J) -- cycle;
				\fill[fill=gray,opacity=1/8](D) -- (G) -- (E) -- (C) -- (K) -- (N) -- (L) -- (J) -- cycle;
                \draw (A) -- (L);
                \draw (C) -- (K);
                \draw[dashed] (B) -- (M);
                \draw (D) -- (J);
                \draw[dashed] (E) -- (H);
                \draw (F) -- (I);
				\draw (A) -- (D) -- (G) -- (F) -- cycle;
                \draw (C) -- (E) -- (G) -- (F) -- cycle;
				\draw[dashed] (E) -- (B) -- (D);
                \draw (I) -- (K) -- (N) -- (L) -- cycle;
                \draw[dashed] (K) -- (H) -- (M) -- (N);
				\draw[dashed] (M) -- (J);
                \draw (J) -- (L);
				\draw[red,line width=3,line cap=round,opacity=3/4] (J) -- (L)  node[midway,anchor=north]{$F_1$};
                \draw[red,line width=3,line cap=round,opacity=3/4] (E) -- (G)  node[midway,anchor=south]{$F_2$};
                \draw[red,line width=3,line cap=round,opacity=3/4] (C) -- (K)  node[midway,anchor=south]{$F_3$};
			\end{tikzpicture}
			\caption{The unit ball for the tropical norm, and a face type which induces a unique Fr\'echet mean.}
		\end{subfigure}
		\hfill
		\begin{subfigure}[b]{0.51\textwidth}
			\centering
			\tdplotsetmaincoords{70}{40}
			\begin{tikzpicture}[scale=1.2, tdplot_main_coords]
				\coordinate (O) at (0,0,0);
				
				\draw[->] (-2,0,0) -- (2,0,0) node[anchor=north east]{$x$};
				\draw[->] (0,-2.5,0) -- (0,2.5,0) node[anchor=north west]{$y$};
				\draw[->] (0,0,-1.5) -- (0,0,1.5) node[anchor=south]{$z$};
				
				\coordinate (A) at (1,0,0);
				\coordinate (B) at (0,1,0);
				\coordinate (C) at (0,0,1);				
                \coordinate (D) at (-1,0,0);
				\coordinate (E) at (0,-1,0);
				\coordinate (F) at (0,0,-1);
				\coordinate (G) at (1,-1,0);
				\coordinate (H) at (0,1,-1);
				\coordinate (I) at (-1,0,1);				
                \coordinate (J) at (1,0,-1);
				\coordinate (K) at (-1,1,0);
				\coordinate (L) at (0,-1,1);
                \fill[fill=gray,opacity=1/8](A) -- (B) -- (C) -- (I) -- (D) -- (E) -- (F) -- (J) -- cycle;
				\fill[fill=gray,opacity=1/8](A) -- (B) -- (C) -- (I) -- (D) -- (E) -- (F) -- (J) -- cycle;
                \draw (A) -- (B) -- (C) -- (A);
                \draw (D) -- (E) -- (F);
                \draw[dashed] (D) -- (F);
                \draw (A) -- (G) -- (J) -- cycle;
                \draw[dashed] (B) -- (K) -- (H) -- cycle;
                \draw (C) -- (I) -- (L) -- cycle;
                \draw[dashed] (D) -- (K) -- (I);
                \draw (D) -- (I);
                \draw (E) -- (G) -- (L) -- cycle;
                \draw[dashed] (F) -- (H) -- (J);
                \draw (F) -- (J);

                \draw[red,line width=3,line cap=round,opacity=3/4] (F) -- (J)  node[midway,anchor=north]{$G_1$};
                \draw[red,line width=3,line cap=round,opacity=3/4] (B) -- (C)  node[midway,anchor=west]{$G_2$};
                \draw[red,line width=3,line cap=round,opacity=3/4] (I) -- (L)  node[midway,anchor=east]{$G_3$};
			\end{tikzpicture}
			\caption{The polar polytope for the 3-dimensional tropical unit ball and the faces polar to $F_1, F_2, F_3$.}
		\end{subfigure}
		\caption{A face type for the tropical unit ball in 3 dimensions that induces a unique Fr\'echet mean. In the proof of \cref{prop:tropical_ball}, we use a generalisation of this example to show that the unique Fr\'echet mean sample threshold for the tropical norm in $\R^k$ is $(k+3)/2$ for odd $k$.}
		\label{fig:tropical_ball_face_types}
	\end{figure*}

    Our construction for the tropical sample threshold of $(k+3)/2$ for odd $k$ is shown in \cref{fig:tropical_ball_face_types}. As $k$ increases, we add suitable 2-dimensional faces to our construction.
    
    The unique Fr\'echet mean sample thresholds considered in this section demonstrate near extreme behaviour. For the $\ell_1$ norm, the threshold is just 3 in any dimension --- \cref{rmk:non_unique_for_2n} tells us this is the smallest possible threshold for any norm. For the $\ell_{\infty}$ norm, the threshold is $k$, which is one away from our universal upper bound of $k+1$. We also note that the $\ell_{\infty}$ and tropical norms share similar geometry, with polynomial facet counts, and both have uniqueness thresholds which are $O(k)$. In contrast, the $\ell_1$ norm has exponential facet count and constant uniqueness threshold.
    
     \FloatBarrier

    \section{Exact Computation of Fr\'echet Means} \label{sec:computation}

    We conclude this paper by presenting an algorithm to compute the exact Fr\'echet mean set and its dimension, which is based on Wolfe's combinatorial algorithm \citep{wolfe1976finding} --- an exact method to find the point in a polytope of minimal Euclidean norm. Using this algorithm, we compute the dimension of Fr\'echet mean sets for Gaussian samples, showing that empirical results agree with our computed uniqueness thresholds.

    We require an exact representation of the Fr\'echet mean set to calculate its dimension, so we need to perform our calculations on rational data to avoid overflow errors. We sample our data from a Gaussian distribution to satisfy our assumption that our data has positive density everywhere, then rationalize our data by converting base 2 floating points to the exact corresponding fraction. We use the Julia packages JuMP \citep{dunning2017jump} and Polyhedra \citep{legat2023polyhedral} to handle these intensive rational computations.

    In this section, we outline how Wolfe's combinatorial algorithm \citep{wolfe1976finding} can be used to compute the Fr\'echet mean set induced by a polytope norm. To utilise Wolfe's method, we translate Fr\'echet means from an unconstrained minimum of a non-smooth objective function to the minimum of a quadratic form over a polyhedron (\Cref{lem:constrained_quadratic_optimisation}). The full method for computing the dimension of the Fr\'echet mean set via Wolfe's algorithm is presented in \Cref{alg:exact_computation}. In this section, we consider our data to be fixed rather than random, and so our sample is denoted $\vec x_1, \dots, \vec x_n$.
	\begin{lemma}\label{lem:constrained_quadratic_optimisation}
        For fixed data $\vec x_1,\dots, \vec x_n$ and polytope norm $\|\vec x\|_B = \max_i (A \vec x)_i$, let the polyhedron $Q$ be $$Q = \{(\vec d, \bm \theta) : \forall i \leq n \text{ and } \vec a \in A , \, \vec a^\T(\vec x_i -  \bm \theta) \leq d_i \} \subset \R^{n+k}.$$
		A point $ \bm \theta$ is in $ \FMset_n(\vec x_1,\dots, \vec x_n)$ if and only if there is some $\vec d \in \mb{R}^n$ such that $\vec d, \bm \theta$ give a solution to the following constrained quadratic optimisation problem: 
		\begin{align} \label{eq:FM_objective}
			\text{minimise} \quad\|\vec d\|^2_2,
            \text{given} \quad (\vec d, \bm \theta) \in Q.
		\end{align}    
	\end{lemma}
	\begin{proof}
		Notice that any $(\vec d,  \bm \theta)$ in $Q$ which minimises \Cref{eq:FM_objective} will be of the form $( (d_B( \bm \theta,\vec x_i))_{i \leq n}, \bm \theta)$, as otherwise there is some $d_i$ we can decrease. Hence minimisation of $\|\vec d\|_2^2$ over $Q$ is equivalent to
		\begin{align*}
			\text{minimise}\quad \|\vec d\|_2^2, \qquad
			\text{given} \quad \forall \;  i \leq n, d_i = \|\vec x_i -  \bm \theta\|_B.
		\end{align*}
		This is exactly the optimisation problem of sample Fr\'echet means.
	\end{proof}

    This problem is almost the problem of minimizing Euclidean norm over a polytope, which can be solved exactly using Wolfe's combinatorial algorithm \citep{wolfe1976finding}. With minimal adaptation, we can use Wolfe's combinatorial algorithm to solve \Cref{eq:FM_objective}, where the semi-norm $\|\vec d\|_2$ is minimised. We note that through testing, we found it is far more efficient to perform Wolfe's algorithm in $\R^{n+k}$ with respect to a semi-norm than to compute a projection of $Q$ onto $\R^n$ and perform the original Wolfe's algorithm in $\R^n$ with respect to the Euclidean norm. 
    
    Wolfe's algorithm is built on two observations. Firstly, a point $\vec x$ in a polyhedron $Q \subset \R^{D}$ has minimal Euclidean norm if and only if it is a solution to the linear program $\left(\min_{\vec v \in Q} \vec x^{\top} \vec v\right)$, that is: $\vec x^{\top}\vec x = \min_{\vec v \in Q} \vec x^{\top} \vec v$. Secondly, such an optimum is a linear combination of at most $D+1$ vertices of $Q$ (Carath\'eodorey's Theorem). Starting from an initial vertex set $J = \{ \vec v \}$, Wolfe's algorithm constructs a set of at most $D+1$ vertices whose convex hull contains the optimum by iteratively:
    \begin{enumerate}
        \item finding the point $\vec x \in \conv J$ with minimal norm and removing any unnecessary vertices from $J$, 
        \item adding a vertex $\vec v'$ to $J$ if $\vec v' = \argmin_{\vec v \in Q} \vec x^{\top}\vec v$.
    \end{enumerate}
    This process stops once $\vec x^{\top} \vec x = \min_{\vec v \in Q} \vec x^{\top} \vec v$ is satisfied. We adapt this algorithm to problem \Cref{eq:FM_objective} by finding a point $\vec x \in \conv J$ with minimal semi-norm in step 1, and adding $\vec v'$ to $J$ if it solves the linear program $\vec v' = \argmin_{\vec v \in Q} ((\vec x_{[n]}, \vec 0_k)^{\top}\vec v)$ in step 2. The optimal semi-norm value over $\conv J$ decreases over each change to $J$, and there are finitely many possible vertex sets $J$ can take, so this algorithm must terminate with an exact solution \citep{de2018minimum, wolfe1976finding}. \Cref{alg:exact_computation} shows the full version of Wolfe's algorithm which is used to compute the Fr\'echet mean set.
	
\begin{algorithm}[ht]
\caption{Wolfe's Algorithm for Fr\'echet Mean Set Computation}\label{alg:exact_computation}
\begin{algorithmic}
\Require normal vector matrix $A$, data matrix $\vec x$.
\State $Q \gets \{(\vec d,  \bm \theta) : \forall i \leq n \text{ and } \vec a \in A , \, \vec a^\T \vec x_i - d_i \leq \vec a^\T  \bm \theta \}$\Comment{Feasible region for $(\vec d, \bm \theta)$ in $\mathbb{R}^{n+k}$}
\State $\vec p_1 \gets$ some initial vertex of $Q$ 
\State $S \gets (\vec p_1)$, matrix of computed vertices
\State $t = 1$, counter for number of computed vertices
\State $J = \{ 1 \}$, set of active vertices in $\vec x$ 
\State $\lambda_1 \gets 1$, coefficient of $\vec p_1$ in $\vec x$
\State $\vec x \gets S_J \vec \lambda_J$, optimal convex combination of current active vertices
\State $\vec c \gets (\vec x_{[n]}, \vec 0_k),$ gradient of semi-norm at $\vec x$ 
\State $\vec p_{t} \gets \argmin_{\vec v \in Q} \vec c^{\top} \vec v$, compute new vertex by solving LP 
\State $t \gets t+1$,
\While{$\vec c^{\top} \vec p_{t} < \vec c^{\top} \vec x$} \Comment{If $\vec x$ does not solve LP, then $\vec x$ does not solve \Cref{eq:FM_objective}}
    \State $S \gets (S; \vec p_t)$, adding to list of computed vertices
    \State $J = J \cup \{ t \} , \lambda_t \gets 0$, add $\vec p_t$ to active vertices, and assign zero coefficient 
    \State Solve for $\alpha : \vec 1_{|J|}^{\top}\alpha = 1, S_{[n],J}^{\top}S_{[n],J}\alpha = \mu\vec 1_{|J|}$ to find aff comb of $S_J$ of min semi-norm, $\vec y = S_{[n],J}\alpha$ 
    \While{any $\alpha_i \leq 0$} \Comment{If $\vec y$ outside of $\conv (S_J)$, need to reduce $J$}
        \State $I \gets \{ i: \alpha_i \leq 0\}$
	    \State $j = \argmin_{i \in I} \frac{\lambda_i}{\lambda_i - \alpha_i}$
        \State $\gamma \gets \frac{\lambda_j }{\lambda_j - \alpha_j}$
        \State $\lambda_J \gets (1 - \gamma) \lambda_J + \gamma \alpha$, point where path from $\vec x$ to $\vec y$ meets $\conv(S_J)$
	    \State $J \gets J \setminus j$, as $\lambda_j = 0$ so $j$ no longer active
    	\State Solve for $\alpha : \vec 1_{|J|}^{\top}\alpha = 1,  S_{[n],J}^{\top}S_{[n],J}\alpha = \mu\vec 1_{|J|}$ to find new $\vec y$ coefficients
    \EndWhile
    \State $\lambda_J \gets \alpha$, coefficients for $S_J$ in $\vec x$
    \State $\vec x \gets S_J \lambda_J$, optimal convex combination of current active vertices
    \State $\vec c \gets (\vec x_{[n]}, \vec 0_k)$, gradient of semi-norm at $\vec x$
    \State $\vec p_{t} \gets \argmin_{\vec v \in Q} \vec c^{\top}  \vec v$, new optimal $\vec v$ with respect to gradient of partial norm at $\vec x$
    \State $t \gets t+1$
\EndWhile
\State $\vec d \gets \vec x_{[n]}$
\State $ \FMset_n \gets \{  \bm \theta \in \R^k : \, \forall \vec a \in A, \, \vec a^\T \vec x_i - d_i \leq \vec a^\T  \bm \theta \}$ \Comment{Compute Fr\'echet mean set}\\
\Return $ \FMset_n$
\end{algorithmic}
\end{algorithm}

As discussed by De Loera et. al. \cite{de2018minimum}, Wolfe's algorithm is, in general, exponential. However, the polytope $Q$ over which we optimize is highly structured, even for generic data, which reduces the computational cost of Wolfe's algorithm in practice. The exact computational cost can vary significantly between polytope norms --- in 13 dimensions, $\ell_1$ computations took on average 24 minutes, while $\ell_{\infty}$ tropical computations took under 1 minute.
	
With \Cref{alg:exact_computation} computing the exact Fr\'echet mean set, we use the Polyhedra package  \citep{legat2023polyhedral} to compute its dimension.

\subsection{Numerical Experiments}\label{subsec:numerical_experiments}

To study the empirical unique Fr\'echet mean sample threshold for positive density distributions, we generate Gaussian $n \times k$ data in Julia then convert the resulting floats to rationals. We generate 1000 datasets for each polytope norm ($\ell_1, \ell_{\infty},$ trop) and for $2\leq n, k \leq 13$, then compute the dimension of its Fr\'echet mean set. The code, results, and run times for these experiments can be found at \url{https://github.com/Roroast/PFM}.

Our experimental results are similar for the tropical and $\ell_{\infty}$ norms, and comparatively different for the $\ell_1$ norm. Notably, this mirrors the similarities in their polytope geometry; tropical and $\ell_{\infty}$ unit balls are zonotopes whose facet counts are polynomial in $k$, while the $\ell_1$ unit ball facet count is exponential in $k$.

\begin{figure}[pht]
    \centering
    \includegraphics[height=0.42\textheight]{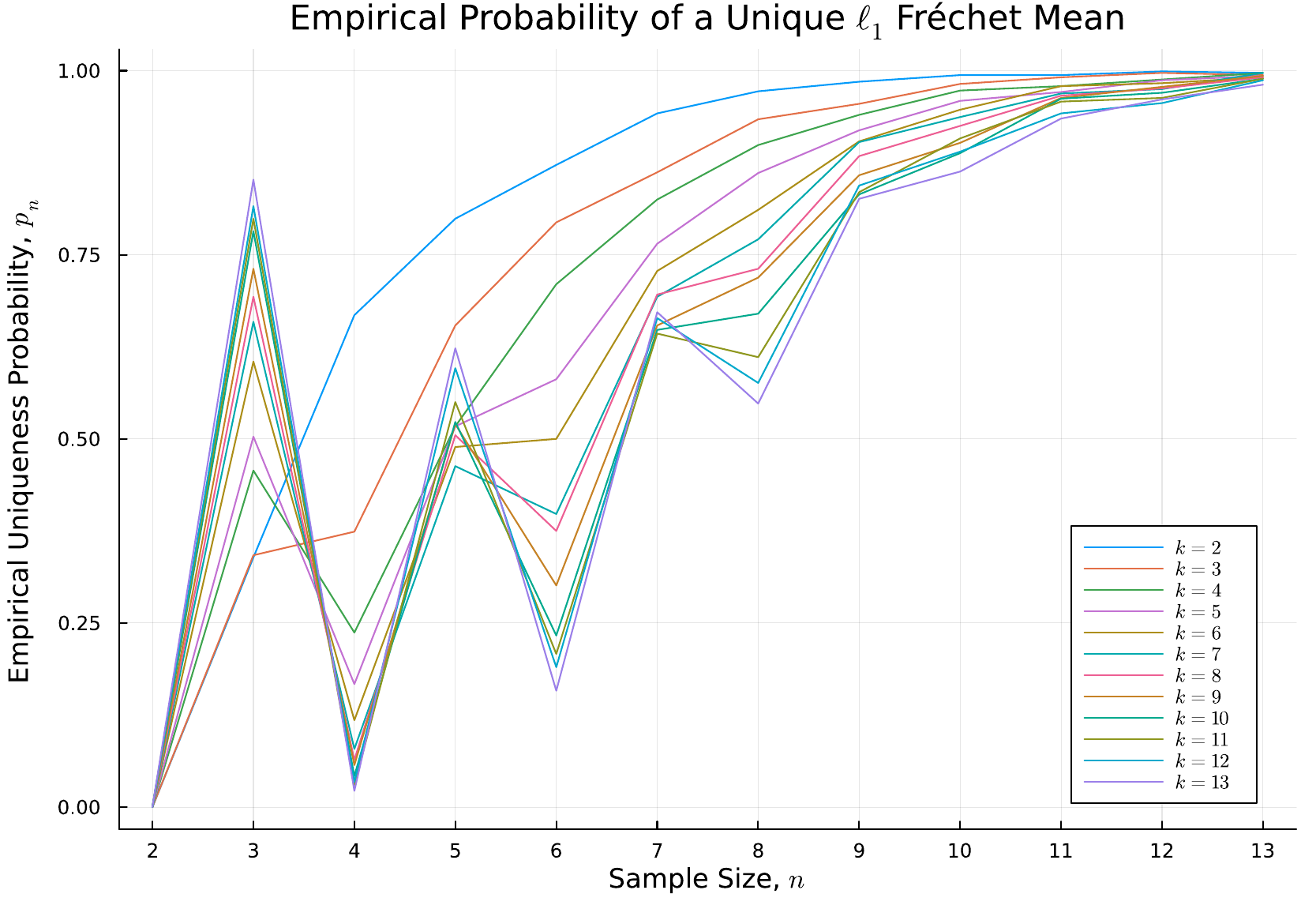}
    \caption{The empirical probability of a standard Gaussian sample having a unique $\ell_1$ Fr\'echet mean, as observed over 1000 independent random samples. The dimension $k$ and sample size $n$ of our data varies for $2 \leq n,k \leq 13$.}
    \label{fig:L_1_uniqueness_probability}
\end{figure}

\begin{figure}[pht]
    \centering
    \includegraphics[height=0.42\textheight]{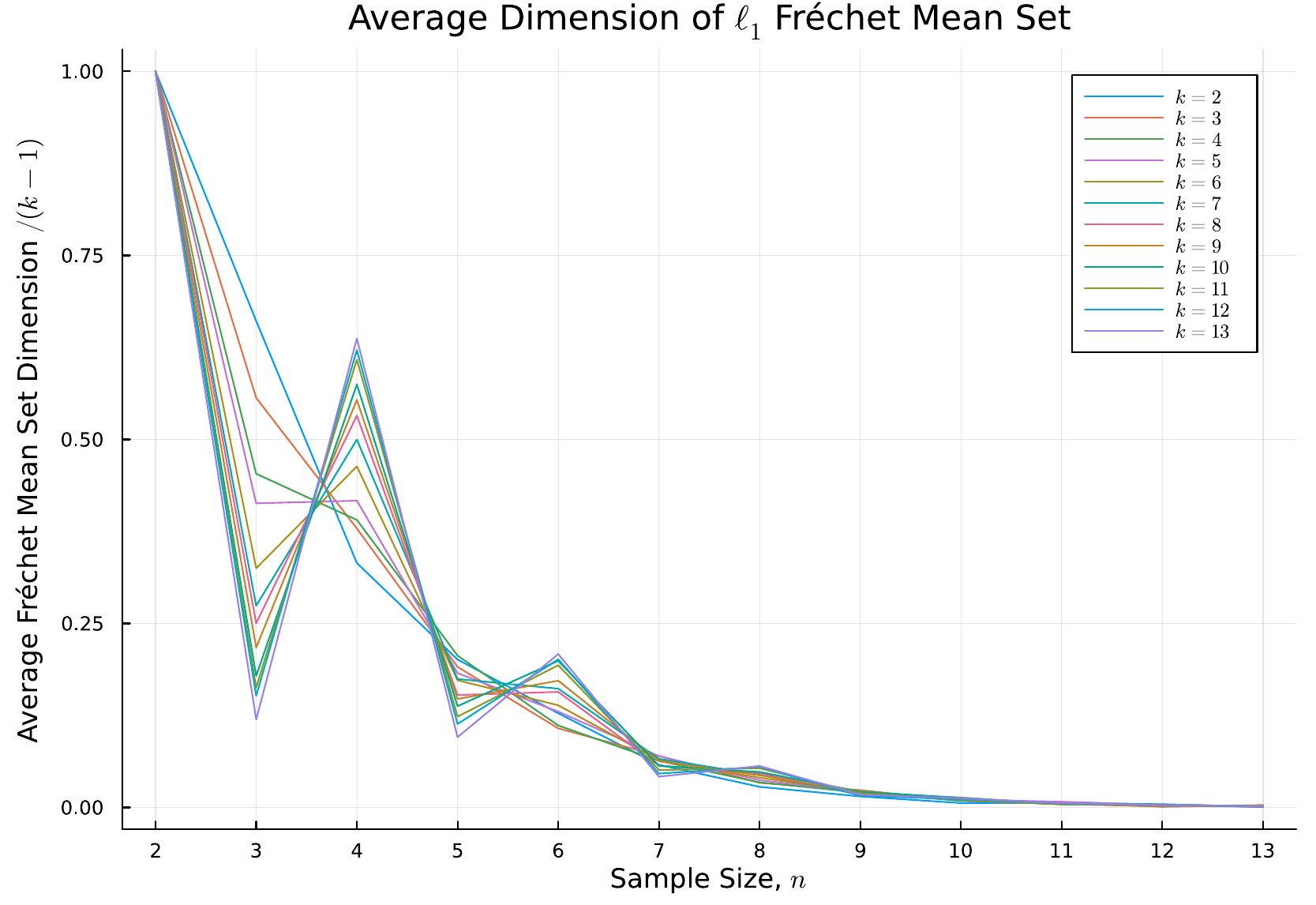}
    \caption{The empirical average dimension (as a fraction of $k-1$) of the $\ell_1$ Fr\'echet mean of a standard Gaussian sample, as observed over 1000 independent random samples. The dimension $k$ and sample size $n$ of our data varies for $2 \leq n,k \leq 13$.}
    \label{fig:L_1_expected_dimension}
\end{figure}

\subsubsection{Empirical $\ell_1$ Fr\'echet Means}

Fr\'echet means with respect to the $\ell_1$ norm were the most computationally expensive, likely due to the facet count increasing exponentially with dimension; computing a Fr\'echet mean set for a sample of size 13 in 13 dimensions took on average $\approx$ 24 minutes. This was the computational bottleneck of our experiments.

\cref{fig:L_1_uniqueness_probability} shows the empirical probability of a Gaussian dataset having a unique $\ell_1$ Fr\'echet mean, which shows evidence to support several of our key theoretical results. Firstly, we see that for all experiments with $n=2$, the Fr\'echet mean is non-unique (\cref{rmk:non_unique_for_2n}). Secondly, once $n$ increases to 3 we start to see a positive empirical probability of a unique Fr\'echet mean, for any $k$ value (\cref{prop:cross-polytope}). Finally, the probability of a unique Fr\'echet mean looks to be tending to 1 as $n$ increases (\cref{thm:uniqueness_prob_convergence}). 

As well as supporting our theoretical results, \cref{fig:L_1_uniqueness_probability} also shows an oscillation in the uniqueness probability $p_n$ for small $n$ --- lower probability for even $n$, higher for odd $n$. Notably, this oscillation becomes more pronounced for higher $k$. As such, the empirical probability of having a unique $\ell_1$ Fr\'echet mean is not only dependent on whether $n$ is past the sample threshold $N(B_{\ell_1})$, but also the parity of $n$. We may be able to develop a better understanding of the behaviour of the uniqueness probability by further studying the combinatorics of the conditions in \cref{thm:wwp_dim_condition}.

With the dimension data computed in our numerical experiments, we can also estimate the expected dimension of the Fr\'echet mean set. \cref{fig:L_1_expected_dimension} shows the empirical expected dimension of the $\ell_1$ Fr\'echet mean set, as a fraction of its maximum possible dimension, $k-1$. In agreement with \cref{rmk:non_unique_for_2n}, we see that samples of size $n=2$ almost surely have Fr\'echet mean sets of size $k-1$. We also see that the expected value of $\dim  \FMset_n/(k-1)$ seems to tend to 0 as $n$ tends to $\infty$, which is a direct consequence of the uniqueness probability tending to 1 (\cref{thm:uniqueness_prob_convergence}). As in the case of the uniqueness probability, we see oscillation on expected dimension as the parity of $n$ changes. Roughly, the expected dimension of Fr\'echet means seems to decrease with $k$ for odd sample sizes, but increase with $k$ for even sample sizes.

\subsubsection{Empirical $\ell_{\infty}$ and Tropical Fr\'echet Means}

Computation of Fr\'echet mean sets with respect to the $\ell_{\infty}$ or tropical norms is far more computationally tractable than $\ell_1$ Fr\'echet means; on average, the Fr\'echet mean set of a 13 dimensional sample of size 13 can be computed in under 1 minute. We see very similar experimental behaviour for both $\ell_{\infty}$ and tropical Fr\'echet means, so we present these results together in this section.

\begin{table}[ht]
    \centering
\begin{tabular}{cllllllllllll}
\hline
$k$ & 2 & 3 & 4 & 5 & 6 & 7 & 8 & 9 & 10 & 11 & 12 & 13 \\
\hline
\hline 
$n=$ 2 & 0.0 & 0.0 & 0.0 & 0.0 & 0.0 & 0.0 & 0.0 & 0.0 & 0.0 & 0.0 & 0.0 & 0.0 \\\cmidrule(lr){2-3}
$n=$ 3 & 0.33 & 0.035 & 0.0 & 0.0 & 0.0 & 0.0 & 0.0 & 0.0 & 0.0 & 0.0 & 0.0 & 0.0 \\\cmidrule(lr){4-4}
$n=$ 4 & 0.63 & 0.084 & 0.003 & 0.0 & 0.0 & 0.0 & 0.0 & 0.0 & 0.0 & 0.0 & 0.0 & 0.0 \\\cmidrule(lr){5-5}
$n=$ 5 & 0.8 & 0.27 & 0.044 & 0.0 & 0.0 & 0.0 & 0.0 & 0.0 & 0.0 & 0.0 & 0.0 & 0.0 \\\cmidrule(lr){6-6}
$n=$ 6 & 0.89 & 0.45 & 0.079 & 0.005 & 0.0 & 0.0 & 0.0 & 0.0 & 0.0 & 0.0 & 0.0 & 0.0 \\\cmidrule(lr){7-7}
$n=$ 7 & 0.93 & 0.57 & 0.17 & 0.028 & 0.004 & 0.0 & 0.0 & 0.0 & 0.0 & 0.0 & 0.0 & 0.0 \\\cmidrule(lr){8-8}
$n=$ 8 & 0.97 & 0.72 & 0.31 & 0.058 & 0.009 & 0.001 & 0.0 & 0.0 & 0.0 & 0.0 & 0.0 & 0.0 \\\cmidrule(lr){9-9}
$n=$ 9 & 0.99 & 0.78 & 0.43 & 0.12 & 0.018 & 0.001 & 0.0 & 0.0 & 0.0 & 0.0 & 0.0 & 0.0 \\\cmidrule(lr){10-10}
$n=$ 10 & 0.99 & 0.87 & 0.56 & 0.22 & 0.049 & 0.007 & 0.001 & 0.0 & 0.0 & 0.0 & 0.0 & 0.0 \\\cmidrule(lr){11-11}
$n=$ 11 & 1.0 & 0.92 & 0.62 & 0.32 & 0.09 & 0.017 & 0.002 & 0.0 & 0.0 & 0.0 & 0.0 & 0.0 \\\cmidrule(lr){12-12}
$n=$ 12 & 1.0 & 0.92 & 0.71 & 0.42 & 0.17 & 0.04 & 0.009 & 0.0 & 0.0 & 0.0 & 0.0 & 0.0 \\\cmidrule(lr){13-13}
$n=$ 13 & 1.0 & 0.96 & 0.79 & 0.5 & 0.24 & 0.071 & 0.013 & 0.002 & 0.0 & 0.0 & 0.0 & 0.0 \\
\bottomrule 
\end{tabular}
    \caption{The empirical probability of a standard Gaussian sample having a unique $\ell_{\infty}$ Fr\'echet mean, as observed over 1000 independent random samples. The dimension $k$ and sample size $n$ of our data varies for $2 \leq n,k \leq 13$. The horizontal bars show the unique Fr\'echet mean sample threshold for each dimension.}
    \label{tab:L_inf_uniqueness_prob}
\end{table}

\cref{tab:L_inf_uniqueness_prob,tab:trop_uniqueness_prob} show the empirical uniqueness probability for $\ell_{\infty}$ and tropical Fr\'echet means. Unlike in the $\ell_1$ case, these empirical probabilities do clearly support our theory. In both cases, we see the empirical probabilities become non-zero at the uniqueness threshold for small $k$ ($k\leq 4$ for the tropical norm, $k \leq 3$ for the $\ell_{\infty}$ norm) but for larger $k$ we do not observe any unique Fr\'echet means even when our theory suggests it should be possible. This suggests that, in contrast to the $\ell_1$ Fr\'echet means, this probability can remain very small beyond the uniqueness threshold. We also note that the uniqueness probability for small $n$ appears to grow faster for $\ell_{\infty}$ than for the tropical norm, even though it becomes non-zero later (see, e.g. $k=5,6$).

\begin{table}[ht]
    \centering
\begin{tabular}{cllllllllllll}
\hline
$k$ & 2 & 3 & 4 & 5 & 6 & 7 & 8 & 9 & 10 & 11 & 12 & 13 \\
\hline
\hline 
$n=$ 2 & 0.0 & 0.0 & 0.0 & 0.0 & 0.0 & 0.0 & 0.0 & 0.0 & 0.0 & 0.0 & 0.0 & 0.0 \\\cmidrule(lr){2-4}
$n=$ 3 & 0.51 & 0.055 & 0.003 & 0.0 & 0.0 & 0.0 & 0.0 & 0.0 & 0.0 & 0.0 & 0.0 & 0.0 \\\cmidrule(lr){5-6}
$n=$ 4 & 0.75 & 0.27 & 0.034 & 0.002 & 0.0 & 0.0 & 0.0 & 0.0 & 0.0 & 0.0 & 0.0 & 0.0 \\\cmidrule(lr){7-8}
$n=$ 5 & 0.88 & 0.47 & 0.1 & 0.013 & 0.001 & 0.0 & 0.0 & 0.0 & 0.0 & 0.0 & 0.0 & 0.0 \\\cmidrule(lr){9-10}
$n=$ 6 & 0.94 & 0.58 & 0.22 & 0.049 & 0.006 & 0.001 & 0.0 & 0.0 & 0.0 & 0.0 & 0.0 & 0.0 \\\cmidrule(lr){11-12}
$n=$ 7 & 0.97 & 0.69 & 0.29 & 0.1 & 0.016 & 0.001 & 0.0 & 0.0 & 0.0 & 0.0 & 0.0 & 0.0 \\\cmidrule(lr){13-13}
$n=$ 8 & 0.99 & 0.77 & 0.42 & 0.16 & 0.031 & 0.005 & 0.004 & 0.0 & 0.0 & 0.0 & 0.0 & 0.0 \\
$n=$ 9 & 0.99 & 0.82 & 0.46 & 0.18 & 0.05 & 0.008 & 0.002 & 0.0 & 0.0 & 0.0 & 0.0 & 0.0 \\
$n=$ 10 & 1.0 & 0.86 & 0.52 & 0.24 & 0.063 & 0.013 & 0.002 & 0.002 & 0.0 & 0.0 & 0.0 & 0.0 \\
$n=$ 11 & 1.0 & 0.91 & 0.59 & 0.26 & 0.084 & 0.024 & 0.006 & 0.0 & 0.0 & 0.0 & 0.0 & 0.0 \\
$n=$ 12 & 1.0 & 0.93 & 0.61 & 0.3 & 0.11 & 0.03 & 0.011 & 0.003 & 0.0 & 0.0 & 0.0 & 0.0 \\
$n=$ 13 & 1.0 & 0.94 & 0.65 & 0.35 & 0.12 & 0.039 & 0.01 & 0.001 & 0.0 & 0.0 & 0.0 & 0.0 \\
\bottomrule  
\end{tabular}
    \caption{The empirical probability of a standard Gaussian sample having a unique tropical Fr\'echet mean, as observed over 1000 independent random samples. The dimension $k$ and sample size $n$ of our data varies for $2 \leq n,k \leq 13$. The horizontal bars show the unique Fr\'echet mean sample threshold for each dimension.}
    \label{tab:trop_uniqueness_prob}
\end{table}

Unlike $\ell_1$ Fr\'echet means, we see strong monotone behaviour from the uniqueness probability and expected dimension (\cref{tab:L_inf_ex_dim,tab:trop_ex_dim}) of $\ell_{\infty}$ and tropical Fr\'echet means. 
The expected value of $\dim  \FMset_n/(k-1)$ in both cases appears to be tending to 0 monotonically as $n$ increases, which would follow from \cref{thm:uniqueness_prob_convergence}. Furthermore, $\dim  \FMset_n/(k-1)$ appears to be increasing in $k$, which may be tied to the fact that for larger $k$, the empirical uniqueness probabilities seen in \cref{tab:trop_uniqueness_prob,tab:L_inf_uniqueness_prob} are sometimes zero when we would expect them to be positive.

\begin{table}[ht]
    \centering
\begin{tabular}{cllllllllllll}
\hline
$k$ & 2 & 3 & 4 & 5 & 6 & 7 & 8 & 9 & 10 & 11 & 12 & 13 \\
\hline
\hline 
$n=$ 2 & 1.0 & 1.0 & 1.0 & 1.0 & 1.0 & 1.0 & 1.0 & 1.0 & 1.0 & 1.0 & 1.0 & 1.0 \\
$n=$ 3 & 0.67 & 0.74 & 0.77 & 0.8 & 0.82 & 0.84 & 0.85 & 0.86 & 0.88 & 0.88 & 0.89 & 0.9 \\
$n=$ 4 & 0.37 & 0.55 & 0.66 & 0.73 & 0.77 & 0.8 & 0.83 & 0.84 & 0.86 & 0.87 & 0.88 & 0.89 \\
$n=$ 5 & 0.2 & 0.4 & 0.52 & 0.59 & 0.64 & 0.68 & 0.72 & 0.74 & 0.76 & 0.78 & 0.79 & 0.81 \\
$n=$ 6 & 0.11 & 0.29 & 0.42 & 0.51 & 0.59 & 0.63 & 0.68 & 0.71 & 0.73 & 0.76 & 0.78 & 0.79 \\
$n=$ 7 & 0.069 & 0.22 & 0.35 & 0.43 & 0.49 & 0.55 & 0.59 & 0.64 & 0.66 & 0.69 & 0.71 & 0.73 \\
$n=$ 8 & 0.029 & 0.14 & 0.27 & 0.38 & 0.44 & 0.51 & 0.56 & 0.59 & 0.63 & 0.65 & 0.68 & 0.7 \\
$n=$ 9 & 0.014 & 0.11 & 0.21 & 0.32 & 0.38 & 0.44 & 0.5 & 0.54 & 0.57 & 0.6 & 0.63 & 0.65 \\
$n=$ 10 & 0.008 & 0.064 & 0.16 & 0.26 & 0.34 & 0.4 & 0.45 & 0.49 & 0.54 & 0.57 & 0.6 & 0.62 \\
$n=$ 11 & 0.005 & 0.04 & 0.13 & 0.21 & 0.29 & 0.36 & 0.41 & 0.45 & 0.49 & 0.53 & 0.56 & 0.59 \\
$n=$ 12 & 0.005 & 0.042 & 0.098 & 0.17 & 0.24 & 0.31 & 0.37 & 0.42 & 0.46 & 0.5 & 0.52 & 0.55 \\
$n=$ 13 & 0.003 & 0.022 & 0.071 & 0.14 & 0.21 & 0.28 & 0.34 & 0.38 & 0.42 & 0.45 & 0.49 & 0.52 \\
\bottomrule 
\end{tabular}
    \caption{The empirical average dimension (as a fraction of $k-1$) of the $\ell_{\infty}$ Fr\'echet mean of a standard Gaussian sample, as observed over 1000 independent random samples. The dimension $k$ and sample size $n$ of our data varies for $2 \leq n,k \leq 13$.}
    \label{tab:L_inf_ex_dim}
\end{table}

\begin{table}[t]
    \centering
\begin{tabular}{cllllllllllll}
\hline
$k$ & 2 & 3 & 4 & 5 & 6 & 7 & 8 & 9 & 10 & 11 & 12 & 13 \\
\hline
\hline 
$n=$ 2 & 1.0 & 1.0 & 1.0 & 1.0 & 1.0 & 1.0 & 1.0 & 1.0 & 1.0 & 1.0 & 1.0 & 1.0 \\
$n=$ 3 & 0.49 & 0.63 & 0.7 & 0.74 & 0.77 & 0.79 & 0.82 & 0.83 & 0.84 & 0.86 & 0.87 & 0.87 \\
$n=$ 4 & 0.25 & 0.41 & 0.52 & 0.59 & 0.63 & 0.66 & 0.7 & 0.73 & 0.75 & 0.76 & 0.79 & 0.8 \\
$n=$ 5 & 0.12 & 0.28 & 0.39 & 0.47 & 0.51 & 0.55 & 0.59 & 0.62 & 0.65 & 0.68 & 0.7 & 0.72 \\
$n=$ 6 & 0.065 & 0.21 & 0.31 & 0.37 & 0.41 & 0.46 & 0.5 & 0.53 & 0.57 & 0.6 & 0.62 & 0.64 \\
$n=$ 7 & 0.028 & 0.16 & 0.26 & 0.3 & 0.35 & 0.38 & 0.42 & 0.45 & 0.49 & 0.52 & 0.55 & 0.58 \\
$n=$ 8 & 0.007 & 0.12 & 0.2 & 0.26 & 0.29 & 0.33 & 0.36 & 0.39 & 0.43 & 0.46 & 0.49 & 0.51 \\
$n=$ 9 & 0.007 & 0.092 & 0.18 & 0.24 & 0.26 & 0.29 & 0.31 & 0.34 & 0.38 & 0.4 & 0.43 & 0.45 \\
$n=$ 10 & 0.005 & 0.068 & 0.16 & 0.21 & 0.24 & 0.26 & 0.27 & 0.3 & 0.33 & 0.36 & 0.38 & 0.41 \\
$n=$ 11 & 0.003 & 0.046 & 0.14 & 0.19 & 0.22 & 0.23 & 0.24 & 0.27 & 0.28 & 0.31 & 0.34 & 0.37 \\
$n=$ 12 & 0.001 & 0.037 & 0.13 & 0.18 & 0.2 & 0.21 & 0.22 & 0.24 & 0.26 & 0.28 & 0.31 & 0.33 \\
$n=$ 13 & 0.003 & 0.03 & 0.12 & 0.17 & 0.19 & 0.19 & 0.21 & 0.22 & 0.23 & 0.25 & 0.28 & 0.29 \\
\bottomrule 
\end{tabular}
    \caption{The empirical average dimension (as a fraction of $k-1$) of the tropical Fr\'echet mean of a standard Gaussian sample, as observed over 1000 independent random samples. The dimension $k$ and sample size $n$ of our data varies for $2 \leq n,k \leq 13$.}
    \label{tab:trop_ex_dim}
\end{table}

\FloatBarrier

    \section{Discussion}
	
	In this paper we have shown that the probability that the Fr\'echet mean is unique converges to one for polytope normed spaces, and we have identified geometric conditions that govern the number of sample points needed before Fr\'echet means are unique with positive probability. In doing so, we defined the \emph{unique Fr\'echet mean sample threshold} of a polytope norm --- Fr\'echet means are almost surely not unique below this sample threshold, but above this threshold the probability of uniqueness is positive and converges to one. We have shown this threshold to be at most $k+1$, where $k$ is the dimension of our space, and computed it exactly for the $\ell_\infty$, $\ell_1$ and tropical norms. We contrast these results by considering the uniqueness probability for Fr\'echet means induced by the nuclear norm which is neither a polytope nor strictly convex; we show that nuclear Fr\'echet means are almost surely non-unique for samples of size $2$, but are almost surely unique for samples of size 3 or more. Finally, we presented an algorithm for exact computation of Fr\'echet means induced by polytope norms, which can be used to study the probability of uniqueness through simulation.

    This work raises interesting geometric questions on how sample thresholds depend on the geometry of a polytope, and how we can compute a sample threshold directly. However, this work also lays the foundation for more sophisticated statistical questions, which would be particularly important for addressing inferential questions and broadening the reach of applications of polytope Fr\'echet means.

    \subsection{Geometric Theory}

    \cref{thm:wwp_dim_condition} is able to translate a probabilistic question into a purely geometric condition, which provided a very effective tool for calculating the unique Fr\'echet mean sample thresholds of specific norms in \cref{sec:calculating_examples}. However, further study of these geometric conditions is necessary to understand unique Fr\'echet mean sample thresholds for more general polytope norms. We expect this further study to focus on one of two main questions --- how can the sample threshold be computed directly for an arbitrary norm, and can we establish tighter bounds for the sample threshold, both in general and for certain classes of norms?

    \textbf{Polytope Invariants.} To understand the relationship between a polytope and its unique Fr\'echet mean sample threshold, it is natural to consider which polytope features fully determine its sample threshold. The conditions in \cref{thm:wwp_dim_condition} are invariant under linear transformations, so if two polytope norms are equivalent under some linear map, then they must have the same sample threshold. However, it is not clear whether the conditions in \cref{thm:wwp_dim_condition} are preserved for combinatorially equivalent polytopes (assuming symmetry about the origin is preserved). Understanding whether the unique Fr\'echet mean sample threshold is a property of the combinatorial structure of a polytope or its geometric realisation is particularly important in identifying which representation of the polytope would be most useful when trying to compute its sample threshold.

    \textbf{Computing Sample Thresholds.} Currently, the geometric conditions in \cref{thm:wwp_dim_condition} that dictate the sample threshold for unique Fr\'echet means do not provide a way of computing the threshold for an arbitrary polytope. One could test every $n$-tuple of faces in $\mathfrak F (B^{\Delta})$ against the conditions in \cref{thm:wwp_dim_condition} for increasing $n$ until the necessary conditions are satisfied --- while this would be guaranteed to terminate, it would be very computationally inefficient. It would be of interest to find a more direct method of computation, perhaps utilising the Gale transform of $B^{\Delta}$ \cite[Chapter 6]{ziegler2012lectures}, which records sets of vertices that contain $\vec 0$ in the relative interior of their convex hull. While this could be particularly useful for translating condition \cref{eq:wpp_cond_possible} into the language of matroids, it is not clear how to apply this to a set of faces rather than a set of vertices. We hope that by relating the unique Fr\'echet mean sample threshold of $B$ to known quantities in matroid theory, we may be able to utilise existing computational methods in order to compute unique Fr\'echet mean sample thresholds.

    \textbf{Strict Upper Bounds.} We have shown an upper bound of $k+1$ for the unique Fr\'echet mean sample threshold of a general polytope norm, which is a strict upper bound in $k=2$ dimensions, but not necessarily in higher dimensions --- our largest known example is the $\ell_{\infty}$ threshold which is exactly $k$. It is therefore not clear whether this upper bound is strict. We hope to close this gap in future work, in order to understand whether the $\ell_{\infty}$ norm is extremal in this sense and to identify which other polytope norms have maximal sample threshold. We could then test experimentally whether polytope norms with maximal sample thresholds tend to have smaller uniqueness probabilities beyond their sample threshold, as observed in \cref{subsec:numerical_experiments}. Understanding which polytope norms have large sample threshold is particularly useful in experimental design, in order to avoid implementing statistical methods which are unlikely to scale well.
	
	\textbf{Higher Dimensional Sample Thresholds.} The conditions given in \cref{thm:wwp_dim_condition} state which face types have a positive probability of occurring, and independently, which face types give rise to a unique Fr\'echet mean. We can extend this to get conditions for face types to produce a $d$-dimensional Fr\'echet mean set with positive probability, and then we can adapt \cref{cor:inductive_result} to formalise a $d$-dimensional Fr\'echet mean sample threshold. This raises the question -- how do the sample thresholds compare for different values of $d$? Is there a monotone relationship between $d$-dimensional sample thresholds? Understanding the relationship between successive $d$-dimensional sample thresholds may be particularly useful for identifying more useful bounds on the uniqueness sample threshold for certain well behaved polytopes. As an initial investigation, their behaviour can be studied computationally using \cref{alg:exact_computation}.

    \subsection{Statistical Applications}

    We have shown that with enough samples, the Fr\'echet mean of a sample will be a unique point. This allows us to develop statistical methods which rely on unique Fr\'echet means, and gives an indication of the number of samples needed to employ such methods. With this theoretical foundation, we can ask more practical questions; how many samples are needed for uniqueness on average, rather than in the best case? What is the limiting distribution of polytope Fr\'echet means? How does this theory generalise to data with different support?

    \textbf{Sample Threshold in Practice.} As shown by \cref{tab:L_inf_uniqueness_prob,tab:trop_uniqueness_prob}, even when theory suggests the probability of uniqueness is positive, this probability can be close to zero for a long time. It would be more helpful to understand the number of samples required for uniqueness with probability $>p$, rather than just in the best case. 
    However, quantifying this is challenging --- the exact probability of uniqueness is highly dependent on the underlying distribution of our data, and we have no known relationship between successive Fr\'echet mean sets as samples grow. It would therefore be more practical to study the probability of uniqueness through a numerical study. \cref{alg:exact_computation} is designed for exactly this --- our results in \cref{subsec:numerical_experiments} already demonstrate a methodology for this study using Gaussian data.

    \textbf{Limiting Distributions.} We have shown that Fr\'echet means are unique asymptotically almost surely, which is particularly powerful in enabling us to identify a limiting distribution for sample Fr\'echet means without a random set formulation. Understanding the limiting distribution of sample Fr\'echet means is particularly valuable, as estimator distributions are a crucial ingredient in calculating confidence regions, performing hypothesis tests, and several other fundamentals statistical methods. Fr\'echet means are just M-estimators in a Euclidean space, for which asymptotic distributions are well-studied \citep{Vaart_1998,van1996m,welsh1989m}, but the non-smooth nature of the Fr\'echet function is a significant obstacle in applying most of the existing literature.

    \textbf{Phylogenetic Applications.} Phylogenetic data analysis typically makes use of the tropical norm by representing data in a vectorised form. This puts our data in a polytope normed space, but in doing so, the state space of our data becomes a more complex stratified subspace (see, e.g. \cite{speyer2004grassmannian}), rather than the full polytope normed space. As such, our assumption that the data has positive density everywhere no longer holds. To tailor the theory in this paper to specific applications such as phylogenetic data analysis, we need to either generalise our theoretical results, or mirror this approach in application-specific ways. This may be most tangible when studying equidistant tree data, which is supported on a Bergman fan \citep{bernstein2020infinity}; this space is tropically convex, which has been useful in showing tropical centrality statistics are stable under noise \citep{cox2025tropical}. Understanding how this theory carries over to tropical stratified subspaces is crucial for making a meaningful comparison between the statistical power of tropical Fr\'echet means and other centrality statistics for phylogenetic data. In other geometric settings for phylogenetic data \citep{billera2001geometry}, Fr\'echet means display ``sticky'' behaviour --- the mean of a dataset is unchanged by changes to the data. This limits the power of Fr\'echet means to distinguish between different distributions. Exploring whether the same behaviour occurs for tropical Fr\'echet means would be a significant contribution to the ongoing study of the relative statistical strengths of different phylogenetic tree space geometries.

    \textbf{Matrix Norms.} In this paper, our running examples were the $\ell_1$, $\ell_{\infty}$ and tropical norms, but polytope norms also appear frequently in matrix analysis as operator norms, which have applications far beyond semi-discrete data analysis \citep{belitskii2013matrix}. \cref{prop:C_prods_additive} gives a non-trivial upper bound for the sample thresholds of particular matrix operator norms, but it would be of interest to compute their exact value. This would likely require a non-trivial lower bound on the sample threshold which is much harder as it requires proving non-existence of suitable face types. There are also popular matrix norms such as the nuclear norm which are neither polytope norms nor are they strictly convex. A study of uniqueness of Fr\'echet means induced by the nuclear norm would offer valuable insight into which properties are a result of polytope geometry, and which arise more generally from a lack of strict convexity.

    \textbf{Norm Antipodal Symmetry.} Throughout this paper, we have studied Fr\'echet means being induced by a norm, whose unit ball is necessarily fixed under negation. However, the optimality of Fr\'echet means is a local property, and so should not depend on a global property like the symmetry of the underlying polytope. We have used the antipodal symmetry of our norm in a couple places throughout this paper such as in \cref{lemma:RelintGrad_to_RelintFMset}, where we swapped the maximum and the square operations in the Fr\'echet function --- however, we expect it is possible to reach the same results for asymmetric norms such as the tropical asymmetric distance \citep{comuaneci2024tropical}. Without the assumption of symmetry, this problem becomes a much more general problem which resembles the training of bias through a maxout layer of a neural network in order to minimise an $\ell_2$ loss.
    
    \FloatBarrier

	\section{Acknowledgements}
    The authors wish to thank G\"unter Rote for helpful discussions.\\
    
	A.McC. acknowledges the support of the Natural Sciences and Engineering Research Council of Canada (NSERC), [RGPIN-2025-03968, DGECR-2025-00237]. R.T. and A.Mo. are supported by EPSRC grant EP/Y028872/1, Mathematical Foundations of Intelligence: An “Erlangen Programme” for AI. This work has made use of the Hamilton HPC Service of Durham University.

    \section{Declaration of generative AI use}

    No AI was used in the formulation of this theory or the writing of this manuscript. The authors used GitHub Copilot to assist in the coding of numerical experiments, and used ChatGPT to assist in final draft proof reading. The authors take full responsibility for the content of the published article.
	
	\bibliographystyle{authordate3}
	
	\bibliography{Sources}

	\appendix
	
	\section{Type Event Proofs} \label{appsec:lemma_proofs}

    In this section, we include the proofs of \Cref{lemma:polyhedral_cone_W,lemma:RelintGrad_to_RelintFMset,lemma:U_in_ClV} from \Cref{sec:FM_uniqueness}.
	
    \polyhedronW*
	
	\begin{proof}
		Let $\vec g_1^{(i)}, \dots, \vec g_{r_i}^{(i)} \in \R^k$ be the vertices of $G_i$, and set $R = \sum_{i \leq n} r_i$. Let $M_G \in \R^{k \times R}$ be the matrix whose columns are given by all of the $\vec g_j^{(i)}, \, i\leq n, j \leq r_i$. We index a vector $\vec x \in  \R^{R}$ by $\vec x = (x_1^{(1)},\ldots,x_{r_1}^{(1)},\ldots,x_{1}^{(n)},\ldots,x^{(n)}_{r_n})$, and define the linear map $L_W: \R^{R} \rightarrow \R^n$ by $\left(L_W(\vec x)\right)_i = \sum_{j \leq r_i} x_j^{(i)}$. We will show that $W = L_W(\Ker M_G \cap \R^{R}_{>0})$. \\
		By \cite[Theorem 6.9]{rockafellar1970convex}, which we quoted in \cref{eq:rockafellar_thm6.9}:
		\[
		\relint G_i = \bigg\{ \textstyle \sum_{j \leq r_i} \lambda^{(i)}_{j} \vec g_j^{(i)} : \sum_{j \leq r_i} \lambda^{(i)}_j = 1,\; \lambda^{(i)}_j > 0 \bigg\}.
		\]
		Therefore $\vec w \in W$ if and only if there exist positive $\left\{\lambda^{(i)}_j\right\}_{i \leq n, j \leq r_i}$ satisfying $\sum_{j \leq r_i} \lambda_j^{(i)} = 1$ and
\begin{align}
\label{eqn:ConditionInW}
    		\vec 0 = \textstyle \sum_{i \leq n} \sum_{j \leq r_i} w_i \lambda^{(i)}_j \vec g_j^{(i)}.
\end{align}
		For any $\vec w \in W$, pick $\vec \lambda$ satisfying the above and define $\vec x$ by $x_j^{(i)} = w_i\lambda_j^{(i)} > 0$. Then we have that $M_G \vec x = \vec 0$ and $L_W(\vec x) = \vec w$, and we conclude that $W \subseteq L_W(\Ker M_G \cap \R^{R}_{>0})$. \\
		Now suppose $\vec x \in \Ker M_G \cap \R^{R}_{>0}$. Then setting $\lambda_j^{(i)} = x_j^{(i)}\left(\sum_{j \leq r_i} x_j^{(i)}\right)^{-1}$, we see that $L_W(\vec x)$ satisfies the condition \eqref{eqn:ConditionInW} and so is in $W$. We conclude that $L_W(\Ker M_G \cap \R^{R}_{>0}) \subseteq W$. \\
		As $\Ker M_G \cap \R^{R}_{+}$ is a relatively open polyhedral cone, so is $L_W(\Ker M_G \cap \R^{R}_{>0}) = W$.
	\end{proof}
	
	\Containment*
	
	\begin{proof}
		First, we translate $\vec 0 \in \relint \partial f_n( \bm \theta)$ into a condition on the directional derivatives of $f_n$.
		\begin{claim*}
		When $\vec 0 \in \relint\partial f_n( \bm \theta)$, then $\forall \,\vec v \in \R^k$, either $\partial_{\vec v} f_n( \bm \theta) = \partial_{- \vec v} f_n( \bm \theta) = 0$ or $\partial_{\vec v} f_n( \bm \theta), \partial_{-\vec v} f_n( \bm \theta) > 0$.
		\end{claim*}
		\begin{proof}
			For all $\vec v \in (\partial f_n( \bm \theta))^{\perp}$, we immediately have
			\begin{align*}
				\partial_{\vec v} f_n( \bm \theta) = \sup_{\vec g \in \partial f_n( \bm \theta)} \vec g^{\T}\vec v = 0, \qquad 
				\text{ and } \qquad \partial_{-\vec v} f_n( \bm \theta) = \sup_{\vec g \in \partial f_n( \bm \theta)} -\vec g^{\T}\vec v = 0.
			\end{align*}
			Now consider $\vec v \notin (\partial f_n( \bm \theta))^{\perp}$, so there is some $\vec g_1 \in \partial f_n( \bm \theta)$ such that $\vec g_1^{\T}\vec v \neq 0$. As $\vec 0 \in \relint\partial f_n( \bm \theta)$, there is some $\vec g_2 \in \partial f_n( \bm \theta)$ such that $\vec 0 = \lambda \vec g_1 + (1-\lambda) \vec g_2$, $0 < \lambda < 1$. Therefore $0=\lambda \vec g_1^{\T}\vec v + (1-\lambda) \vec g_2^{\T} \vec v$, and $\vec g_1^{\T}\vec v$ and $\vec g_2^{\T}\vec v$ must be non-zero with opposite signs. Hence:
			\begin{align*}
				\partial_{\vec v} f_n( \bm \theta) = \sup_{\vec g \in \partial f_n( \bm \theta)} \vec g^{\T}\vec v > 0, \qquad
				\text{ and } \qquad \partial_{-\vec v} f_n( \bm \theta) = \sup_{\vec g \in \partial f_n( \bm \theta)} -\vec g^{\T}\vec v > 0.\tag*{\qedhere}
			\end{align*}
		\end{proof}    
        We prove this lemma by contradiction. Suppose $\vec 0 \in \relint\partial f_n( \bm \theta)$ and $ \bm \theta \notin \relint  \FMset_n$, so $ \bm \theta$ is on the boundary of $ \FMset_n$; there is some $\epsilon>0$ and direction $\vec v$ such that, for $\lambda \in (0, \epsilon)$, $ \bm \theta + \lambda \vec v \in  \FMset_n$ and $ \bm \theta - \lambda\vec v \notin  \FMset_n$. Then by the convexity of $f_n$ we have that $\partial_{\vec v} f_n( \bm \theta) = 0$, and so by the claim above we also have $\partial_{-\vec v} f_n( \bm \theta) = 0$.
		
		Note that $f_n$ is a maximum over quadratic functions whose coefficients are given by vectors $\vec a \in A$:
		\[
		f_n(\vec x) = \textstyle \frac{1}{n} \sum_{i \leq n} \|\vec X_i - \vec x \|_B^2 =  \frac{1}{n} \sum_{i \leq n} \max_{\vec a \in A} \left[ \vec a^{\T}(\vec X_i- \vec x) \right]^2 =  \max_{\vec a_1,\dots, \vec a_n \in A}  \frac{1}{n} \sum_{i \leq n} \left[\vec a_i^{\T}(\vec X_i- \vec x)\right]^2.
		\]
		We can fix $\vec a_1,\dots, \vec a_n, \vec b_1, \dots, \vec b_n$ to be the coefficients of quadratics $f_1, f_2$ which agree with $f_n$ on the intervals $[ \bm \theta,  \bm \theta+\epsilon \vec v]$ and $[ \bm \theta - \epsilon \vec v,  \bm \theta]$ respectively. Explicitly, pick $\vec a_1,\dots, \vec a_n, \vec b_1, \dots, \vec b_n \in A$ such that for some $\epsilon > 0$:
		\begin{align*}
			\text{ for }\lambda \in [0,\epsilon]:  \qquad f_n( \bm \theta + \lambda \vec v) = f_1(\lambda) &\coloneqq \textstyle \sum_{i \leq n} \vec a_i^{\T}(\vec X_i-( \bm \theta + \lambda \vec v))^2, \\
			\text{ for }\lambda \in [-\epsilon,0]:  \qquad f_n( \bm \theta + \lambda \vec v) = f_2(\lambda) &\coloneqq \textstyle\sum_{i \leq n} \vec b_i^{\T}(\vec X_i-( \bm \theta + \lambda \vec v))^2. 
		\end{align*}
		As $\vec v$ is directed into $ \FMset_n$, $f_n( \bm \theta+\lambda \vec v)$ is constant for small $\lambda$. Hence $f_1$ must be constant for all $\lambda$. Meanwhile, $-\vec v$ is directed away from $ \FMset_n$ so $f_2$ must be strictly decreasing for $\lambda \in [-\epsilon,0]$. Moreover,
		\begin{align*}
			f_2'(0)  = -\partial_{-\vec v}f_n( \bm \theta) = 0.
		\end{align*}  
		The quadratic $f_2$ is therefore strictly decreasing on $[-\epsilon,0]$ with $f_2'(0) = 0$, so it follows that $f_2$ is strictly increasing for positive $\lambda$. As $f_1$ is constant for positive $\lambda$ and $f_1(0) = f_2(0) = f_n( \bm \theta)$, $f_2$ must be strictly greater than $f_1$ for small $\lambda$. This contradicts the maximality of $ f_1(\lambda)$ for $\lambda \in [0,\epsilon]$.
	\end{proof}
	
	\Limit*
	
	\begin{proof}
		We begin by showing that we may assume that $\vec 0 \in \relint \left( \conv \left(\textstyle \bigcup_{i \leq n} G_i \right) \right)$. Fix some $(\sample,  \bm \theta)$ in $U_{\CType,\Theta}$, and assume the contrary, that $\vec 0 \notin \relint \left( \conv \left(\textstyle \bigcup_{i \leq n} G_i \right) \right)$.
		
		Then there is some $\vec v$ such that for all $i$, $\vec g \in G_i$, we have $\vec g^{\T}\vec v\geq 0$ and there is some $i$ and $\vec g \in G_i$ such that $\vec g^{\T}\vec v> 0$. Hence we can bound the directional derivatives of each term in the Fr\'echet function at $\bm \theta$:
		\begin{align*}
			\partial_{\vec v}(d_B( \bm \theta,\vec X_i)^2) = \sup_{\vec g \in G_i} [-2d_B( \bm \theta, \vec X_i)\vec g^{\T}\vec v] \leq 0, \qquad \text{ and } \qquad
			\partial_{-\vec v}(d_B( \bm \theta,\vec X_i)^2) = \sup_{\vec g \in G_i} [2d_B( \bm \theta, \vec X_i)\vec g^{\T}\vec v] \geq 0.
		\end{align*}
		As $ \bm \theta$ is a Fr\'echet mean of $\vec X_1, \dots \vec X_n$, every directional derivative of $f_n$ at $ \bm \theta$ is non-negative. That is:
		\begin{align*}
			0 &\leq \partial_{\vec v} f_n( \bm \theta) = \textstyle n^{-1} \sum_{i \leq n} \sup_{\vec g \in G_i} [-2d_B( \bm \theta, \vec X_i)\vec g^{\T}\vec v] \leq 0.
		\end{align*}
		The $\vec v$-directional derivative of every $d_B( \bm \theta,\vec X_i)^2$ term is therefore $0$. To compute the directional derivatives explicitly, we note that there is some $\epsilon>0$ and $\vec a_i \in A$ such that, for $\lambda \in [0,\epsilon]$, we have $d_B( \bm \theta+\lambda \vec v,\vec X_i) = \vec a_i^{\T}( \bm \theta + \lambda \vec v - \vec X_i)$. Then:
		\begin{align*}
			0 &= \partial_{\vec v} (d_B^2( \bm \theta,\vec X_i)), \\
			&= \left(\frac{\partial}{\partial \lambda}\right)_{\lambda = 0} [\vec a_i^{\T}( \bm \theta + \lambda \vec v - \vec X_i)]^2, \\
			&= (\vec a_i^{\T}\vec v) \cdot d_B( \bm \theta, \vec X_i).
		\end{align*}
		We have assumed that $d( \bm \theta, \vec X_i)>0$, as face type is well defined for $\sample$; this assumption holds almost surely by \cref{lemma:data_is_not_FM}. Therefore we have that $\vec a_i^{\T}\vec v = 0$, and for every $i$:
		\begin{align*}
			d_B( \bm \theta+\lambda \vec v,\vec X_i) &= \vec a_i^{\T}( \bm \theta + \lambda \vec v - \vec X_i), \\
			&= \vec a_i^{\T}( \bm \theta - \vec X_i), \\
			&= d_B( \bm \theta,\vec X_i),
		\end{align*}
		which shows that $ \bm \theta + \lambda \vec v \in  \FMset_n$ for all $\lambda \in [0,\epsilon]$. 
		By the definition of $U_{\CType,\Theta}$ we know $ \bm \theta \in \relint  \FMset_n$, and so $ \bm \theta - \lambda \vec v$ must also be in $ \FMset_n$ for sufficiently small $\lambda$. Hence
		\begin{align*}
			0 &= \partial_{-\vec v} f_n( \bm \theta) =\textstyle n^{-1} \sum_{i \leq n} \sup_{\vec g \in G_i} [2d_B( \bm \theta, \vec X_i)\vec g^{\T}\vec v].
		\end{align*}
		The direction $\vec v$ was chosen such that for all $i$ and $\vec g \in G_i$, we have $\vec g^{\T}\vec v\geq 0$, with strict inequality in at least one case. But from the equality above, we conclude that for all $i$ and $\vec g \in G_i$, $d_B( \bm \theta, \vec X_i)\vec g^{\T}\vec v \leq 0$, which contradicts the definition of $\vec v$. Hence we have $\vec 0 \in \relint \left( \conv \left(\textstyle \bigcup_{i \leq n} G_i \right) \right)$.
		
		By \cite[Theorem 6.9]{rockafellar1970convex}, $\vec 0 \in \relint \left( \conv \left(\textstyle \bigcup_{i \leq n} G_i \right) \right)$ implies that there exist $\vec g_i' \in \relint G_i$ and $ \lambda_i > 0$ such that $\textstyle \vec 0 = \sum_{i \leq n} \lambda_i \vec g_i'$.
		Also, as $(\vec X_1,\dots, \vec X_n, \bm \theta) \in U_{F_1,\dots, F_n,\Theta}$, we know that $ \bm \theta$ is a Fr\'echet mean and $\vec 0 \in \sum_{i \leq n} \| \vec X_i -  \bm \theta \|_B G_i$. Then there also exist $\vec g_i \in G_i, w_i = \|\vec X_i -  \bm \theta\|_B > 0$ such that $\vec 0 = \sum_{i \leq n} w_i \vec g_i$.
		Interpolating between these expressions for $\vec 0$ we get, for any $\delta \in [0,1]$:
		\[
		\textstyle \vec 0 = \sum_{i \leq n} (1-\delta)w_i \vec g_i + \delta \lambda_i \vec g_i' = \sum_{i \leq n} \left((1-\delta)w_i + \delta \lambda_i\right) \left( \frac{(1-\delta)w_i}{(1-\delta)w_i + \delta \lambda_i} \vec g_i + \frac{\delta \lambda_i}{(1-\delta)w_i + \delta \lambda_i} \vec g_i' \right) .
		\]
		As $\vec g_i'$ is in the relative interior of $G_i$, any point on the line segment between $\vec g_i'$ and $\vec g_i$ is also in the relative interior. In particular, for $\delta \in (0,1]$:
		\[
		\frac{(1-\delta)w_i}{(1-\delta)w_i + \delta \lambda_i} \vec g_i + \frac{\delta \lambda_i}{(1-\delta)w_i + \delta \lambda_i} \vec g_i' \eqqcolon \vec g_i'' \in \relint G_i.
		\]
		By picking $\delta >0$ sufficiently small, we can therefore write $\vec 0 = \sum_{i \leq n} w_i' \vec g_i''$, where $\vec g_i'' \in \relint G_i$, and $w_i' \coloneqq (1-\delta)w_i + \delta \lambda_i$ satisfies $|w_i' - w_i| < \epsilon, w_i'>0$. \\
		We now define $\vec Y_i =  \bm \theta + \frac{w_i'}{w_i}(\vec X_i- \bm \theta)$. We will show that $\| \vec Y_i - \vec X_i \|_B < \epsilon$ and $(\vec Y_1, \dots, \vec Y_n, \bm \theta) \in V_{\CType, \Theta}$.
		By the definitions of $\vec Y_i$, $w_i$, and $w_i'$:
		\begin{align*}
			\|\vec Y_i - \vec X_i\|_B &= |w_i'/w_i - 1| \|\vec X_i -  \bm \theta\|_B, \\
			&= |w_i' - w_i| < \epsilon.
		\end{align*}
		We also have $\|\vec Y_i -  \bm \theta\|_B = w_i'$ and $\vec Y_i -  \bm \theta = \tfrac{w_i'}{w_i}(\vec X_i- \bm \theta) \in \relint \mb{R}_{>0}F_i$ by the definition of $w_i$ and because $(\sample, \bm \theta) \in U_{\CType,\Theta}$.
		By the construction of $w_i', \vec g_i''$:
		\begin{align*}
			\vec 0 &\in \textstyle\sum_{i \leq n} w_i' \relint G_i = \relint \sum_{i \leq n} w_i' G_i = \relint \sum_{i \leq n} \|\vec Y_i -  \bm \theta\|_B G_i,
		\end{align*}
        where the first equality follows from \citep[Cor 6.6.1, 6.6.2]{rockafellar1970convex}. Hence $\vec Y_1,\dots,\vec Y_n, \bm \theta$ is in $V_{\CType,\Theta}$ with $\|\vec X_i-\vec Y_i\|_B < \epsilon$.
	\end{proof}
		
\end{document}